\documentclass[a4paper,10pt]{article}

\usepackage[utf8]{inputenc}
\usepackage[a4paper, total={6in, 10in}]{geometry}
\usepackage{import} 
\usepackage{marginnote}

\input{preamble.tex}

\title{Feedback vertex sets of planar digraphs with fixed digirth\footnote{This research was partially funded by the French National Research Agency (ANR) under grant agreement No. ANR-24-CE48-3758-01. In accordance with the objective of open access dissemination, the author applies a Creative Commons Attribution (CC-BY) license to any accepted article or manuscript (AAM) resulting from this submission.}}
\author[1]{Simon Dreyer}
\author[1]{Alexandre Pinlou}
\author[1]{Petru Valicov}
\affil[1]{LIRMM, Université de Montpellier, CNRS, Montpellier, France.}

\begin{document}

\maketitle

\begin{abstract}
  Let $\fv(G)$ denote the size of a minimum feedback vertex set of a digraph $G$. We study $\fv_g(n)$, which is the maximum $\fv(G)$ over all $n$-vertex planar digraphs $G$ of digirth $g$. We prove a planar-digraph analogue of the celebrated Lucchesi-Younger theorem showing that the minimum feedback vertex set is at most the maximum packing of a special type of directed cycles. As a corollary, we derive that $\fv_g(n)\le \frac{n-2}{g-2}$ for all $g\geq 3$. This improves all previously known upper bounds for $g \ge 4$, and for $g \ge 6$ it supersedes the best known upper bound of $\frac{2n-6}{g}$ (Esperet, Lemoine and Maffray, 2017) by a factor of 2.
  
  On the other hand, we develop a new framework to construct planar digraphs of fixed digirth and large $\fv$. Using it, for $g = 6$ and every $g \ge 8$, we construct an infinite family of planar digraphs of digirth $g$ and $\fv(G) = \frac{g+2}{g^2} n + O(1)$. For $g= 7$, our construction gives $\fv(G) = \frac{2}{11} n + O(1)$ and for $g = 4$ and $5$, $\fv(G) = \frac{n}{g-1}$. These improve the best known lower bound of $\frac{n-1}{g-1}$ (Knauer, Valicov and Wenger, 2017) for all $g \ge 4$.
  
  We thus obtain the two-sided bound $\frac{g+2}{g^2} \le \sup_{n \ge 1} \frac{\fv_g(n)}{n} \le \frac{1}{g-2}$ for all values $g  = 6$ and every $g \ge 8$. The gap between the lower and the upper bound for $\sup_{n \ge 1} \frac{\fv_g(n)}{n}$ decreases from $\frac{g-2}{g(g-1)}$ to $\frac{4}{g^2(g-2)}$.
\end{abstract}

\section{Introduction}
Given a graph $G=(V,E)$, a \emph{feedback vertex set} is a set $S \subseteq V$ such that the subgraph induced by $V \setminus S$ is acyclic. If $G$ is directed, this means that the induced subdigraph contains no directed cycle.
The set $V\setminus S$ is then called an \emph{acyclic set}. Computing the \emph{minimum size of a feedback vertex set}, denoted by $\fv(G)$, is a hard problem even for planar graphs (see~\cite{Karp72, gareyjohnson1979}) and has been extensively studied in the literature.
 
In this paper, we focus on planar graphs. By the Four Color Theorem, it is known that every planar graph on $n$ vertices contains an independent set of vertices of size $\frac{n}{4}$. A widely open conjecture proposes a strengthening of this result:

\begin{conjecture}[Albertson and Berman, 1976~\cite{AB76}]\label{conj:albertson-berman}
    Every simple undirected planar graph $G$ on $n$ vertices contains an induced forest of size at least $\frac{n}{2}$.
\end{conjecture}

Equivalently, \Cref{conj:albertson-berman} states that every undirected planar graph $G$ on $n$ vertices has a feedback vertex set of size at most $\frac{n}{2}$. The best known general bound for \Cref{conj:albertson-berman} is due to Borodin~\cite{B79}, who proved that every undirected planar graph admits an induced forest of size at least $\frac{2n}{5}$, in other words $\fv(G) \le \frac{3n}{5}$. Better upper bounds for $\fv$ are known depending on the \emph{girth} (the length of a smallest cycle) and we summarize them in the second column of \Cref{tab:undirected_fvs}. The third column of this table shows the conjectured upper bounds, which are tight if true.

\begin{table}[htbp]
    \centering
    \begingroup
    \renewcommand{\arraystretch}{1.5}
    \begin{tabular}{|c|c|c|}
        \hline
        Girth $g$ & Best proved upper bound & Conjectured upper bound (tight if true) \\
        \hline
        3 & $\frac{3n}{5}$ ~\cite{B79} & $\frac{n}{2}$~\cite{AB76} \\\hline
        4 & $\frac{4n}{9}$~\cite{L18} & $\frac{3n}{8}$~\cite{AW87,L18} \\\hline
        5 & $\frac{m}{5}$~\cite{KL17,SX17} & $\frac{3n}{10}$~\cite{KLS10} \\\hline
        $\geq 6$ & $\frac{4m}{3g}$~\cite{DMP16} and $\frac{2m-n+2}{7}$~\cite{KL17,SX17} & $\frac{m}{g}$~\cite{DMP16} \\
        \hline
        \end{tabular}
    \endgroup
    \medskip
    \caption{\label{tab:undirected_fvs}Upper bounds for $\fv$ for planar undirected graphs depending on the girth $g$, the order $n$ and the number of edges $m$.}
\end{table}
 
Note that for undirected planar graphs, Euler's formula gives $\frac{m}{g} \le \frac{n-2}{g-2}$. Therefore, the known upper bounds of the last row of \Cref{tab:undirected_fvs} can be restated as $\fv(G) \le \frac{4m}{3g} \le \frac{4(n-2)}{3(g-2)}$. Moreover, if the conjecture $\fv(G) \le \frac{m}{g}$ of~\cite{DMP16} is true, then it would imply that $\fv(G) \le \frac{n-2}{g-2}$.

We switch to directed graphs (or \emph{digraphs}) with no parallel arcs and no loops. The length of a shortest directed cycle of a digraph $D$ is called the \emph{digirth} of $D$. An \emph{oriented graph} is a digraph with digirth at least 3 (that is, has no digons). As forests are the cycle-free structures in undirected graphs, acyclic sets are their natural analogs in digraphs.

\subsection{Upper bounds}

Albertson posed the following weakening of \Cref{conj:albertson-berman} (see~\cite{Mohar2002, W06, KVW17}): does every planar oriented graph contain an acyclic set of size at least $\frac{n}{2}$? In other words, is it true that $\fv(D) \le \frac{n}{2}$ for every planar oriented graph $D$ on $n$ vertices?
The best known general result for Albertson's question is the same as for \Cref{conj:albertson-berman} and is due to Borodin~\cite{B79}: $\fv(D) \le \frac{3n}{5}$. Moreover, the question is a weakening of the following famous conjecture.

\begin{conjecture}[Neumann-Lara, 1985~\cite{NL85}]\label{conj:NL}
    Every planar oriented graph can be vertex-partitioned into two acyclic sets.
\end{conjecture}

While \Cref{conj:NL} remains open in general, it was established by Li and Mohar~\cite{LM17} for planar oriented graphs without directed triangles. Therefore, $\fv(D) \le \frac{n}{2}$ for every planar oriented graph $D$ of digirth at least 4.
Better upper bounds for $\fv$ of planar oriented graphs with a larger digirth have been obtained by Esperet, Lemoine and Maffray~\cite{ELM17}. We summarize the known upper bounds for $\fv$ in \Cref{thm:SotA_digirth} below.

\begin{theorem}[Borodin, 1979~\cite{B79}, Esperet, Lemoine and Maffray, 2017~\cite{ELM17}, Li and Mohar, 2017~\cite{LM17}]\label{thm:SotA_digirth}
    Let $D$ be an oriented planar graph of digirth $g$ on $n$ vertices. Then
    \begin{itemize}
        \item $\fv(D) \le \frac{3n}{5}$~\cite{B79},
        \item $\fv(D) \le \frac{n}{2}$~\cite{LM17} if $g\ge4$,
        \item $\fv(D) \le \frac{2n-5}{4}$~\cite{ELM17} if $g\ge5$, and
        \item $\fv(D) \le \frac{2n-6}{g}$~\cite{ELM17} if $g\ge 6$.
    \end{itemize}
\end{theorem}

Moreover, the authors of~\cite{ELM17} suggest that the factor 2 for $g\ge6$ should be reducible to a constant closer to 1.

A \emph{feedback arc set} of a digraph $D$ is a set of arcs $A$ such that $D - A$ is acyclic and we denote by $\fa(D)$ the \emph{minimum size of a feedback arc set} of $D$. A celebrated result of Lucchesi and Younger~\cite{LY78} combined with planar duality relates minimal feedback arc sets to maximum families of arc-disjoint directed cycles.

\begin{theorem}[Lucchesi and Younger, 1978~\cite{LY78}]\label{thm:LY}
    Let $D$ be a directed planar graph. Let $\mathcal{C}_{\max}$ be a maximum set of arc-disjoint directed cycles of $D$. Then $\fa(D) = |\mathcal{C}_{\max}|$.    
\end{theorem}

As a consequence of \Cref{thm:LY}, the inequality $\fv(D)\leq \frac{m}{g}$ holds for planar digraphs of digirth $g$. Observe that this is a directed analogue of the conjecture of Dross, Montassier and the second author~\cite{DMP16} for planar undirected graphs. Let $D$ be a planar digraph of digirth $g$ and $\mathcal{C}_{\max}$ be a maximum set of arc-disjoint directed cycles of $D$. Since each directed cycle of $D$ has length at least $g$, clearly $g\times \lvert \mathcal{C}_{\max} \rvert \le m$ and thus $\fa(D) \le \frac{m}{g}$. Taking one endpoint from each arc in a minimum feedback arc set yields a feedback vertex set of size at most $\fa(D)$, hence $\fv(D) \le \fa(D) \le \frac{m}{g}$. However, the inequality $\frac{m}{g} \le \frac{n-2}{g-2}$ does not hold for every planar digraph $D$ of digirth $g$, as $D$ might contain faces of length smaller than $g$.

\subsubsection*{Our contributions}

Our first main contribution in this paper (in \Cref{sec:normal_set_of_cycles}) is to establish a Lucchesi-Younger type theorem for planar digraphs. Given an embedded planar digraph $D$ (a \emph{plane} digraph), an arc-disjoint set of directed cycles $\mathcal{N}$ of $D$ is said to be \emph{normal} if every vertex $v$ of $D$ is \emph{alternatingly oriented} with respect to $\mathcal{N}$, that is, the arcs of $\mathcal{N}$ incident to $v$ alternate between incoming and outgoing arcs in the cyclic order around $v$ (see \Cref{fig:normal_set} for an illustration).

\begin{figure}[htbp]
  \centering
  \begin{subfigure}[b]{0.53\linewidth}
    \centering
    \includegraphics[width=\linewidth]{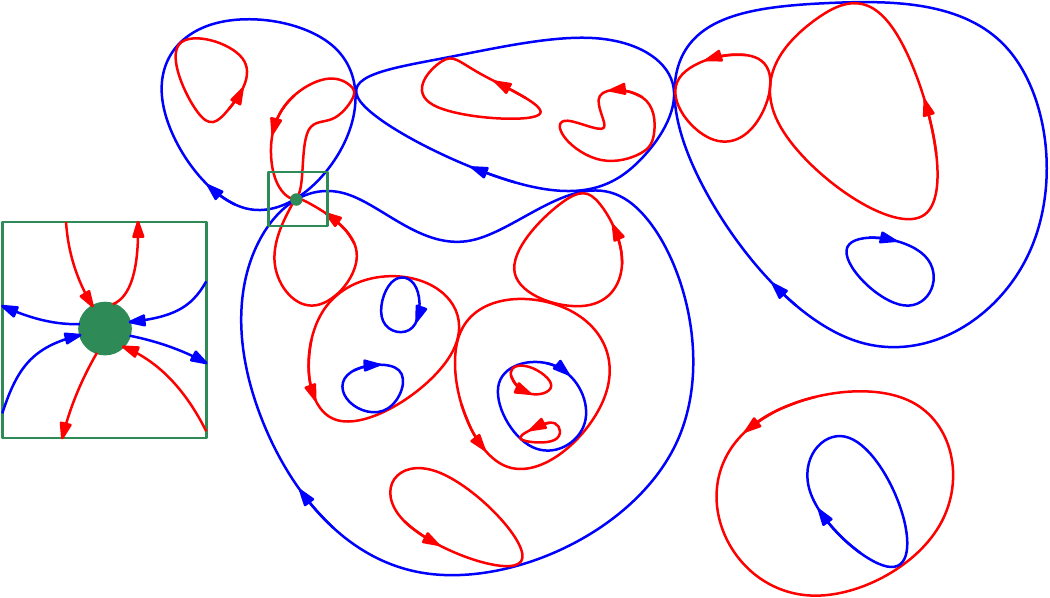}
  \end{subfigure}
  \hfill
  \begin{subfigure}[b]{0.43\linewidth}
    \centering
    \includegraphics[width=\linewidth]{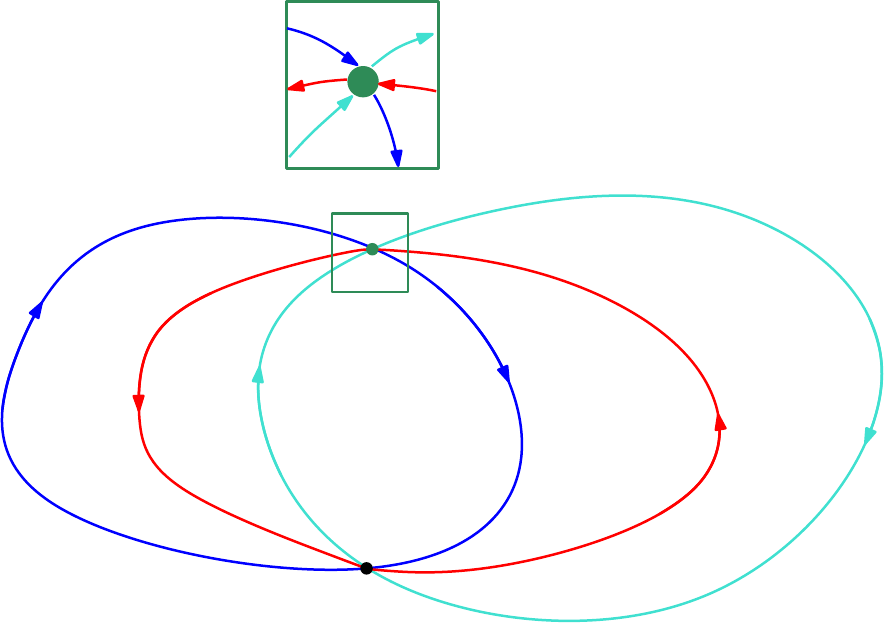}
  \end{subfigure}
  \medskip
  \caption{Examples of normal sets. The blue and red cycles are oriented clockwise and counterclockwise respectively. Each vertex is alternatingly oriented, see the zoomed-in view of the green vertices.}\label{fig:normal_set}
\end{figure}

\begin{theorem}\label{thm:maximal_normal_set_greater_than_fvs}
    Let $D$ be a directed plane graph. Let $\mathcal{N}_{\max}$ be a maximum normal set of cycles of $D$. Then $\fv(D) \le |\mathcal{N}_{\max}|$. 
\end{theorem}

Note that, unlike in the Lucchesi-Younger theorem (\Cref{thm:LY}), the inequality $\fv(D) \le |\mathcal{N}_{\max}|$ can be strict. As an example, consider the directed plane graph formed by $k$ arc-disjoint directed cycles, oriented clockwise and sharing a single common vertex. The size of a minimum feedback vertex set is $1$, while the size of a maximum normal set is $k$.
Using \Cref{thm:maximal_normal_set_greater_than_fvs} we prove the following two upper bounds for $\fv$ in \Cref{sec:normal_set_of_cycles} improving \Cref{thm:SotA_digirth} when $g\ge 4$:

\begin{theorem}\label{thm:upper_bound_FVS}
  Let $D$ be a planar oriented graph of digirth $g$ with $n \ge 3$ vertices. Then:
  \begin{displaymath}
    \fv(D) \le \frac{n-2}{g-2}
  \end{displaymath}
  and equality holds if and only if $D = C_g$, the directed cycle of length $g$.
\end{theorem}

\begin{theorem}\label{thm:upper_bound_strongly_connected}
  Let $D$ be a planar oriented graph of digirth $g \ge 4$ having $n$ vertices and $m$ arcs. Suppose that $D$ is strongly connected. Then:
  \begin{displaymath}
    \fv(D) \le \frac{n-2}{g-3} - \frac{m}{g(g-3)}.
  \end{displaymath}
\end{theorem}


The two results are complementary. \Cref{thm:upper_bound_FVS} has broader applicability since it holds for all planar oriented graphs, but \Cref{thm:upper_bound_strongly_connected} — which additionally requires strong connectivity and digirth $g \ge 4$ — is tighter whenever $D$ has many arcs relative to its order. More precisely, \Cref{thm:upper_bound_strongly_connected} improves upon \Cref{thm:upper_bound_FVS} if and only if $\frac{m}{g} > \frac{n-2}{g-2}$.

\subsection{Lower bounds}

In the second part of this paper we tackle the converse question. What is the maximal possible value of $\fv(D)$ when $D$ is a planar oriented graph on $n$ vertices of digirth $g$? Taking a disjoint union of directed cycles of length $g$ yields an infinite family of oriented graphs such that $\fv = \frac{n}{g}$. A general construction improves this trivial lower bound:

\begin{theorem}[Knauer, Valicov and Wenger, 2017~\cite{KVW17}]\label{thm:lower_bound_KVW}
    For every $g\ge 3$, there exists an infinite family $\mathcal{F}_g$ of planar oriented graphs of digirth $g$ such that every $D \in \mathcal{F}_g$ satisfies $\fv(D) = \frac{n-1}{g-1}$.
\end{theorem}

In particular, for $g=3$, the construction almost matches the upper bound of $\frac{n}{2}$ proposed in Albertson's question.

\subsubsection*{Our contributions}
In \Cref{sec:skeleton_and_coating} we construct new families of planar oriented graphs of digirth $g\ge 4$ with a large minimum feedback vertex set. We build these families using a new framework based on the notions of \emph{skeletons} and \emph{coatings}. The main idea is to start from a planar undirected graph $G$ called a \emph{skeleton}. From $G$ we construct a planar digraph $D$ by replacing each vertex $v$ of $G$ with a directed cycle $C_v$ and replacing each edge $uv$ of $G$ with a special branching between the cycles $C_u$ and $C_v$. The obtained digraph $D$ is called a \emph{coating} of $G$ (see \Cref{fig:example_coating} for an example).

\begin{figure}[htbp]
    \centering
    \includegraphics[height=5cm]{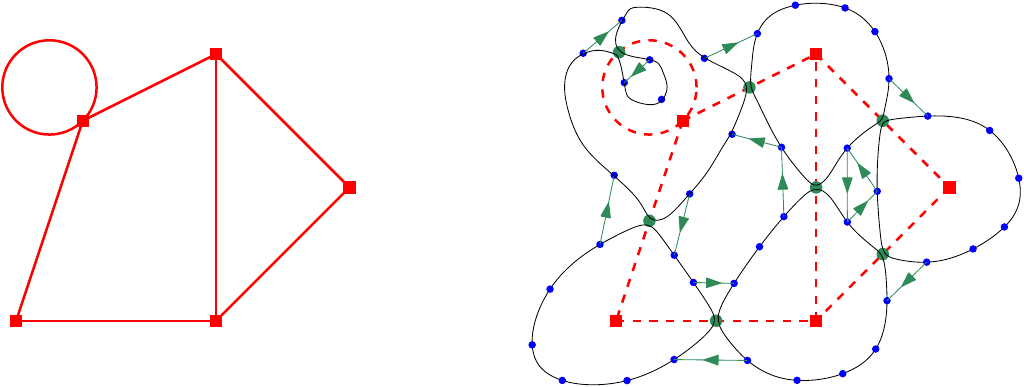}
    \caption{Left: an undirected graph $G$. Right: a coating $D$ of $G$. The black cycles associated to the vertices of $G$ are oriented clockwise.}\label{fig:example_coating}
\end{figure}

Using this new concept, we construct infinite families of planar oriented graphs with a minimum feedback vertex set larger than $\frac{n-1}{g-1}$, for every digirth $g \ge 4$:

\begin{theorem}\label{thm:general_lower_bound}
    For $g \ge 4$, there exists an infinite family $\mathcal{F}_g$ of planar oriented graphs of digirth $g$ such that every $D \in \mathcal{F}_g$ satisfies:
    \begin{itemize}
        \item Case $4 \le g \le 5$: \quad $\displaystyle\fv(D) = \frac{n}{g-1}$. (\Cref{appl:skeletonCk})
        \item Case $6 \le g \le 7$: \quad $\displaystyle\fv(D) = \frac{n - 2}{g-\frac{3}{2}}$. (\Cref{thm:general_lower_bound_small_digirth})
        \item Case $8 \le g \le 11$: \quad $\displaystyle\fv(D) = \frac{n - \frac{4(g-2)}{g+2}}{g-\frac{2g}{g+2}} = \frac{n(g+2)-4(g-2)}{g^2}$. (\Cref{thm:general_lower_bound_small_digirth})
        \item Case $g \ge 12$: \quad $\displaystyle\fv(D) = \frac{n - \frac{2g}{g+2}}{g - \frac{2g}{g+2}} = \frac{n(g+2) - 2g}{g^2}$. (\Cref{thm:general_lower_bound_high_digirth})
    \end{itemize}
\end{theorem}

 \Cref{thm:general_lower_bound} improves \Cref{thm:lower_bound_KVW} for every value of $g\ge 4$. While the improvement for $g = 4$ and $5$ is rather minor (we increase the size of the minimum feedback vertex set by an additive constant), for $g \ge 6$ we improve the asymptotic factor of the lower bound from $\frac{1}{g-1}$ to $\frac{g+2}{g^2}$ (except for the case $g = 7$), which is a significant improvement for large $g$. Moreover, several observations lead us to believe that $\frac{n(g+2)-2g}{g^2}$ is the highest possible value for $\fv(D)$ when $D$ is a planar oriented graph on $n$ vertices of digirth $g$ (see \Cref{conj:normal_family_energy}).

To compare our lower and upper bounds asymptotically when $n \to \infty$, we introduce the following notation. 
\begin{displaymath}
    \tau_g = \sup \left\{ \frac{\fv(D)}{n} \;\middle|\; D \text{ is a planar oriented graph of digirth } g \text{ on } n \text{ vertices} \right\}.
\end{displaymath}

\Cref{thm:lower_bound_KVW,thm:general_lower_bound} give a lower bound on $\tau_g$, while \Cref{thm:SotA_digirth,thm:upper_bound_FVS} give an upper bound on $\tau_g$. \Cref{tab:bounds_tau_g} summarizes the known lower and upper bounds for $\tau_g$ as well as the gaps between them, and compares them with the new results obtained in this paper. Although the old and new bounds on $\tau_g$ in light gray cells are equal, we improve the bounds by an additive constant. For example, for $g = 4$, we construct an infinite family of oriented graphs with $\fv = \frac{n}{3}$ instead of $\frac{n-1}{3}$~\cite{KVW17} and we prove that $\fv \le \frac{n-2}{2}$ instead of $\frac{n}{2}$~\cite{LM17}. As a result, the gap between the upper and lower bounds goes from $\frac{n}{6} + \frac{1}{3}$ to $\frac{n}{6} - 1$. Note that the gap between the lower and upper bounds for $\tau_g$ significantly decreases from $\frac{g-2}{g(g-1)}$ to $\frac{4}{g^2(g-2)}$ for $g \ge 6$ and $g \neq 7$.

\begin{table}[htbp]
    \centering
    \begingroup
    \renewcommand{\arraystretch}{1.75}
    \begin{tabular}{|c||c|c||c|c||c|c|}
        \hline
        \multirow{2}{*}{Digirth $g$} & \multicolumn{2}{c||}{Lower bounds} & \multicolumn{2}{c||}{Upper bounds} & \multicolumn{2}{c|}{Gap}  \\
        \cline{2-7}
        & Old & New & Old & New & Old & New \\
        \hline
        3 & \multirow{6}{*}{$\frac{1}{g-1}$~\cite{KVW17}}& \cellcolor{gray} & $\frac{3}{5}$~\cite{B79} & \cellcolor{gray} & $\frac{1}{10}$ & \cellcolor{gray} \\
        \cline{1-1} \cline{3-7}
        4  &  & \cellcolor{gray!25} & $\frac{1}{2}$~\cite{LM17} & \cellcolor{gray!25} $\frac{1}{2}$ & $\frac{1}{6}$ & \cellcolor{gray!25} $\frac{1}{6}$ \\
        \cline{1-1} \cline{4-7}
        5  & & \multirow{-2}{*}{\cellcolor{gray!25} $\frac{1}{g-1}$} & $\frac{1}{2}$~\cite{ELM17} & \multirow{4}{*}{$\frac{1}{g-2}$} & $\frac{1}{4}$ & $\frac{1}{12}$ \\
        \cline{1-1} \cline{3-4} \cline{6-7}
        6  & & & \multirow{3}{*}{$\frac{2}{g}$~\cite{ELM17}} & & $\frac{2}{15}$ & $\frac{1}{36}$ \\
        \cline{1-1} \cline{6-7}
        7 & & \multirow{-2}{*}{$\frac{1}{g-\frac{3}{2}}$} &  & & $\frac{5}{42}$  &  $\frac{1}{55}$\\
        \cline{1-1} \cline{3-3} \cline{6-7}
        $\ge 8$ & & $\frac{g+2}{g^2}$ & & & $\frac{g-2}{g(g-1)}$ & $\frac{4}{g^2(g-2)}$ \\
        \hline
    \end{tabular}
    \endgroup
    \medskip
    \caption{\label{tab:bounds_tau_g}Bounds for $\tau_g$: previous results and new bounds obtained in this paper. \colorbox{gray!25}{Light gray} cells ($g \in \{4,5\}$) indicate that the asymptotic value of $\tau_g$ is unchanged, but the bounds are improved by an additive constant. \colorbox{gray}{Dark gray} cells ($g = 3$) indicate that no new result is obtained.}
\end{table}

\subsection*{Definitions and notations}
Given a digraph $G$, we denote by $V(G)$ and $E(G)$ the sets of vertices and arcs of $G$. We will always work on \emph{plane graphs}, that is, planar graphs with a fixed embedding. For a plane graph $G$, we denote by $F(G)$ the set of faces of $G$ in its planar embedding. We denote by $n_G, m_G$ and $f_G$ (or simply $n,m$ and $f$ when there is no ambiguity) the numbers of vertices, arcs and faces of $G$.

For a vertex $v \in V(G)$ we denote by $\degin_G(v)$ (resp. $\degout_G(v)$) the number of incoming (resp. outgoing) arcs of $v$ in $G$ and we call it the \emph{indegree} (resp. \emph{outdegree}) of $v$. The \emph{degree} of $v$ is $\deg_G(v)= \degin_G(v) + \degout_G(v)$.

We denote by $\ell_F$ the length of a face $F \in F(G)$, that is, the number of arcs on the boundary of $F$ (counting multiple occurrences of an arc if needed).

For a cycle $C$ of $G$ the notation $\lvert C \rvert$ denotes the length of $C$ which is the number of arcs of $C$.

A plane sub(di)graph $H$ of a plane (di)graph $G$ is obtained by removing some arcs and vertices of $G$ without changing the embedding.

Given a set of arc-disjoint cycles $\mathcal{C}$ of an plane digraph $G$, the graph $G[\mathcal{C}]$ is the subgraph of $G$ induced by the arcs of the cycles of $\mathcal{C}$ where all isolated vertices (vertices not incident to any arc of $\mathcal{C}$) are removed. We denote by $n_\mathcal{C}, m_\mathcal{C}, f_\mathcal{C}$ and $c_\mathcal{C}$ the number of vertices, of arcs, of faces and of connected components of $G[\mathcal{C}]$ respectively. We denote by $V(\mathcal{C}), E(\mathcal{C})$ and $F(\mathcal{C})$ the sets of vertices, of arcs and of faces of $G[\mathcal{C}]$.

In order to capture the local structure of arcs around each vertex $v$, we will use the term \emph{half-arcs of $v$} to refer to the portion of an arc incident to $v$, considered from $v$'s perspective. For instance, if $v$ has a loop, then $v$ has two half-arcs induced by this arc: one outgoing and one incoming. A vertex $v \in V(G)$ is said to be \emph{alternatingly oriented} if every pair of two consecutive \emph{half-arcs} around $v$ in the cyclic order have different orientations with respect to $v$.

A set $\mathcal{N}$ of arc-disjoint directed cycles of $G$ is called \emph{normal} if $G[\mathcal{N}]$ has only alternatingly oriented vertices. We give some examples in \Cref{fig:normal_set}. We denote $\mathcal{N}_{\max}(G)$ a maximum normal set of $G$.

\section{Upper bounds: normal sets of cycles}\label{sec:normal_set_of_cycles}
\subsection{Outline of the section}

Denote by $\mathcal{C}_{\max}(G)$ a maximum set of arc-disjoint cycles of a directed plane graph $G$. Since every normal set is a set of arc-disjoint cycles we have $\Modul{\mathcal{N}_{\max}(G)} \le \Modul{\mathcal{C}_{\max}(G)}$. Inspired by the ideas from the proof of Lucchesi and Younger~\cite{LY78}, we establish an analogous inequality between $\fv(G)$ and $\Modul{\mathcal{N}_{\max}(G)}$ in the next theorem.

\begin{reptheorem}{thm:maximal_normal_set_greater_than_fvs}
  Let $G$ be a directed plane graph. Then $\fv(G) \le |\mathcal{N}_{\max}(G)|$.
\end{reptheorem}
 
To bound $\fv(G)$ from above, it therefore suffices to bound the size of $\mathcal{N}_{\max}(G)$, which we do in the following proposition.

\begin{proposition}\label{prop:upper_bound_for_maximal_normal_set}
  Let $G$ be an oriented plane graph on $n$ vertices of digirth $g$. Then $\Modul{\mathcal{N}_{\max}(G)} \le \frac{n-2}{g-2}$.
\end{proposition}
This is a direct corollary of a stronger statement (\Cref{thm:energy}) proved in \Cref{sec:upper_bound_maximal_normal_set}, where we also construct an infinite family of oriented plane graphs attaining the bound.

As a direct consequence of \Cref{thm:maximal_normal_set_greater_than_fvs} and \Cref{prop:upper_bound_for_maximal_normal_set}, we obtain upper bounds for $\fv(G)$ in terms of $n$ and $g$, or of $n$, $m$ and $g$.
   
\begin{reptheorem}{thm:upper_bound_FVS}
  Let $G$ be a planar oriented graph of digirth $g$ with $n \ge 3$ vertices. Then:
  \begin{displaymath}
    \fv(G) \le \frac{n-2}{g-2}
  \end{displaymath}
  and equality holds if and only if $G = C_g$, the directed cycle of length $g$.
\end{reptheorem}

While the inequality follows immediately from \Cref{thm:maximal_normal_set_greater_than_fvs} and \Cref{prop:upper_bound_for_maximal_normal_set}, characterising when equality holds requires more work. When $G$ is moreover strongly connected, the arc count $m$ provides another upper bound.

\begin{reptheorem}{thm:upper_bound_strongly_connected}
  Let $G$ be a planar oriented graph of digirth $g \ge 4$ having $n$ vertices and $m$ arcs. Suppose that $G$ is strongly connected. Then:
  \begin{displaymath}
    \fv(G) \le \frac{n-2}{g-3} - \frac{m}{g(g-3)}.
  \end{displaymath}
\end{reptheorem}

The rest of the section is organised as follows. In \Cref{sec:upper_bound_maximal_normal_set} we prove \Cref{prop:upper_bound_for_maximal_normal_set} and establish its tightness. Both \Cref{thm:upper_bound_FVS} and \Cref{thm:upper_bound_strongly_connected} are then deduced assuming \Cref{thm:maximal_normal_set_greater_than_fvs}: \Cref{thm:upper_bound_strongly_connected} is proved in \Cref{sec:upper_bound_strongly_connected}, while for \Cref{thm:upper_bound_FVS} we first introduce laminar (multi)sets of cycles (sets of cycles that do not cross each other) in \Cref{sec:laminarity} and then settle the equality case in \Cref{sec:tightness_general_upper_bound_fvs}. We conclude with the proof of \Cref{thm:maximal_normal_set_greater_than_fvs} in \Cref{sec:proof_fvs_lower_than_maximal_normal_set}.

\subsection{Energies of an arc-disjoint set of cycles: proof of \Cref{prop:upper_bound_for_maximal_normal_set}}\label{sec:upper_bound_maximal_normal_set}

\begin{observation}\label{obs:double_counting_G[N]}
    Let $\mathcal{C}$ be a set of arc-disjoint directed cycles of $G$. Then 
    \begin{displaymath}
        m_\mathcal{C} = \sum_{C \in \mathcal{C}} \lvert C \rvert = \frac{1}{2} \sum_{F \in F(\mathcal{C})} \ell_F
    \end{displaymath}
\end{observation}

\begin{definition}[Energies]\label{def:energies}
    Let $\mathcal{C}$ be a set of arc-disjoint directed cycles of a plane digraph $G$ of digirth $g$.
    We define the \emph{energies} of $\mathcal{C}$ in $G$ as follows:
    \begin{itemize}
        \item \emph{Energy of cycles}: $\mathbf{E_1}(\mathcal{C}) = \frac{g-2}{g} \sum_{C \in \mathcal{C}} ( \lvert C \rvert - g)$ 
        \item \emph{Energy of faces}: $\mathbf{E_2}(\mathcal{C}) = \frac{1}{g} \sum_{F \in F(\mathcal{C})} ( \ell_F -g)$ 
        \item \emph{Energy of internal vertices}: $\mathbf{E_3}(\mathcal{C}) = n_G - n_\mathcal{C}$
        \item \emph{Energy of components}: $\mathbf{E_4}(\mathcal{C}) = c_\mathcal{C} - 1$
        \item \emph{Total energy}: $\mathbf{E_{tot}}(\mathcal{C})=\mathbf{E_1}(\mathcal{C})+\mathbf{E_2}(\mathcal{C})+\mathbf{E_3}(\mathcal{C})+\mathbf{E_4}(\mathcal{C})$
    \end{itemize}
\end{definition}

\begin{proposition}\label{prop:positive_energies}
    If $\mathcal{N}$ is a normal set of cycles of a plane digraph $G$ of digirth $g$, then $\mathbf{E_1}(\mathcal{N})$, $ \mathbf{E_2}(\mathcal{N})$, $\mathbf{E_3}(\mathcal{N})$ and $\mathbf{E_4}(\mathcal{N})$ are all non-negative.
\end{proposition}

\begin{proof}
    By definition $\mathbf{E_3}(\mathcal{N}) \ge 0$ and $\mathbf{E_4}(\mathcal{N}) \ge 0$. Moreover, $\mathbf{E_1}(\mathcal{N}) \ge 0$ since $g$ is the digirth of $G$. Finally, as $\mathcal{N}$ is a normal set of cycles, the faces of $G[\mathcal{N}]$ are bounded by directed cycles. Indeed, if there is a face $F \in F(\mathcal{N})$ that is not bounded by directed cycles, then there exist two consecutive $e$ and $e'$ along $F$ that have opposite directions. But then the vertex adjacent to both $e$ and $e'$ cannot be alternatingly oriented, contradiction. Thus $\mathbf{E_2}(\mathcal{N}) \ge 0$. 
\end{proof}

\begin{theorem}\label{thm:energy}
    Let $\mathcal{N}$ be a normal set of cycles of a plane digraph $G$ of order $n$ and digirth $g$. Then:
    \begin{displaymath}
        (g-2) \lvert \mathcal{N} \rvert = n-2 - \mathbf{E_{tot}}(\mathcal{N}) \le n-2
    \end{displaymath}
\end{theorem}

\begin{proof}
    By Euler's formula $n_\mathcal{N} + f_\mathcal{N} = m_\mathcal{N} + 2 + (c_\mathcal{N} - 1)$. Hence by \Cref{obs:double_counting_G[N]} we get
    \begin{displaymath}
        n-(n - n_\mathcal{N}) +  f_\mathcal{N} = \frac{g-2}{g} \sum_{C \in \mathcal{N}} \lvert C \rvert + \frac{2}{g} \times \frac{1}{2} \sum_{F \in F(\mathcal{N})} \ell_F + 2 + (c_\mathcal{N} - 1)
    \end{displaymath}
    Then it follows that:
    \begin{displaymath}
        \begin{split}
            n-2 &= \frac{g-2}{g} \sum_{C \in \mathcal{N}} (g + \lvert C \rvert - g) +  \frac{1}{g} \sum_{F \in F(\mathcal{N})} \ell_F - f_\mathcal{N} + (n - n_\mathcal{N}) + (c_\mathcal{N} - 1) \\
            &= (g-2) \lvert \mathcal{N} \rvert + \frac{g-2}{g} \sum_{C \in \mathcal{N}} ( \lvert C \rvert - g) + \frac{1}{g} \sum_{F \in F(\mathcal{N})} ( \ell_F -g) + (n - n_\mathcal{N}) + (c_\mathcal{N} - 1) \\
            &= (g-2) \lvert \mathcal{N} \rvert +\mathbf{E_1}(\mathcal{N})+\mathbf{E_2}(\mathcal{N})+\mathbf{E_3}(\mathcal{N})+\mathbf{E_4}(\mathcal{N})
        \end{split}
    \end{displaymath}
    Then we get the desired formula and the inequality comes directly from \Cref{prop:positive_energies}.
\end{proof}

Clearly, \Cref{thm:energy} implies \Cref{prop:upper_bound_for_maximal_normal_set}. Now we show that the upper bound of \Cref{prop:upper_bound_for_maximal_normal_set} is tight, by building an infinite family of oriented plane graphs reaching the upper bound of the theorem. We first show the construction for digirth $g=3$ and then we generalize to larger values of $g$. For every integer $k \ge 1$, let $O_k$ be the plane directed graph on $n_k=3k$ vertices with digirth $g=3$ and $\Modul{\mathcal{N}_{\max}(O_k)} = 3k-2 = \frac{n_k-2}{g-2}$. The base case is the graph $O_1$ depicted in \Cref{subfig:octahedron-left}. Each vertex of $O_1$ is alternatingly oriented, so the only directed cycle of $O_1$ constitutes a normal set of size $\Modul{\mathcal{N}_{\max}(O_1)} = 1 = \frac{n_1-2}{g-2}$. Then, to construct $O_{k+1}$ from $O_k$, we consider the central triangular face $xyz$ of $O_k$ and add three new  vertices $a,b,c$ and nine new arcs inside $xyz$ as depicted in \Cref{subfig:octahedron}. By construction, the vertices of $O_{k+1}$ are also alternatingly oriented. The set $\mathcal{N}_{\max}(O_k) \cup \{xba,yac,zcb\}$ forms a normal set of $O_{k+1}$ of size $\Modul{\mathcal{N}_{\max}(O_{k})} + 3 = 3(k+1)-2 = \frac{n_{k+1}-2}{g-2}$. Thus by \Cref{prop:upper_bound_for_maximal_normal_set}, this normal set is maximum and therefore $\Modul{\mathcal{N}_{\max}(O_{k+1})} = \frac{n_{k+1}-2}{g-2}$.

\begin{figure}[htbp]
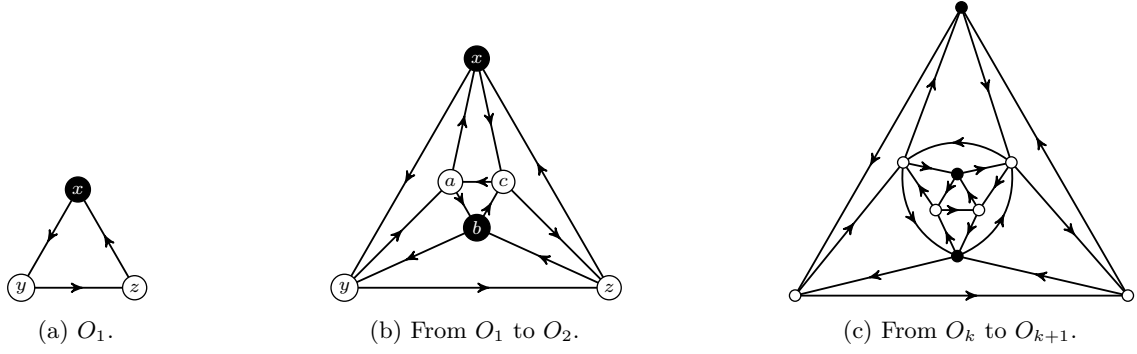

  \begin{subfigure}[b]{0.15\linewidth}
    \centering 
    \begin{tikzpicture}[scale=0.3]
      \input{Figures/triangle.tex}
    \end{tikzpicture}
    \caption{$O_1$.}\label{subfig:octahedron-left}
  \end{subfigure}
  \hfill
  \begin{subfigure}[b]{0.3\linewidth}
    \centering
    \begin{tikzpicture}[scale=0.7]
      \input{Figures/octahedron.tex}
    \end{tikzpicture}
    \caption{From $O_1$ to $O_{2}$.}\label{subfig:octahedron}
  \end{subfigure}
  \hfill
  \begin{subfigure}[b]{0.3\linewidth}
    \centering
    \begin{tikzpicture}[scale=2.2,rotate=180]
      \input{Figures/octahedron_third_iteration.tex}
    \end{tikzpicture}
    \caption{From $O_k$ to $O_{k+1}$.}\label{subfig:stacked-octahedron}
  \end{subfigure}
  \smallskip
  \caption{Construction of an infinite family of graphs with $g= 3$ and $\Modul{\mathcal{N}_{\max}(G)} = \frac{n-2}{g-2}$. The black vertices form a minimum feedback vertex set.}\label{fig:octahedron}
\end{figure}

In order to generalize $O_k$ for higher digirth $g$, we use a minimum feedback vertex set $S$ of $O_k$.

First observe that $|S| = \fv(O_k)=  k$. Indeed, for the base case $O_1$ it is clear that $\fv(O_1) = 1$. Suppose that $\fv(O_k) = k$ and let $S$ be a solution of this size. Note that one of the vertices of the face $xyz$ must be in $S$, say $x \in S$. Then take $S' = S \cup \{ b \}$ as a feedback vertex set of $O_{k+1}$. Note that $S'$ is minimum with respect to $O_{k+1}$ since $abc$ is a directed cycle and we are done.

Now, for $g > 3$, we build a plane digraph $O_k^{(g)}$ by subdividing $g-3$ times each in-going arc to $S$. For $k \ge 2$, $S$ contains 2 vertices of in-degree 2 (vertices on the central and outer faces) and $k-2$ vertices of in-degree 3 so $2 \times 2 + 3 \times (k-2) = 3k-2$ arcs are subdivided. Thus $O_k^{(g)}$ has order $n_k^{(g)} = 3k + (g-3)(3k - 2)$. Moreover, since $S$ is a feedback vertex set of $O_k^{(g)}$, every cycle of $O_k^{(g)}$ passes through a subdivided arc ingoing to a vertex of $S$, and thus it has length at least $g$. Finally, note that the subdivisions do not change the size of the maximum normal set of cycles and therefore:

\begin{displaymath}
  \frac{n_k^{(g)} - 2}{g-2} = \frac{3k + (3k - 2)(g-3) - 2}{g-2} = 3k-2 = |\mathcal{N}_{\max}(O_{k})| = |\mathcal{N}_{\max}(O_{k}^{(g)})|
\end{displaymath}

\subsection{Strongly connected digraphs: proof of \Cref{thm:upper_bound_strongly_connected}}\label{sec:upper_bound_strongly_connected}

\begin{proposition}\label{prop:facial_cycle}
  Let $G$ be a strongly connected plane digraph. Let $\mathcal{F}$ be a connected set of faces of $G$ such that the boundary of the union of $\mathcal{F}$ is a directed cycle $C$. Then some face in $\mathcal{F}$ has a directed facial boundary.
\end{proposition}

\begin{proof}
  Denote by $G^*$ the dual of $G$, that is the dual graph of the underlying undirected graph of $G$ where each arc is oriented by a quarter-turn clockwise from its corresponding primal arc.
  
  Since $G$ is strongly connected, $G^*$ is acyclic. Since $\mathcal{F}$ is a connected set of faces of $G$, $\mathcal{F}^*$ is a set of vertices of $G^*$ such that the subgraph $G^*[\mathcal{F}^*]$ is connected. Moreover since the boundary of $\mathcal{F}$ is a directed cycle, all the arcs of $G^*$ that have an endpoint in $\mathcal{F}^*$ and an endpoint outside of $\mathcal{F}^*$ are either all ingoing or all outgoing. Without loss of generality, suppose that they are all ingoing. Since $G^*[\mathcal{F}^*]$ is a connected acyclic graph, it has a sink $F^* \in \mathcal{F}^*$ and from the previous remark, $F^*$ is a sink in $G^*$. By duality, $F^*$ corresponds to a face $F \in \mathcal{F}$ whose boundary is a directed cycle.
\end{proof}

\begin{lemma}\label{lem:number_facial_cycle}
  Let $G$ be a strongly connected plane digraph. Let $\mathcal{N}$ be a normal set of cycles of $G$. Then $G$ has at least $f_\mathcal{N} - (c_\mathcal{N} - 1)$ faces whose boundary is a directed cycle.
\end{lemma}

\begin{proof}
  Let $y$ be the number of faces of $G[\mathcal{N}]$ whose boundary is a directed cycle. 
  We show by induction on $c_\mathcal{N}$ that $y \ge f_\mathcal{N} - (c_\mathcal{N} - 1)$.
  
  If $G[\mathcal{N}]$ is connected then since $\mathcal{N}$ is a normal set of cycles, every face is bounded by a directed cycle and the formula holds.

  Suppose that $G[\mathcal{N}]$ has $c_\mathcal{N} \ge 2$ connected components. Let $G_1$ be one of the connected components of $G[\mathcal{N}]$ and set $G_2 = G[\mathcal{N}] - G_1$. Let $\mathcal{N}_1$ and $\mathcal{N}_2$ be the partition of $\mathcal{N}$ such that $G_1 = G[\mathcal{N}_1]$ and $G_2 = G[\mathcal{N}_2]$. We have $c_{\mathcal{N}_1} = 1$ and $c_{\mathcal{N}_2} = c_\mathcal{N} - 1$. Let $y_1$ and $y_2$ be respectively the number of faces of $G_1$ and $G_2$ whose boundary is a directed cycle. By induction, $y_1 \ge f_{\mathcal{N}_1} - (c_{\mathcal{N}_1} - 1)$ and $y_2 \ge f_{\mathcal{N}_2} - (c_{\mathcal{N}_2} - 1)$. Now note that $y \ge y_1 + y_2 -2$. Indeed, let $F_1$ (resp.\ $F_2$) be the face of $G_1$ (resp.\ $G_2$) that contains the other component; these two faces merge into a single face $F$ in $G[\mathcal{N}]$. Since $c_{\mathcal{N}_1} = 1$, every face of $G_1$ has a directed boundary (base case), so in particular $F_1$ is counted in $y_1$. In the worst case $F_2$ is also counted in $y_2$ but the merged face $F$ is not directed, giving $y \ge (y_1 - 1) + (y_2 - 1) = y_1 + y_2 - 2$. Moreover we have $c_\mathcal{N} = c_{\mathcal{N}_1} + c_{\mathcal{N}_2}$ and $f_\mathcal{N} = f_{\mathcal{N}_1} + f_{\mathcal{N}_2} - 1$ since the face of $G[\mathcal{N}]$ that is incident to both $G_1$ and $G_2$ is counted twice in $f_{\mathcal{N}_1} + f_{\mathcal{N}_2}$. Thus we have:
  \begin{displaymath}
    y \ge y_1 + y_2 - 2 \ge f_{\mathcal{N}_1} - (c_{\mathcal{N}_1} - 1) + f_{\mathcal{N}_2} - (c_{\mathcal{N}_2} - 1) - 2 = f_\mathcal{N} + 1 - c_\mathcal{N}.
  \end{displaymath}
  This establishes $y \ge f_\mathcal{N} - (c_\mathcal{N} - 1)$ by induction.

  Now for every face $F$ of $G[\mathcal{N}]$ whose boundary is a directed cycle, denote by $\mathcal{F}_F$ the set of faces of $G$ contained in $F$. Note that $\mathcal{F}_F$ is a connected set of faces of $G$ whose boundary is a directed cycle. Then one can apply \Cref{prop:facial_cycle} to get that there exists a face in $\mathcal{F}_F$ whose boundary is a directed cycle. Since the $\mathcal{F}_F$ are disjoint for different faces $F$ of $G[\mathcal{N}]$, we get that there are at least $y \ge f_\mathcal{N} - (c_\mathcal{N} - 1)$ faces of $G$ whose boundary is a directed cycle.
\end{proof}

We can now prove \Cref{thm:upper_bound_strongly_connected}.

\begin{proof}[Proof of \Cref{thm:upper_bound_strongly_connected}]
  Let $f$ be the total number of faces of $G$ and let $x$ be the number of faces of $G$ whose boundary is a directed cycle. Using Euler's formula ($G$ is connected since it is strongly connected) and the fact that facial cycles have length at least $g$ and other faces have length at least 3, we have:
  \begin{displaymath}
    2m = \sum_{F \in F(G)} \ell_F \ge g x + 3(f - x) = gx + 3(m+2-n - x) \quad \text{thus} \quad x \le \frac{3n-6-m}{g-3}.
  \end{displaymath}

  From \Cref{thm:maximal_normal_set_greater_than_fvs}, $G$ has a normal set of cycles $\mathcal{N}$ of size $k = \fv(G)$. We compute $\mathbf{E_2}(\mathcal{N})$ (see \Cref{def:energies}). We use the fact that $\sum_{F \in F(\mathcal{N})} \ell_F = 2 m_\mathcal{N} = 2 \sum_{C \in \mathcal{N}} \lvert C \rvert$ and that $\lvert C \rvert \ge g$ for every $C \in \mathcal{N}$ to get:
  \begin{displaymath}
    \mathbf{E_2}(\mathcal{N}) = \frac{1}{g} \sum_{F \in F(\mathcal{N})} ( \ell_F -g) = \frac{1}{g} \left( 2 \sum_{C \in \mathcal{N}} \lvert C \rvert - g f_\mathcal{N} \right) \ge \frac{1}{g}(2gk - g f_\mathcal{N}) = 2k - f_\mathcal{N}.
  \end{displaymath}
  Moreover since $\mathbf{E_1}(\mathcal{N}) \ge 0$, $\mathbf{E_3}(\mathcal{N}) \ge 0$ (from \Cref{prop:positive_energies}) and $\mathbf{E_4}(\mathcal{N}) = c_\mathcal{N} - 1 \ge 0$ we get:
  \begin{displaymath}
    \mathbf{E_{tot}}(\mathcal{N}) = \mathbf{E_1}(\mathcal{N}) + \mathbf{E_2}(\mathcal{N}) + \mathbf{E_3}(\mathcal{N}) + \mathbf{E_4}(\mathcal{N}) \ge 2k - f_\mathcal{N} + c_\mathcal{N} - 1.
  \end{displaymath}
  By \Cref{lem:number_facial_cycle}, we have $x \ge f_\mathcal{N} - (c_\mathcal{N} - 1)$. Then we have:
  \begin{displaymath}
    \mathbf{E_{tot}}(\mathcal{N}) \ge 2k - f_\mathcal{N} + c_\mathcal{N} - 1 \ge 2k - x \ge 2k - \frac{3n-6-m}{g-3}.
  \end{displaymath}
  Finally from \Cref{thm:energy}, we have $k = \frac{n-2}{g-2} - \frac{\mathbf{E_{tot}}(\mathcal{N})}{g-2}$ and thus:
  \begin{displaymath}
    \begin{split}
      k &\le \frac{n-2}{g-2} - \frac{1}{g-2} \left( 2k - \frac{3n-6-m}{g-3} \right) \\
      (g-2) k &\le n-2 - 2k + \frac{3n-6-m}{g-3} \\
      gk &\le n-2 + \frac{3(n-2)-m}{g-3} = \frac{(n-2)(g-3) + 3(n-2) - m}{g-3} = \frac{g(n-2) - m}{g-3} \\
      \fv(G) = k &\le \frac{n-2}{g-3} - \frac{m}{g(g-3)}.
    \end{split}
  \end{displaymath}
\end{proof}

\subsection{Laminar (multi)sets of cycles}\label{sec:laminarity}

Let $G$ be a plane digraph. For a given directed cycle $C$ of $G$, we define a \emph{region} of $C$ as the closed area of the plane bounded by $C$ and is denoted by $\overline{C}$; the notation $\mathring{C}=\overline{C}\setminus C$ denotes the interior of the region of $C$.

Two directed cycles $C_1$ and $C_2$ of $G$ are said to be \emph{crossing} if $\mathring{C_1} \cap \mathring{C_2} \neq \emptyset$ and $\mathring{C_1} \not\subseteq \mathring{C_2}$, and $\mathring{C_2} \not\subseteq \mathring{C_1}$. That is, $C_1$ and $C_2$ cross if there is at least one element (arc or vertex) of $G$ lying in $\mathring{C_1} \cap C_2$ and at least one element of $G$ lying in $\mathring{C_2} \cap C_1$. See \Cref{fig:cycles_with_crossing} for an example.

\begin{figure}[htbp]
  \centering
  \includegraphics[width=0.4\linewidth]{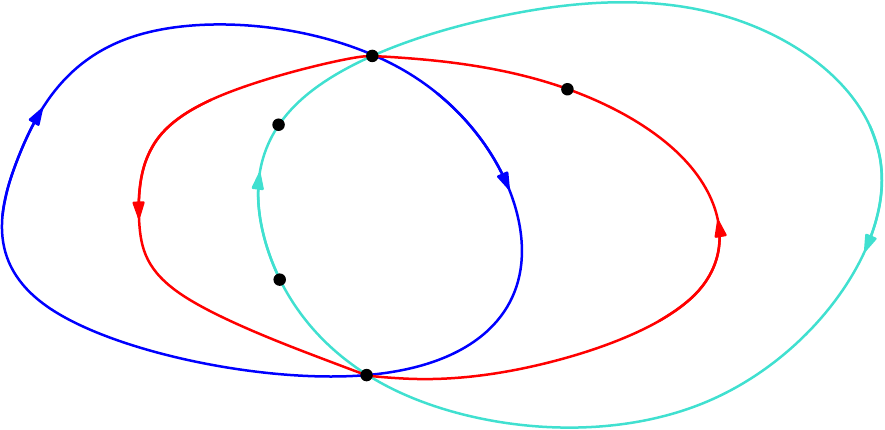}
  \caption{Example of three pairwise crossing cycles.}\label{fig:cycles_with_crossing}
\end{figure}

A set of pairwise non-crossing directed cycles is called \emph{laminar}. Since two directed cycles can share arcs without crossing, a laminar family is not necessarily arc-disjoint.

In the rest of this section we will need the following \namecref{prop:operation_lam}, already shown in a more general setting~\cite{LY78,Lovasz76}. 

\begin{proposition}\label{prop:operation_lam}
  Given a plane graph $G$ and a multiset of directed cycles $\mathcal{F}$
  of $G$, there exists a multiset of cycles $\mathcal{F}'$ of $G$ such that:
  \begin{enumerate}
  \item $|\mathcal{F}'| = |\mathcal{F}|$,
  \item for every arc $e$, the number of cycles in $\mathcal{F}'$
   containing $e$ is equal to the number of cycles in $\mathcal{F}$
   containing $e$,
  \item $\mathcal{F}'$ is laminar.
  \end{enumerate}
\end{proposition}

\begin{proof}
  We start with the initial multiset $\mathcal{F'}:=\mathcal{F}$, and as long as there are two crossing cycles in $\mathcal{F'}$, we transform them as shown in \Cref{fig:laminar_transfo}.
  \begin{figure}[htbp]
    \centering
     \begin{subfigure}[b]{0.3\linewidth}   
    \centering 
      \includegraphics[scale=0.45]{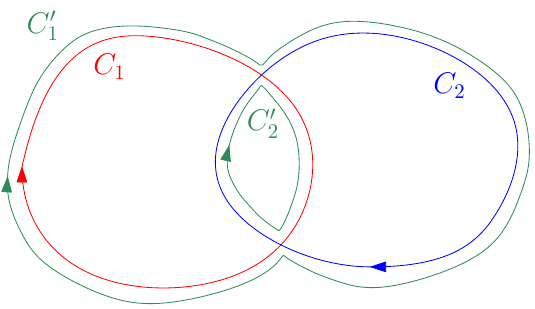}
    \caption{}\label{subfig:laminar_transfo_a}
  \end{subfigure}  
  \begin{subfigure}[b]{0.3\linewidth}   
    \centering 
      \includegraphics[scale=0.45]{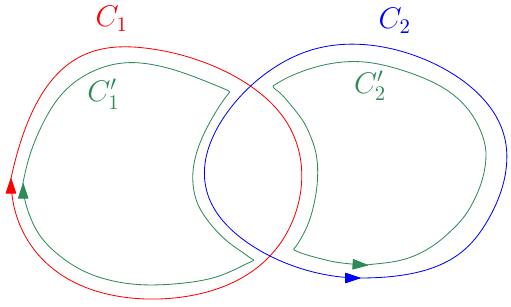}
    \caption{}\label{subfig:laminar_transfo_b}
  \end{subfigure}  
  \begin{subfigure}[b]{0.3\linewidth}
    \centering 
      \includegraphics[scale=0.45]{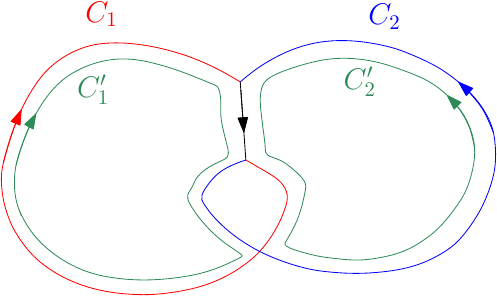}
    \caption{}\label{subfig:laminar_transfo_c}
  \end{subfigure}  
    \caption{Transformations of two crossing oriented cycles (red and blue) into two non-crossing cycles (green).}\label{fig:laminar_transfo}
  \end{figure} 
 
  These transformations turn two oriented cycles (red and blue) into two other oriented cycles (green) that use exactly the same arcs and if an arc belongs to both cycles $C_1$ and $C_2$ then it belongs to both $C_1'$ and $C_2'$. At the end of the process, we obtain a multiset $\mathcal{F'}$ that satisfies the three properties of the proposition. It remains to show that this process terminates.
    
  For a region $R$ of the graph, let $||R||$ be the number of elements (vertices and arcs) in the region $R$ of the graph. For a multiset of cycles $\mathcal{F}$, we define the quantity $\psi(\mathcal{F}) =  \prod_{C \in \mathcal{F}} (1 + ||\mathring{C}||)$.
  Note that $\psi(\mathcal{F})$ is always a positive integer. We show that this quantity strictly decreases at each transformation. Indeed, given two cycles $C_1$ and $C_2$, transformed into two cycles $C_1'$ and $C_2'$, we prove that the following holds: $(1 + ||\mathring{C_1'}||) \cdot (1 + ||\mathring{C_2'}||) < (1 + ||\mathring{C_1}||) \cdot (1 + ||\mathring{C_2}||)$.

  Let $x = ||\mathring{C_1}||$, $y = ||\mathring{C_2}||$, and $i = ||\mathring{C_1} \cap \mathring{C_2}||$. Since $C_1$ and $C_2$ cross, we have $x-i \ge ||\mathring{C_1} \cap C_2|| \ge 1$, and $y-i \ge ||\mathring{C_2} \cap C_1|| \ge 1$.
  
  There are two cases depending on the orientation of $C_1$ and $C_2$ (both clockwise or counterclockwise, or one clockwise and the other counterclockwise), as illustrated on \Cref{fig:laminar_transfo}:

  \begin{itemize}
  \item If both $C_1$ and $C_2$ are oriented in the same direction (see \Cref{subfig:laminar_transfo_a}):
    \begin{displaymath}
      \begin{split} 
        (1 + ||\mathring{C_1'}||) \cdot (1 + ||\mathring{C_2'}||) &= (1 + x + y - i) \cdot (1 + i) = 1 + x + y + xi + yi - i^2 \\
                                                                  &= (1+x) \cdot (1+y) - (x-i) \cdot (y-i)  < (1+x) \cdot (1+y)
      \end{split}
    \end{displaymath}
      \item If $C_1$ is clockwise and $C_2$ is anticlockwise (see \Cref{subfig:laminar_transfo_b,subfig:laminar_transfo_c}):
    \begin{displaymath}
      (1 + ||\mathring{C_1'}||) \cdot (1 + ||\mathring{C_2'}||) = (1 + x - i - ||\mathring{C_1} \cap C_2||) \cdot (1 + y-i - ||\mathring{C_2} \cap C_1||) < (1+x) \cdot (1+y)
    \end{displaymath}

  \end{itemize}
  Thus, $\psi(\mathcal{F})$ strictly decreases at each transformation, and since it is a positive integer, the process terminates.
\end{proof}

\begin{remark*}
  If $\mathcal{N}$ is a normal set of cycles of a plane digraph $G$ then the associated laminar set of cycles $\mathcal{N'}$ obtained with \Cref{prop:operation_lam} is also normal. Indeed $\mathcal{N'}$ uses exactly the same arcs as $\mathcal{N}$ and hence $G[\mathcal{N}'] = G[\mathcal{N}]$ and so each vertex is still alternatingly oriented.
\end{remark*}

\subsection{Digraphs reaching the upper bound: proof of \Cref{thm:upper_bound_FVS}}\label{sec:tightness_general_upper_bound_fvs}

For a plane digraph $G$ of order $n$ and digirth $g \ge 3$, \Cref{thm:upper_bound_FVS} states that  $\fv(G) \le \frac{n-2}{g-2}$ with equality if and only if $G = C_g$ the directed cycle of length $g$. Recall that the upper bound directly follows from \Cref{thm:maximal_normal_set_greater_than_fvs,prop:upper_bound_for_maximal_normal_set}. In this subsection, we prove
that it is tight only for $C_g$. 

\begin{lemma}\label{lem:complete_arc_disjoint_set_of_cycles}
    Let $\mathcal{C}, \mathcal{C'}$ be two sets of arc-disjoint directed cycles of $G$ such that $E(\mathcal{C}') \subsetneq E(\mathcal{C})$. Then there exists a directed cycle $C$ of $G$ that uses only arcs in $E(\mathcal{C}) \backslash E(\mathcal{C}')$.
\end{lemma}

\begin{proof}
    As $\mathcal{C}$ and $\mathcal{C}'$ are sets of arc-disjoint cycles, every vertex in $G[\mathcal{C}]$ (resp. in $G[\mathcal{C'}]$) has its in-degree equals to its out-degree. Since $E(\mathcal{C}') \subsetneq E(\mathcal{C})$, one can find $v_0 v_1 \in E(\mathcal{C}) \backslash E(\mathcal{C}')$. Then 
    \begin{displaymath}
        \degout_{G[\mathcal{C'}]}(v_1) = \degin_{G[\mathcal{C'}]}(v_1) \le \degin_{G[\mathcal{C}]}(v_1) - 1 = \degout_{G[\mathcal{C}]}(v_1) - 1
    \end{displaymath}
    (with the convention that $\degout_{G[\mathcal{C'}]}(v_1) = \degin_{G[\mathcal{C'}]}(v_1) = 0$ in case $v_1 \not \in V(G[\mathcal{C'}])$).
        Therefore, $v_1$ has at least one outgoing arc in $G[\mathcal{C}]$ that is not in $G[\mathcal{C'}]$. Thus, there exists a vertex $v_2$ such that $v_1 v_2 \in E(\mathcal{C}) \backslash E(\mathcal{C}')$. We continue the process until we find $v_{k-1} v_k \in E(\mathcal{C}) \backslash E(\mathcal{C}')$ with $v_k$ equal to some vertex $v_i$, with $1\leq i < k$. Then the cycle $C = v_i v_{i+1} \dots v_{k-1}$ is a cycle that uses only arcs in $E(\mathcal{C}) \backslash E(\mathcal{C}')$.        
\end{proof}

For a set $\mathcal{C}$ of arc-disjoint cycles of $G$, we define the quantity $q(\mathcal{C}) = m_\mathcal{C} - n_\mathcal{C}$.
Note that $q(\mathcal{C}) \ge 0$ and that $q(\mathcal{C}) = 0$ if and only if $\mathcal{C}$ is a set of vertex-disjoint cycles (or equivalently $G[\mathcal{C}]$ is a union of vertex-disjoint cycles).

\begin{lemma}\label{lem:finding_a_weaker_normal_set}
    Let $\mathcal{N}$ be a non-empty normal set of directed cycles of $G$. There exists a normal set of cycles $\mathcal{N'}$ of size $|\mathcal{N}|-1$ such that $E(\mathcal{N}') \subseteq E(\mathcal{N})$ and with $q(\mathcal{N}') \le \max(0,q(\mathcal{N}) - 1)$.
\end{lemma}

\begin{proof}
  The proof of the case $q(\mathcal{N}) = 0$ is straightforward as $\mathcal{N}$ is a set of vertex-disjoint cycles and removing one of them yields a normal set $\mathcal{N'}$ as required.
  
  Suppose now that $q(\mathcal{N}) > 0$.
  Let $\mathcal{N}_{lam}$ be a laminar normal set of cycles obtained from $\mathcal{N}$ by applying \Cref{prop:operation_lam}. As $E(\mathcal{N}_{lam}) = E(\mathcal{N})$ we have $q(\mathcal{N}_{lam}) = q(\mathcal{N}) > 0$, thus $\mathcal{N}_{lam}$ is not a set of vertex-disjoint cycles. Take $C \in \mathcal{N}_{lam} \backslash \{ \text{isolated cycles} \}$ that is minimal for the inclusion relation (i.e. there is no $C' \in \mathcal{N}_{lam}$ such that $\overline{C'} \subsetneq \overline{C}$ except isolated cycles). Then $\mathcal{N}' = \mathcal{N}_{lam} - C$ is still normal and $|\mathcal{N'}| = |\mathcal{N}| - 1$. Furthermore, as $C$ is not an isolated cycle we have $V(\{C \}) \cap V( \mathcal{N'}) \neq \emptyset$ then
    \begin{displaymath}
        q(\mathcal{N'}) = m_{\mathcal{N'}} - n_{\mathcal{N'}} = (m_{\mathcal{N}} - |C|) - (n_{\mathcal{N}} - |C| + |V(\{C \}) \cap V(\mathcal{N'})| ) < m_{\mathcal{N}} - n_{\mathcal{N}} = q(\mathcal{N}).
    \end{displaymath}
\end{proof}

\begin{proposition}\label{prop:fvs_of_G[N]}
    Let $\mathcal{N}$ be a normal set of cycles of $G$ of size $k = \fv(G)$ that minimizes $q(\mathcal{N})$. Then $\fv(G[\mathcal{N}]) = k$ if and only if $\mathcal{N}$ is a set of vertex-disjoint cycles.
\end{proposition}

\begin{proof}
    If $\mathcal{N}$ is a set of vertex-disjoint cycles, then $G[\mathcal{N}]$ is the union of $k$ vertex-disjoint cycles so $\fv(G[\mathcal{N}]) = k$.
    
    Reciprocally assume $\fv(G[\mathcal{N}]) = k$. Suppose by contradiction that $\mathcal{N}$ is not a set of vertex-disjoint cycles, so $q(\mathcal{N}) > 0$. Therefore, there exists a vertex $v$ such that $\degin_{G[\mathcal{N}]}(v) \ge 2$.
    
    By hypothesis $\fv(G[\mathcal{N}] - v) \ge k-1$ and by \Cref{thm:maximal_normal_set_greater_than_fvs} let $\mathcal{C}$ be a (normal) set of cycles of size $\ge k-1$ in $G[\mathcal{N}] - v$. Since $E(\mathcal{C}) \subsetneq E(\mathcal{N})$, we apply \Cref{lem:complete_arc_disjoint_set_of_cycles} to find a cycle $C$ that consists of arcs of $E(\mathcal{N}) \backslash E(\mathcal{C})$ and add it to $\mathcal{C}$. We iterate the operation while $E(\mathcal{C}) \subsetneq E(\mathcal{N})$. At the end of the process we obtain a set of arc-disjoint cycles $\mathcal{C}'$ such that $E(\mathcal{N}) = E(\mathcal{C'})$. It means that $G[\mathcal{N}] = G[\mathcal{C}']$ and since $\mathcal{N}$ is normal, all the vertices of $G[\mathcal{N}] = G[\mathcal{C}']$ are alternatingly oriented and thus $\mathcal{C}'$ is a normal set of cycles and $q(\mathcal{C}') = q(\mathcal{N}) > 0$. Moreover, since $v$ is not used in $\mathcal{C}$ and $\degin_{G[\mathcal{N}]}(v) \ge 2$, at least two cycles were added to $\mathcal{C'}$ using \Cref{lem:complete_arc_disjoint_set_of_cycles}. Thus $|\mathcal{C'}| \ge |\mathcal{C}| + 2 \ge k+1$. We apply \Cref{lem:finding_a_weaker_normal_set} $\lvert \mathcal{C'} \rvert - k$ times to $\mathcal{C'}$, in order to obtain a normal set $\mathcal{N}'$ of size $k$, such that $q(\mathcal{N'}) < q(\mathcal{N})$. This contradicts the minimality of $q(\mathcal{N})$.
\end{proof}

\begin{proposition}\label{prop:tightness_upper_bound}
    Let $G$ be a plane digraph of digirth $g \ge 3$ with $n \ge 3$ vertices. Then $\fv(G) = \frac{n-2}{g-2}$ if and only if $G = C_g$ the directed cycle of length $g$.
\end{proposition}

\begin{proof}
    For $G = C_g$ the equality is clear as $\fv(G) = 1$ and $n = g$.

    By \Cref{thm:maximal_normal_set_greater_than_fvs}, there exists a normal set of cycles of size $k = \fv(G)$. Take such a set $\mathcal{N}$ that minimizes $q(\mathcal{N})$.
    Suppose first that $q(\mathcal{N}) > 0$. Then by \Cref{prop:fvs_of_G[N]} we have $\fv(G[\mathcal{N}]) \neq k$ and in particular $G[\mathcal{N}] \neq G$.
    
    If $n_{\mathcal{N}} < n$ then the energy of internal vertices is at least one: $\mathbf{E_3}(\mathcal{N}) \ge 1$.
    
    If $n_{\mathcal{N}} = n$ and $m_{\mathcal{N}} < m$, then there exists an arc $uv \in E(G) \backslash E(G[\mathcal{N}])$. In $G$, $uv$ is a chord of a face $D$ of $G[\mathcal{N}]$. As $\mathcal{N}$ is normal, the boundary of the face $D$ forms a directed cycle. We note $P_{uv}$ (resp. $P_{vu}$) the path from $u$ to $v$ (resp. from $v$ to $u$) in the face $D$. Denote by $\ell_1$ the number of arcs of $P_{uv}$ and by $\ell_2$ the number of arcs of $P_{vu}$. Now observe that $P_{vu}$ together with the arc $uv$ form a cycle in $G$ then $\ell_2 + 1 \ge g$. Observe also that $\ell_1 \ge 2$ as $G$ has no parallel arcs. Finally $\ell_D = \ell_1 + \ell_2 \ge g+1$ and then the energy of faces is not null: $\mathbf{E_2}(\mathcal{N}) \ge \frac{1}{g}$.
    
    Both cases together with \Cref{prop:positive_energies} show that the total energy $\mathbf{E_{tot}}(\mathcal{N})$ is strictly positive and thus 
    \begin{displaymath}
        \fv(G) = k = |\mathcal{N}| = \frac{n-2}{g-2} - \frac{\mathbf{E_{tot}}(\mathcal{N})}{g-2} < \frac{n-2}{g-2}
    \end{displaymath}
    
    Suppose now that $q(\mathcal{N}) = 0$. Then, by \Cref{prop:fvs_of_G[N]} the cycles of the normal set $\mathcal{N}$ are pairwise vertex-disjoint and since  $|\mathcal{N}|=k$, we conclude that $\fv(G) = k \le \frac{n}{g}$. Therefore if $\fv(G) = \frac{n-2}{g-2} \le \frac{n}{g}$ then $n \le g$. On the other hand, as $n \ge 3$ and $\fv(G) = \frac{n-2}{g-2} > 0$, $G$ must contain at least one directed cycle which implies that $n \ge g$. Thus $n=g$ and $G = C_g$.
\end{proof}

\Cref{prop:tightness_upper_bound} together with \Cref{thm:maximal_normal_set_greater_than_fvs} complete the proof of \Cref{thm:upper_bound_FVS}.


\subsection{Valuations of cycles: proof of \Cref{thm:maximal_normal_set_greater_than_fvs}}\label{sec:proof_fvs_lower_than_maximal_normal_set}

In this section, we prove \Cref{thm:maximal_normal_set_greater_than_fvs}, which states that for every plane digraph $G$, $\fv(G) \le |\mathcal{N}_{\max}(G)|$.

Let $G$ be a plane digraph. A vertex $v \in V(G)$ is called \emph{essential} if it lies on a cycle of every maximum normal set of $G$. That is, in every normal set $\mathcal{N}$ of $G$ of maximal cardinality $\Modul{\mathcal{N}} = |\mathcal{N}_{\max}(G)|$, there exists at least one directed cycle that contains $v$. The main ingredient in the proof of \Cref{thm:maximal_normal_set_greater_than_fvs} is the following \namecref{lem:essential_vertex}, which we prove later in \Cref{sec:proof_essential_vertex}.

\begin{lemma}\label{lem:essential_vertex}
  Every directed cycle of $G$ contains at least one essential vertex.
\end{lemma}

\begin{proof}[Proof of \Cref{thm:maximal_normal_set_greater_than_fvs} assuming \Cref{lem:essential_vertex}]
  We proceed by induction on $|\mathcal{N}_{\max}(G)|$ to show that $\fv(G) \le |\mathcal{N}_{\max}(G)|$. If $|\mathcal{N}_{\max}(G)| = 0$, then $G$ has no directed cycle and $\fv(G) = 0$.

  Assume the result holds for every plane graph $G$ with $|\mathcal{N}_{\max}(G)| \le k$. Suppose that $|\mathcal{N}_{\max}(G)| = k+1$. Then $G$ contains a directed cycle, and by \Cref{lem:essential_vertex}, this cycle has an essential vertex $v_0$. Let $G' = G \backslash \lbrace v_0 \rbrace$. Since $v_0$ is essential, we have $|\mathcal{N}_{\max}(G')| < |\mathcal{N}_{\max}(G)| = k+1$ and thus $|\mathcal{N}_{\max}(G')| \le k$. Therefore, by the induction hypothesis $\fv(G') \le |\mathcal{N}_{\max}(G')| \le k$. Moreover, a feedback vertex set of $G$ can be obtained by adding $v_0$ to a feedback vertex set of $G'$, so:
  \begin{displaymath}
    \fv(G) \le \fv(G') + 1 \le k + 1 = |\mathcal{N}_{\max}(G)|.
  \end{displaymath}
\end{proof}
In the remainder of this section, we prove \Cref{lem:essential_vertex}.

\subsubsection{Valuations of cycles}

We introduce some definitions and remarks inspired by Lucchesi and Younger's work~\cite{LY78}.

\begin{definition}[Valuation]
  Let $\mathcal{C}$ be the set of the directed cycles of $G$. A \emph{valuation} of $G$ is a function $\mathcal{V} : \mathcal{C} \to \NN$.
  \begin{itemize}
  \item The \emph{weight} of a valuation $\mathcal{V}$ is defined as: $\displaystyle\# \mathcal{V} = \sum_{C \in \mathcal{C}} \mathcal{V}(C)$.
  \item We say that $\mathcal{V}$ \emph{uses a cycle} $C \in \mathcal{C}$ when $\mathcal{V}(C) > 0$. If $\mathcal{V}(C) = k$ we say that \emph{$\mathcal{V}$ uses $C$ $k$ times}.
  \item We say that $\mathcal{V}$ \emph{uses a vertex $v$ (resp.\ arc $e$) at least $k$ times} when $\sum_{C\ni v} \mathcal{V}(C) \ge k$ (resp. $\sum_{C\ni e} \mathcal{V}(C) \ge k$). Furthermore, if an arc $e$ is used $k$ times, then its two corresponding half-arcs are also used $k$ times.
  \item We denote by $\langle \mathcal{V}\rangle$ the set of cycles used by $\mathcal{V}$, that is: $\displaystyle       \langle \mathcal{V}\rangle = \lbrace C \in \mathcal{C} \mid \mathcal{V}(C) > 0 \rbrace$.
    
  \item We say that $\mathcal{V}$ is \emph{unitary} if $\underset{C \in \mathcal{C}}{\max}\{\mathcal{V}(C)\} = 1$.   
  \item We say that $\mathcal{V}$ is \emph{laminar} if two crossing cycles $C_1$ and $C_2$ are not both used in $\mathcal{V}$, that is: $\mathcal{V}(C_1) \cdot \mathcal{V}(C_2) = 0$.
    In other words $\mathcal{V}$ is a laminar valuation if $\langle \mathcal{V}\rangle$ is a laminar set of cycles.
  \item We say that $\mathcal{V}$ is \emph{normal} if $\mathcal{V}$ is unitary and $\langle\mathcal{V}\rangle$ is a normal set of cycles. Hence, there is a natural bijection between normal sets and normal valuations: a normal valuation is precisely the characteristic function (indicator function) of a normal set of cycles, and conversely, every normal set of cycles defines a unique normal valuation via this correspondence.
  \item For a set of cycles $\mathcal{S} \subseteq \mathcal{C}$, we can define the unitary valuation $\mathcal{F}$ \emph{associated to $S$} as:
    \begin{displaymath}
      \forall C \in \mathcal{C}, \quad \mathcal{F}(C) = \mathbb{1}_{C \in S} = \begin{cases}
      1 \text{ if $C \in S$} \\
      0 \text{ if $C \not \in S$}
      \end{cases}
    \end{displaymath}
  \end{itemize}
\end{definition}

Observe that if $\mathcal{V}$ is a normal valuation then $\langle \mathcal{V}\rangle$ is a normal set of $G$ of size $\# \mathcal{V}$.
Intuitively, a valuation generalizes the notion of a set of cycles where multiplicities indicate how many times each cycle is counted.

Given a half-arc $e$ incident to $v$, we set $\mu_v(e) = 1$ if $e$ is an outgoing half-arc of $v$ and $\mu_v(e) = -1$ if $e$ is an incoming half-arc of $v$. 

\begin{definition}
  Let $\mathcal{V}$ be a valuation of $G$ and $v$ a vertex of $G$.
  \begin{itemize}[label = $\bullet$]
  \item For a half-arc $e$, $\mathcal{V}^*(e)$ denotes the number of times $e$ is used in $\mathcal{V}$, which is:
      $\mathcal{V}^*(e) = \sum_{C \ni e} \mathcal{V}(C)$.
  
  \item A set $S$ of consecutive half-arcs around $v$ in the cyclic order is a \emph{segment of $v$} and we let $\segmentval{v}{\mathcal{V}}{S} = \sum_{e \in S} \mu_v(e) \times \mathcal{V}^*(e)$.
    
  \item Define the \emph{multiplicity of a vertex $v$} in a valuation
    $\mathcal{V}$ as $\mult_{\mathcal{V}}(v) = \underset{S \text{
        segment of $v$}}{\max}
    \left\{\Modul{\segmentval{v}{\mathcal{V}}{S}}\right\}$.
  \item Define the \emph{multiplicity of a valuation $\mathcal{V}$} as the maximal multiplicity among all vertices of $G$: $\mult(\mathcal{V}) = \underset{v \in V(G)}{\max} \{\mult_\mathcal{V}(v)\}$.
  \end{itemize}
\end{definition}

We provide an example in \Cref{fig:example_multiplicity}, where $\segmentval{v}{\mathcal{V}}{S_1} = 1 + 2 - 1 + 1 + 2 = 5$ and $\segmentval{v}{\mathcal{V}}{S_2} = 1 - 3 + 1 + 2 = 1$. The multiplicity of $v$ in this valuation is $\mult_{\mathcal{V}}(v) = 5$.

\begin{figure}[htbp]
  \centering
  \includegraphics[width=0.35\linewidth]{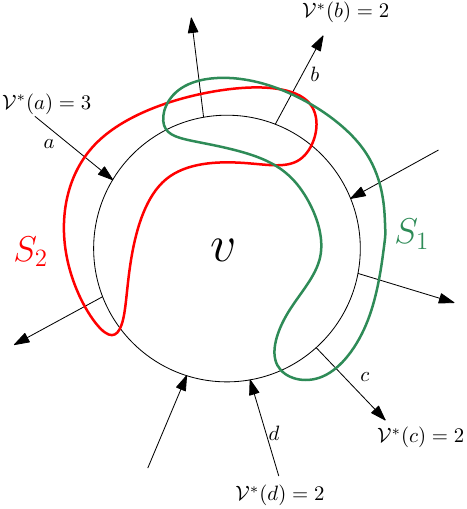}
  \caption{Example for computing the multiplicity of a vertex in some valuation $\mathcal{V}$.}\label{fig:example_multiplicity}
\end{figure}

\begin{remark}\label{rem:laminar_valuation} \ 
  \begin{itemize}
  \item A laminar valuation can use the same cycle multiple times since a cycle does not cross itself.
  \item Any valuation $\mathcal{V}$ can be transformed into a laminar valuation $\mathcal{V}'$ with $\# \mathcal{V}' = \# \mathcal{V}$ and $\mult(\mathcal{V}') = \mult(\mathcal{V})$. To construct
    $\mathcal{V}'$, we iterate the following process until $\langle \mathcal{V}\rangle$ does not contain any pair of crossing cycles:
    \begin{itemize}
    \item Take two crossing cycles $C_1, C_2 \in \langle \mathcal{V}\rangle$.
    \item Define $C_1', C_2'$ the cycles obtained from $C_1, C_2$ using the transformations of \Cref{prop:operation_lam}.
    \item We set $\mathcal{V} := \mathcal{V} + \mathbb{1}_{C_1'} + \mathbb{1}_{C_2'} - \mathbb{1}_{C_1} - \mathbb{1}_{C_2}$
    \end{itemize}

    Indeed, the transformations described in \Cref{prop:operation_lam} do not change the arcs used in the valuation. Then for any vertex $v$, we have $\segmentval{v}{\mathcal{V}'} = \segmentval{v}{\mathcal{V}}$.  
  \end{itemize}
\end{remark}

  \begin{proposition}\label{prop:multiplicity_normal_set}
    Let $\mathcal{V}, \mathcal{V}_1, \mathcal{V}_2$ be valuations of $G$.
    \begin{enumerate}
    \item\label{itm:mult1} $\mathcal{V}$ is normal if and only if $\mult(\mathcal{V}) = 1$.
    \item\label{itm:multK} Suppose $\mathcal{V}$ can be decomposed into $\ell$ normal valuations $\mathcal{N}_1, \ldots, \mathcal{N}_\ell$, i.e. $\mathcal{V} = \mathcal{N}_1 + \cdots + \mathcal{N}_\ell$. If for every $v \in V(G)$, $v$ is used in at most $k$ of these valuations, then $\mult(\mathcal{V}) \le k$.
    \item\label{itm:multSubAdditive} $\mult(\mathcal{V}_1 + \mathcal{V}_2) \le \mult(\mathcal{V}_1) + \mult(\mathcal{V}_2)$.
    \end{enumerate}
  \end{proposition}

  \begin{proof}
    \begin{enumerate}
    \item If $\mathcal{V}$ is normal, then by definition, the cycles of $\langle\mathcal{V}\rangle$ are arc-disjoint and each vertex $v$ in $G[\langle \mathcal{V}\rangle]$ is alternatingly oriented. Therefore, $\segmentval{v}{\mathcal{V}}{S} \in \lbrace -1,0,1 \rbrace$ for every segment $S$ of $v$. So $\mult_{\mathcal{V}}(v) = 1$ and thus $\mult(\mathcal{V}) = 1$.
      
      Conversely, if $\mathcal{V}$ is not normal, then either it is not unitary or $\langle \mathcal{V}\rangle$ is not arc-disjoint or there exists a vertex $v$ that is not alternatingly oriented in $G[\langle \mathcal{V}\rangle]$. In the first two cases, there exists a half-arc $e$ used at least twice in $\mathcal{V}$. Then, taking the segment $S = \lbrace e \rbrace$, we get $\Modul{\segmentval{v}{\mathcal{V}}{S}} = \mathcal{V}^*(e) \ge 2$ and so $\mult_{\mathcal{V}}(v) \ge 2$. In the latter case, there exist two consecutive half-arcs $e_1$ and $e_2$ of $v$ with the same orientation. Then, taking the segment $S = \lbrace e_1, e_2 \rbrace$, we have $\Modul{\segmentval{v}{\mathcal{V}}{S}} =  2$. So $\mult_{\mathcal{V}}(v) \ge 2$.
    \item Let $v \in V(G)$. By hypothesis, $v$ is used in at most $k$ normal valuations among ${(\mathcal{N}_i)}_{1 \le i \le \ell}$, say $\mathcal{N}_1, \ldots , \mathcal{N}_k$. Given $S$ a segment of $v$, we have:
      \begin{displaymath}
        \begin{split}
          \Modul{\segmentval{v}{\mathcal{V}}{S}} &= \Modul{\sum_{e \in S} \mu_v(e) \times \mathcal{V}^*(e)} = \Modul{\sum_{e \in S} \left(\mu_v(e) \times \sum_{i=1}^\ell \mathcal{N}_i^*(e)\right)} = \Modul{\sum_{e \in S} \left(\mu_v(e) \times \sum_{i=1}^k \mathcal{N}_i^*(e)\right)} \\
                                                &= \Modul{\sum_{i=1}^k \sum_{e \in S} \mu_v(e) \times \mathcal{N}_i^*(e)} = \Modul{\sum_{i=1}^k \segmentval{v}{\mathcal{N}_i}{S}} \le \sum_{i=1}^k \Modul{\segmentval{v}{\mathcal{N}_i}{S}} \le k \ \text{(since each $\mathcal{N}_i$ is normal).}
        \end{split}
      \end{displaymath}

      So $\mult_{\mathcal{V}}(v) \le k$ and thus $\mult(\mathcal{V}) \le k$.

    \item The proof works similarly to item\ref{itm:multK}.
    \end{enumerate}
  \end{proof}

\subsubsection{Proof of \Cref{lem:essential_vertex}}\label{sec:proof_essential_vertex}

We prove \Cref{lem:essential_vertex} by contradiction. Suppose $G$ is a counterexample to the statement, that is, there exists a directed cycle in $G$ that contains no essential vertex. We first construct a laminar valuation $\mathcal{V}$ of $G$ with large weight and small multiplicity.

\begin{lemma}\label{lem:existence_of_big_laminar_set_with_low_mult}
  There exists an integer $k\ge 1$ and a valuation $\mathcal{V}$ of $G$ such that:
  \begin{enumerate}
  \item $\mathcal{V}$ is laminar,
  \item $\#\mathcal{V} = k \times  |\mathcal{N}_{\max}(G)| + 1$,
  \item $\mult(\mathcal{V}) \le k$.
  \end{enumerate}
\end{lemma}

\begin{proof} 
  Since $G$ is a counterexample to \Cref{lem:essential_vertex}, there exists a directed cycle $D$ in $G$ that contains no essential vertex. Set $k = \Modul{D}$ and $D = v_1 v_2 \ldots v_k$. For each $i \in \IntSet{1}{k}$, since $v_i$ is not essential, there exists a maximum normal set $N_i$ that does not use $v_i$. Also set $N_{k+1} = \lbrace D \rbrace$, which is a normal set. Set ${(\mathcal{N}_i)}_{1 \le i \le k+1}$ the normal valuations associated with these normal sets. Let $\mathcal{W} = \mathcal{N}_1 + \cdots + \mathcal{N}_{k+1}$. Then:
  \begin{itemize}
  \item One can construct a laminar valuation $\mathcal{V}$ from $\mathcal{W}$ by \Cref{rem:laminar_valuation} such that $\# \mathcal{V} = \# \mathcal{W}$ and $\mult(\mathcal{V}) = \mult(\mathcal{W})$.
  \item $\# \mathcal{V} = \# \mathcal{W} = \sum_{i=1}^{k+1} \# \mathcal{N}_i = k \times |\mathcal{N}_{\max}(G)| + 1$.
  \item Each vertex $v\in V(G)$ is used in at most $k$ valuations $\mathcal{N}_i$. If $v$ is not a vertex of $D$ then $v$ is not used in $\mathcal{N}_{k+1}$, otherwise $v = v_i$ for some $i \in \IntSet{1}{k}$  and then $v$ is not used in $\mathcal{N}_i$. Thus we have $\mult(\mathcal{V}) = \mult(\mathcal{W}) \le k$ by \Cref{prop:multiplicity_normal_set}\ref{itm:multK}.
  \end{itemize}
\end{proof}

\begin{lemma}\label{lem:multiplicity_reduction}
  Let $\mathcal{V}$ be a laminar valuation such that
  $\begin{cases}
    \# \mathcal{V} \ge k \times  |\mathcal{N}_{\max}(G)| + 1 \\
    2 \le \mult(\mathcal{V}) \le k
  \end{cases}$ for some integer $k \ge 2$.
  Then one can construct another laminar valuation $\mathcal{V'}$ such that
  $\begin{cases}
    \# \mathcal{V'} \ge \Ent{\frac{k+1}{2}} \times  |\mathcal{N}_{\max}(G)| + 1 \\
    \mult(\mathcal{V'}) \le \Ent{\frac{k+1}{2}}
  \end{cases}$.
\end{lemma}

\begin{proof}[Proof of \Cref{lem:essential_vertex} assuming \Cref{lem:multiplicity_reduction}]
    Take $(\mathcal{V},k)$ the valuation and the integer given by \Cref{lem:existence_of_big_laminar_set_with_low_mult}.
    Let $k_0 = k$ and $\mathcal{V}_0 = \mathcal{V}$. If $k_0 \ge 2$, we apply \Cref{lem:multiplicity_reduction} to $\mathcal{V}_0$ to obtain a laminar valuation $\mathcal{V}_1$ such that $\begin{cases}
        \# \mathcal{V}_1 \ge k_1 \times |\mathcal{N}_{\max}(G)| + 1 \\
        \mult(\mathcal{V}_1) \le k_1
    \end{cases}$ with $k_1 = \left\lfloor \frac{k_0 + 1}{2} \right\rfloor$.
    
      We iterate the process as long as $k_i \ge 2$ and it clearly terminates.
      At the end, we obtain a laminar set $\mathcal{V}'$ such that:
      \begin{displaymath}
        \#\mathcal{V}' \ge |\mathcal{N}_{\max}(G)| + 1 \qquad \text{and} \qquad \mult(\mathcal{V}') = 1.
      \end{displaymath}
      According to \Cref{prop:multiplicity_normal_set}\ref{itm:mult1}, we deduce that $\mathcal{V}'$ is a normal valuation. Thus $\langle \mathcal{V}' \rangle$ is a normal set of cycles of size at least $|\mathcal{N}_{\max}(G)| + 1$.
      A contradiction.
\end{proof}

It remains to show \Cref{lem:multiplicity_reduction}.

\subsubsection{Proof of \Cref{lem:multiplicity_reduction}}

To prove \Cref{lem:multiplicity_reduction}, we consider two cases based on the parity of $k$.

\paragraph{Case 1:  $k=2p$ is even.}\ 

For a valuation $\mathcal{V}$, we define $\mathcal{C}_\mathcal{V}$ as the multiset of cycles used by $\mathcal{V}$, where each cycle $C$ appears exactly $\mathcal{V}(C)$ times:
\begin{displaymath}
  \mathcal{C}_\mathcal{V} = \bigcup_{C \in \langle \mathcal{V}\rangle} \lbrace C^{(1)}, \dots , C^{(\mathcal{V}(C))} \rbrace \qquad \text{where } C^{(1)}, \dots , C^{(\mathcal{V}(C))} \text{ are copies of $C$}.
\end{displaymath}
Suppose $\mathcal{V}$ is a valuation satisfying the hypothesis of \Cref{lem:multiplicity_reduction}. Since $\mathcal{V}$ is laminar, we can define a forest structure $\mathcal{T}_{\mathcal{V}}$ induced by the inclusion relation on the cycles of $\mathcal{C}_\mathcal{V}$. This forest has vertices corresponding to the cycles in $\mathcal{C}_\mathcal{V}$, and a vertex $A$ is the parent of a vertex $B$ if either:
\begin{itemize}
\item there exists a cycle $C \in \langle \mathcal{V}\rangle$ such that $A = C^{(i)}$ and $B = C^{(i+1)}$, for some $1 \le i < \mathcal{V}(C)$, 
\item or $A = C_1^{(\mathcal{V}(C_1))}$ and $B = C_2^{(1)}$ for some $C_1, C_2 \in \langle\mathcal{V}\rangle$ with $\overline{C_2} \subset \overline{C_1}$ and there is no other cycle $C' \in \langle\mathcal{V}\rangle$ such that $\overline{C_2} \subset \overline{C'} \subset \overline{C_1}$. 
\end{itemize}

The roots of this forest are the maximal elements for the inclusion relation (i.e. the roots correspond to the cycles that are not included in any other cycle, such cycles being $C^{(1)}$ for some $C \in \langle \mathcal{V}\rangle$), and we denote them by $\Root(\mathcal{T}_{\mathcal{V}})$. Moreover, we denote by $\Ch(C)$ the set of children of $C$ in $\mathcal{T}_{\mathcal{V}}$.

  To each cycle $C \in \mathcal{C}_\mathcal{V}$, we associate its layer $\ell(C)$, which is the distance between $C$ and the root of the tree to which it belongs. Thus, if $C_2$ is the child of $C_1$, then $\ell(C_2) = \ell(C_1) + 1$. See \Cref{fig:example_layer} for an example.

  \begin{figure}[htbp]
    \centering
    \includegraphics[width=\linewidth]{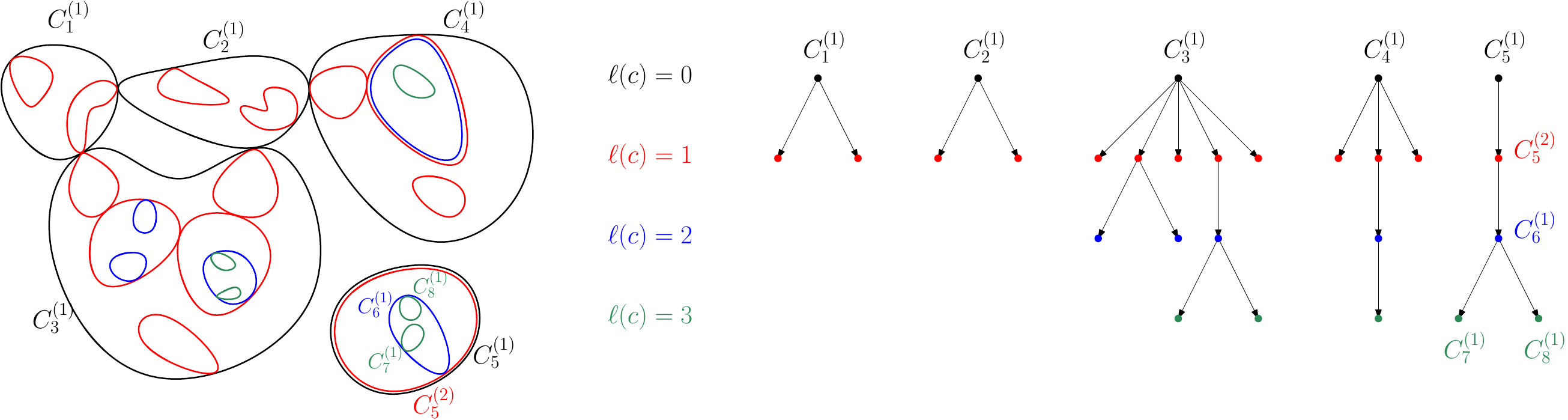}
    \caption{A laminar valuation $\mathcal{V}$ with $\Root(\mathcal{T}_{\mathcal{V}}) = \lbrace C_1^{(1)}, C_2^{(1)}, C_3^{(1)}, C_4^{(1)}, C_5^{(1)} \rbrace$ and its inclusion forest $\mathcal{T}_{\mathcal{V}}$.}\label{fig:example_layer}
  \end{figure}

  Given $par\in\{\text{even}, \text{odd}\}$ and $ori \in\{\text{clockwise}, \text{counterclockwise}\}$, we define the following multiset of cycles:
  $\mathcal{C}_{\mathcal{V}}(par, ori) = \Set{C \in \mathcal{C}_\mathcal{V}}{\ell(C) \text{ has parity } par \text{ and $C$ is oriented } ori}$. We partition the multiset $\mathcal{C}_\mathcal{V}$ into two multisets according to the parity of the layer of the cycles and their orientation:

    \begin{displaymath}
    \begin{split}
      \mathcal{C}_{\mathcal{V}}^1 &=\mathcal{C}_{\mathcal{V}}(\text{even}, \text{clockwise}) \cup \mathcal{C}_{\mathcal{V}}(\text{odd}, \text{counterclockwise}) \\
      \mathcal{C}_{\mathcal{V}}^2 &=\mathcal{C}_{\mathcal{V}}(\text{odd}, \text{clockwise}) \cup \mathcal{C}_{\mathcal{V}}(\text{even}, \text{counterclockwise})
    \end{split}
  \end{displaymath}

  Then we define the valuation $\mathcal{V}_1$ as the valuation induced by the multiset $\mathcal{C}_{\mathcal{V}}^1$: $ \mathcal{V}_1 = \sum_{C \in \mathcal{C}_{\mathcal{V}}^1} \mathbb{1}_{C}$.
  Similarly, we define the valuation $\mathcal{V}_2$ induced by the multiset $\mathcal{C}_{\mathcal{V}}^2$. As $\mathcal{C}_{\mathcal{V}}^1$ and $\mathcal{C}_{\mathcal{V}}^2$ form a partition of $\mathcal{C}_{\mathcal{V}}$, we have $\mathcal{V} = \mathcal{V}_1 + \mathcal{V}_2$.

  We will show that $\mathcal{V}_1$ and $\mathcal{V}_2$ are valuations of multiplicity at most $\frac{k}{2}$.

Whenever an arc $e$ is used by several cycles in $\mathcal{V}$, we replace $e$ by as many parallel arcs as necessary so that each copy of $e$ is used by only one cycle. We choose which cycle to associate with which copy of $e$ such that the laminarity of the set is preserved (note that this is always possible following the embedding of $G$). We denote by $G_{\mathcal{V}}$ the graph obtained after this transformation and delete from  $G_{\mathcal{V}}$ the arcs not used by $\mathcal{V}$. Then note that for $v \in V(G_{\mathcal{V}})$ and $S$ a segment of $v$ in $G_{\mathcal{V}}$, the following holds:
  \begin{displaymath}
    \segmentval{v}{\mathcal{V}}{S} = \sum_{e \in S} \mu_v(e) \qquad \text{since each half-arc is used exactly once in $\mathcal{V}$.}
  \end{displaymath}

  Furthermore, $\mathcal{V} = \mathcal{V}_1 + \mathcal{V}_2$ defines a bipartition of the arcs of $G_{\mathcal{V}}$ depending whether the arc is used in $\mathcal{V}_1$ or in $\mathcal{V}_2$. The graph induced by the arcs of $\mathcal{V}_1$ (resp. $\mathcal{V}_2$) is denoted by $G_{\mathcal{V}_1}$ (resp. $G_{\mathcal{V}_2}$).

  \begin{lemma}\label{lem:consecutive_half_arcs}
    Let $v$ be a vertex of $G_{\mathcal{V}}$ and $e, e'$ two consecutive half-arcs incident to $v$ in $G_{\mathcal{V}_1}$ (or in $G_{\mathcal{V}_2}$). Let $S = ]e, e'[$ be the segment of $v$ in $G_{\mathcal{V}}$ between these two half-arcs ($e$ and $e'$ excluded). Then 
    \begin{displaymath}
      \segmentval{v}{\mathcal{V}}{S} = \frac{\mu_v(e) + \mu_v(e')}{2}.
    \end{displaymath}
  \end{lemma}

  \begin{proof}
    Let $C$ and $C'$ be the cycles in $\mathcal{C}_\mathcal{V}$ that use $e$ and $e'$ respectively. Assume without loss of generality that $C$ and $C'$ are both in $\mathcal{C}_{\mathcal{V}}^1$.
    
    Suppose first that $S = \emptyset$. Then either $\ell(C) = \ell(C')$ and $C$ and $C'$ have the same orientation or $|\ell(C) - \ell(C')| = 1$ and $C$ and $C'$ have opposite orientations (see \Cref{fig:consecutive_half_arcs}). In both cases $e$ and $e'$ have opposite directions. Thus $\frac{\mu_v(e) + \mu_v(e')}{2} = 0 = \segmentval{v}{\mathcal{V}}(S)$.
    
        \begin{figure}[htbp]
      \centering
      \begin{subfigure}[b]{0.45\linewidth}   
        \centering
        \includegraphics[width=0.8\textwidth]{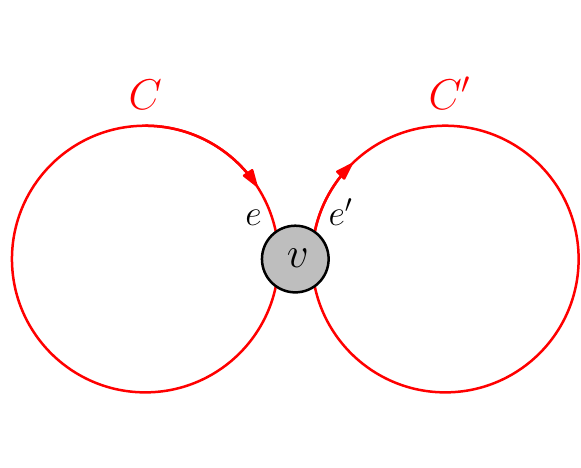}
        \caption{Case $\ell(C) = \ell(C')$ and $C$ and $C'$ have the same orientation.}
      \end{subfigure}
      \hfill
      \begin{subfigure}[b]{0.45\linewidth}   
        \centering
        \includegraphics[width=0.6\textwidth]{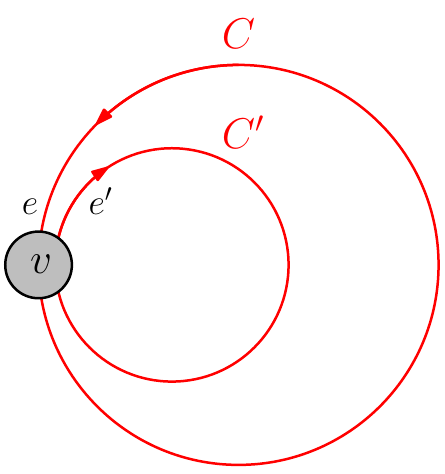}
        \caption{Case $|\ell(C) - \ell(C')| = 1$ and $C$ and $C'$ have opposite orientations.}
      \end{subfigure}

      \caption{Two half-arcs incident to $v$ that are consecutive in $G_{\mathcal{V}}$ (so $]e,e'[ =\emptyset$) and that both belong to $G_{\mathcal{V}_1}$ (or $G_{\mathcal{V}_2}$) have opposite directions.\label{fig:consecutive_half_arcs}}
        
    \end{figure}

    Suppose now that $S = \lbrace e_1, \dots, e_k \rbrace$ with $k \ge 1$. Let $C_1$ (respectively $C_k$) be the cycle that uses $e_1$ (respectively $e_k$). Since $e$ and $e'$ are consecutive in $G_{\mathcal{V}_1}$, then $C_1 \in \mathcal{C}_{\mathcal{V}}^2$ (respectively $C_k \in \mathcal{C}_{\mathcal{V}}^2$). Then either $\ell(C) = \ell(C_1)$ and $C$ and $C_1$ have opposite orientations or $|\ell(C) - \ell(C_1)| = 1$ and $C$ and $C_1$ have the same orientation. In both cases $e$ and $e_1$ have the same direction (either ingoing or outgoing). Thus $\mu_v(e) = \mu_v(e_1)$. Similarly we have $\mu_v(e_k) = \mu_v(e')$. Then:
    \begin{displaymath}
        \segmentval{v}{\mathcal{V}}{S} = \sum_{i = 1}^k \mu_v(e_i) = \frac{\mu_v(e_1)}{2} + \sum_{i = 1}^{k-1} \frac{\mu_v(e_i) + \mu_v(e_{i+1})}{2} + \frac{\mu_v(e_k)}{2} = \frac{\mu_v(e) + \mu_v(e')}{2} + \sum_{i = 1}^{k-1} \frac{\mu_v(e_i) + \mu_v(e_{i+1})}{2}.
    \end{displaymath}
    
    Recall that for $1 \le i \le k-1$, $e_i$ and $e_{i+1}$ are two consecutive half-arcs incident to $v$ in $G_{\mathcal{V}_2}$ such that $]e_i, e_{i+1}[ = \emptyset$ otherwise there would exist a half-arc $e'' \in ]e_i, e_{i+1}[ \subset ]e,e'[$ used in $G_{\mathcal{V}_1}$ which contradicts the fact that $e$ and $e'$ are consecutive in $G_{\mathcal{V}_1}$. Then from the case $S = \emptyset$ it comes that $\frac{\mu_v(e_i) + \mu_v(e_{i+1})}{2} = \segmentval{v}{\mathcal{V}}(]e_i, e_{i+1}[) = 0$. Thus
    \begin{displaymath}
        \segmentval{v}{\mathcal{V}}{S} = \frac{\mu_v(e) + \mu_v(e')}{2} + \sum_{i = 1}^{k-1} \frac{\mu_v(e_i) + \mu_v(e_{i+1})}{2} = \frac{\mu_v(e) + \mu_v(e')}{2}
    \end{displaymath}
\end{proof}

  \begin{lemma}\label{lem:multiplicity_reduced_valuations}    
   We have $\mult(\mathcal{V}_1) \le \frac{k}{2} = p$ and $\mult(\mathcal{V}_2) \le \frac{k}{2} = p$.
  \end{lemma}

  \begin{proof}
    By contradiction, let $v$ be a vertex of $G_{\mathcal{V}}$ such that $\mult_{\mathcal{V}_1}(v) \ge p + 1$.
    Then there exists a segment $S=\{e_1,\ldots,e_j\}$ of $v$ in $G_{\mathcal{V}_1}$ such that $\Modul{\segmentval{v}{\mathcal{V}_1}(S)} \ge p+1$.
    For $i \in \{1,\ldots,j-1\}$, denote $]e_i, e_{i+1}[$ the segment between $e_i$ and $e_{i+1}$ (excluded) seen as a segment of $v$ in $G_{\mathcal{V}}$.
    By \Cref{lem:consecutive_half_arcs}, we have $\segmentval{v}{\mathcal{V}}(]e_i, e_{i+1}[) = \frac{\mu_v(e_i) + \mu_v(e_{i+1})}{2}$. Thus:
    \begin{displaymath}
      \begin{split}
        \segmentval{v}{\mathcal{V}}(S) &= \sum_{i=1}^j \mu_v(e_i) + \sum_{i=1}^{j-1} \segmentval{v}{\mathcal{V}}(]e_i, e_{i+1}[) = \sum_{i=1}^j \mu_v(e_i) +\sum_{i=1}^{j-1} \frac{\mu_v(e_i) + \mu_v(e_{i+1})}{2} \\
                             &= 2 \sum_{i=1}^j \mu_v(e_i) - \left( \frac{\mu_v(e_1) + \mu_v(e_j)}{2} \right) = 2 \segmentval{v}{\mathcal{V}_1}(S) - \left( \frac{\mu_v(e_1) + \mu_v(e_j)}{2} \right) \\
      \end{split}
    \end{displaymath}
    So by the second triangle inequality, we have
    \begin{displaymath}
      \Modul{\segmentval{v}{\mathcal{V}}(S)} \ge 2 \Modul{\segmentval{v}{\mathcal{V}_1}(S)} - \frac{\Modul{\mu_v(e_1) + \mu_v(e_j)}}{2} \ge 2(p+1) - 1 = 2p + 1
    \end{displaymath}
    This is a contradiction since $\mult(\mathcal{V}) \le k = 2p$. So $\mult(\mathcal{V}_1) \le p$ and similarly $\mult(\mathcal{V}_2) \le p$.
  \end{proof}

  Since $\mathcal{V} = \mathcal{V}_1 + \mathcal{V}_2$, we have $\# \mathcal{V}_1 + \# \mathcal{V}_2 = \# \mathcal{V} \ge 2p \times |\mathcal{N}_{\max}(G)| + 1$.
  Hence, at least one of the two valuations $\mathcal{V}_1$ and $\mathcal{V}_2$, say $\mathcal{V}_1$, has weight at least $p \times |\mathcal{N}_{\max}(G)| + 1$. Then $\mathcal{V}_1$ is laminar as $\langle \mathcal{V}_1 \rangle \subset \langle\mathcal{V}\rangle$ and $\langle \mathcal{V}\rangle$ is laminar. Moreover:
  \begin{itemize}[label = $\bullet$]
  \item $\# \mathcal{V}_1 \ge p\times |\mathcal{N}_{\max}(G)| + 1 = \Ent{\frac{k+1}{2}} \times |\mathcal{N}_{\max}(G)| + 1$
  \item $\mult(\mathcal{V}_1) \le p = \frac{k}{2} = \Ent{\frac{k+1}{2}}$ by \Cref{lem:multiplicity_reduced_valuations}.
  \end{itemize}

  \paragraph{Case 2: $k=2p+1$ is odd.}\ 

  Consider a maximum normal set of $G$ and $\mathcal{N}$ its associated normal valuation.
  By \Cref{rem:laminar_valuation}, one can obtain a laminar valuation $\tilde{\mathcal{V}}$ from $\mathcal{V} + \mathcal{N}$ such that $\#\tilde{\mathcal{V}} = \#(\mathcal{V} + \mathcal{N})$ and $\mult(\tilde{\mathcal{V}}) = \mult(\mathcal{V} + \mathcal{N})$.
  
  We have:
  \begin{displaymath}
    \#\tilde{\mathcal{V}} = \#(\mathcal{V} + \mathcal{N}) = \# \mathcal{V} + \# \mathcal{N} \ge (2p + 1) \times |\mathcal{N}_{\max}(G)| + 1 + |\mathcal{N}_{\max}(G)| = 2(p+1) \times |\mathcal{N}_{\max}(G)| + 1.
  \end{displaymath}
  Moreover, from \Cref{prop:multiplicity_normal_set}\ref{itm:multSubAdditive}, we have 
  \begin{displaymath}
    \mult(\tilde{\mathcal{V}}) = \mult(\mathcal{V}+ \mathcal{N}) \le \mult(\mathcal{V}) + \mult(\mathcal{N}) \le (2p + 1) + 1 = 2(p+1)
  \end{displaymath}
  We can therefore apply the reasoning of the even case to $\tilde{\mathcal{V}}$ to obtain a laminar valuation $\mathcal{V}'$ such that:
  \begin{itemize}[label = $\bullet$]
  \item $\# \mathcal{V'} \ge (p+1) \times |\mathcal{N}_{\max}(G)| + 1 = \Ent{\frac{k+1}{2}} \times |\mathcal{N}_{\max}(G)| + 1$
  \item $\mult(\mathcal{V}') \le p+1 = \Ent{\frac{k+1}{2}}$.
  \end{itemize}

\clearpage

\section{Lower bounds: skeletons and coatings}\label{sec:skeleton_and_coating}

In this section, we aim to construct planar digraphs of digirth $g\geq 4$ with a large feedback vertex set relative to their order $n$.

\subsection{An introductory example}\label{sec:introductory_example_coating}
We start by showing a new family ${(G_{k})}_{k \ge 1}$ of planar digraphs of digirth $g\ge 4$ and order $n_k$ such that $\fv(G_k) = \frac{n_k - 1}{g-1}$, which is the same ratio as that of the family obtained in~\cite{KVW17}. See~\Cref{fig:frieze} for an illustration.

\begin{figure}[htbp]
    \centering
    \includegraphics[height=2.8cm]{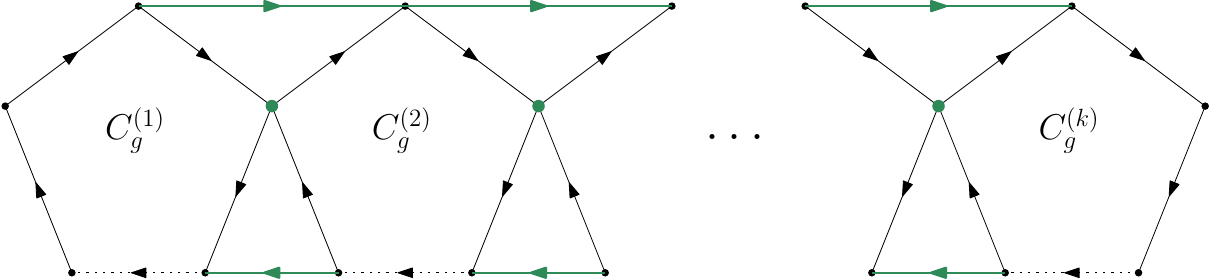}
    \caption{A digraph $G_k$ of digirth $g\geq 4$, satisfying $\fv(G_k) = \frac{n_k - 1}{g-1}$}\label{fig:frieze}
\end{figure}

Begin with $G_1$, which is a directed cycle $C_g^{(1)}=v^1_0v^1_1\ldots v^1_{g-1}$ of length $g$ oriented clockwise (with respect to a fixed planar embedding). For $k \ge 2$, construct $G_k$ from $G_{k-1}$ as follows:
\begin{itemize}
    \item Add a new directed cycle $C_g^{(k)}=v^k_0v^k_1\ldots v^k_{g-1}$ of length $g$ oriented clockwise.
    \item Identify vertices $v^k_{g-1}$ of $C_g^{(k)}$ with $v_1^{k-1}$ of $G_{k-1}$. Call the newly created vertex $v^k_{g-1}$ a \emph{link vertex between $G_{k-1}$ and $C_g^{(k)}$}.
    \item Add two arcs $v^{k-1}_0v^{k}_0$ and $v^{k}_{g-2}v^{k-1}_2$
    which are called the \emph{link arcs between $G_{k-1}$ and $C_g^{(k)}$}.
\end{itemize}

\begin{proposition}\label{prop:frieze}
    For every $k \ge 1$, $G_k$ has digirth $g$, order $n_k= k(g - 1) + 1$ and 
        $\fv(G_k)= k = \frac{n_k - 1}{g - 1}$.
\end{proposition}

\begin{figure}[htbp]
    \centering
    \includegraphics[scale=0.85]{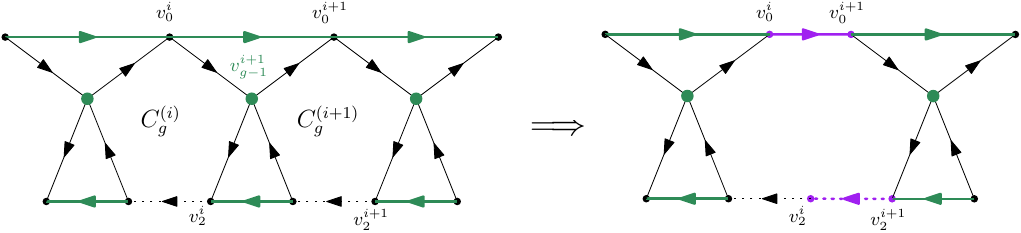}
    \caption{Illustration of the proof of \Cref{prop:frieze}. The purple elements are those contracted to obtain $G_k$ from $G_{k+1} - \{ v^{i+1}_{g-1} \}$.}\label{fig:proof_frieze}
\end{figure}

\begin{proof}
    First note that $\{v^1_0,v^2_0,\ldots,v^k_0\}$ is a feedback vertex set of size $k$ proving that $\fv(G_k) \le k$.

    To show $\fv(G_k)\ge k$, proceed by induction on $k$. 
    For $k = 1$, $G_1$ is a directed cycle of length $g$, so $\fv(G_1) = 1$. Let $F$ be a minimum feedback vertex set of $G_{k+1}$. Distinguish two cases.
    
    If $F$ contains no link vertex, then $F$ must hit each cycle $C_g^{(i)}$ for $i \in \llbracket 1, k+1 \rrbracket$ with $k+1$ distinct vertices and we are done. Now suppose $F$ contains a link vertex $v^{i+1}_{g-1}$ for some $i\in\{1,\ldots,k\}$. Then we apply the induction hypothesis on the digraph obtained from $G_{k+1}$ by removing $v_{g-1}^{i+1}$ and contracting the arc $v_0^i v_0^{i+1}$ and all the arcs between $v_{2}^{i+1}$ and $v_2^i$ (see \Cref{fig:proof_frieze}). Note that the graph obtained is isomorphic to $G_k$. Since $F\setminus \{v^{i+1}_{g-1}\}$ is a feedback vertex set of this newly created digraph, by induction hypothesis it must have size at least $k$, so we are done.
\end{proof}

In the next sections, we generalize this construction in order to increase the ratio $\frac {\fv(G)}n$.

\subsection{Skeletons and coatings}

\begin{definition}[Coating and Skeleton]\label{def:coating}
    Let $G$ be a plane undirected graph.
    \begin{itemize}
        \item A \emph{coating} of $G$ is a plane digraph $H$ constructed from $G$ as follows. To each vertex $v$ of $G$ we associate a directed cycle $C_v$ oriented clockwise. This cycle is said to be \emph{associated to vertex $v$ of $G$}. For every edge $uv$ in $G$, the two cycles $C_u$ and $C_v$ of $H$ are connected by merging a vertex $s_u$ of degree 2 of $C_u$ with a vertex $s_v$ of degree 2 of $C_v$  into a vertex $s_{uv}$ and by adding two extra arcs between $C_u$ and $C_v$: one from the in-neighbor of $s_u$ to the out-neighbor of $s_v$ and one from the in-neighbor of $s_v$ to the out-neighbor of $s_u$ (see \Cref{fig:transfo_coating}). The vertex $s_{uv}$ is called \emph{link vertex associated to the edge $uv$} and the two arcs added on both sides of the link vertex are called \emph{link arcs associated to the edge $uv$}. In case of a loop at $u$ in $G$, the two vertices of degree 2 to be merged are taken in the same cycle $C_u$ and are at distance at least 3 so that the link arcs are never loops.
        \begin{figure}[htbp]
            \centering
            \begin{subfigure}[b]{0.3\textwidth}
                \centering
                \includegraphics[scale=1]{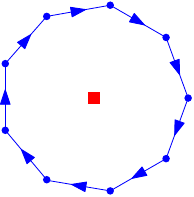}
                \caption{A vertex of $G$.}
            \end{subfigure}
            \begin{subfigure}[b]{0.68\textwidth}
                \centering
                \includegraphics[scale=1]{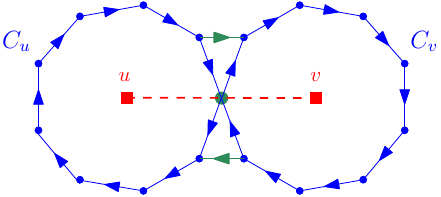}
                \caption{An edge of $G$ with its associated link vertex and link arcs (in green).}
            \end{subfigure}
            \caption{Building a coating $H$ (in blue and green) of a skeleton graph $G$ (in red).}\label{fig:transfo_coating}
        \end{figure}
        
        \item If $H$ is a coating of $G$, we say that $G$ is \emph{the skeleton} of $H$ and by the previous item of the definition it is uniquely defined as a plane graph.
        \item If $f$ is a face of $G$ whose incident vertices are in the same connected component of $G$, the vertices of $H$ that lie inside $f$ (excluding link vertices) form a cycle $C_f$ oriented counterclockwise (or clockwise if $f = f_{ext}$). This cycle is said to be \emph{associated to the face $f$ of $G$}. In case $f$ is a face incident to $k$ connected components in $G$, then the vertices of $H$ that lie inside $f$ (excluding link vertices) form a union of $k$ vertex disjoint directed cycles $C_f^1, \dots , C_f^k$ \emph{associated to the face $f$ of $G$}. In this case by abusing the notation we will denote $\lvert C_f \rvert = \sum_{i=1}^k \lvert C_f^i \rvert$ the number of edges lying inside $f$. See \Cref{fig:example_coating_bis}~(left) for an example of a skeleton with four faces.
        \item If every vertex of $G$ is transformed into a directed cycle with a fixed length $g \ge 1$, we say that $H$ is a \emph{$g$-coating} of $G$.
    \end{itemize}
\end{definition}

Note that the graph $G_k$ of \Cref{fig:frieze} is a $g$-coating of the undirected path on $k$ vertices. However, not every $g$-coating has digirth $g$ (see \Cref{sec:digirth_coating}). Finally, observe that every undirected plane graph $G$ admits a $g$-coating for a sufficiently large $g$.

\begin{figure}[htbp]
    \centering
    \includegraphics[height=5cm]{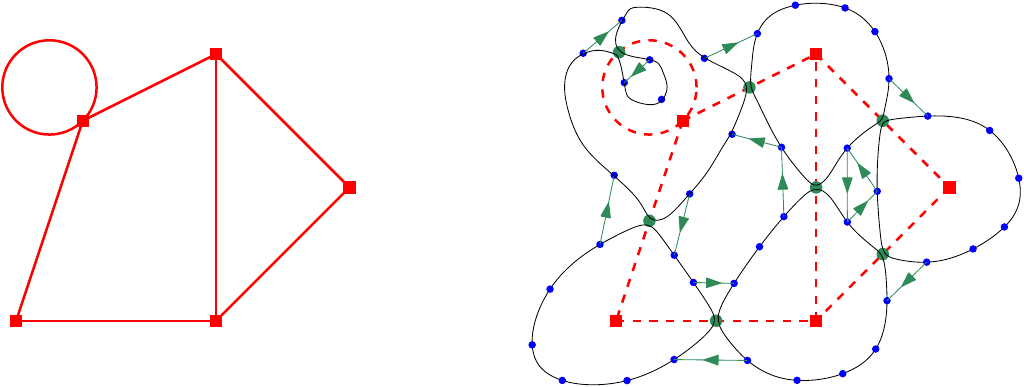}
    \caption{Left: an undirected graph $G$. Right: a coating $H$ of $G$. The black cycles associated to the vertices of $G$ are oriented clockwise. The link vertices and arcs are represented in green.}\label{fig:example_coating_bis}
\end{figure}

Note that coatings depend on the embedding of $G$, and two different embeddings of $G$ do not necessarily yield the same coatings.

\begin{observation}\label{prop:coating_subgraph}
    Let $G$ be a skeleton that admits a $g$-coating of digirth $g$ and let $G' \subset G$ be a subgraph of $G$. Then $G'$ also admits a $g$-coating of digirth $g$.
\end{observation}

\begin{figure}[htbp]
        \centering
        \subfloat[Removal of the edge $uv$ of the skeleton: the link vertex $s$ is split into two vertices $s_1$ and $s_2$ and the link arcs are removed.]{\label{subfig:proof_coating_edge_deletion}
            \centering
            \includegraphics[height=4cm]{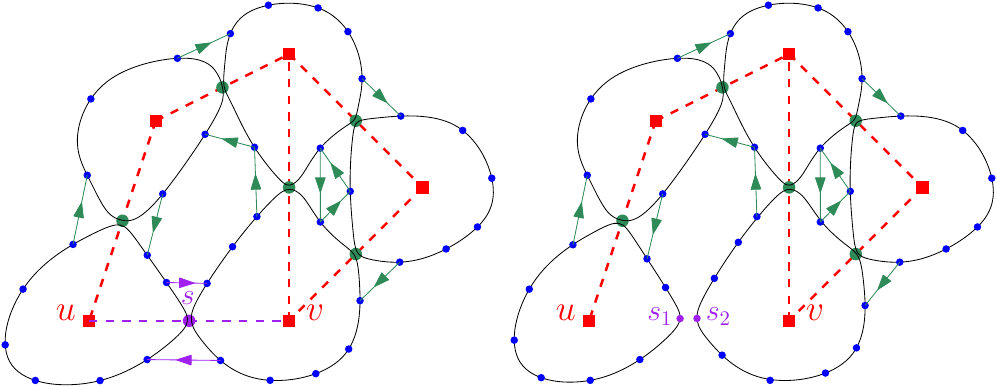}
        }
        \hspace{1cm}
        \subfloat[Removal of the vertex $v$ of the skeleton: all vertices of the coating that are in $C_v$ and that are not link vertices are deleted.]{\label{subfig:proof_coating_vertex_deletion}  
            \centering
            \includegraphics[height=4cm]{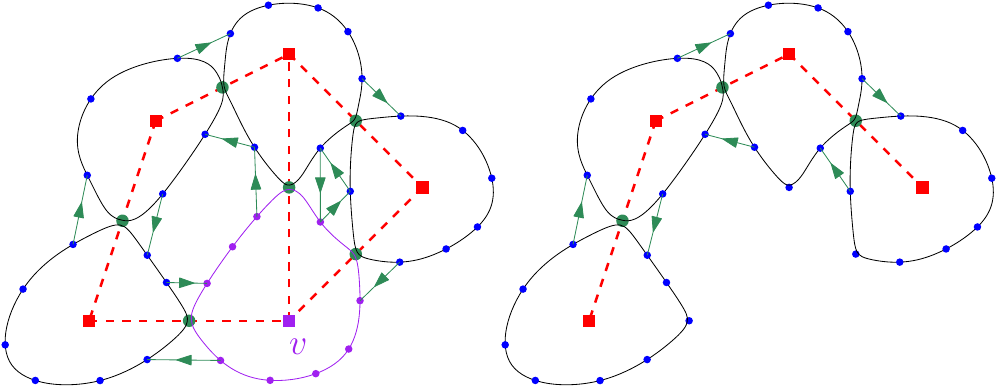}
        }
        \caption{Construction of a $g$-coating of digirth $g$ for a subgraph of a skeleton $G$.}
    \end{figure}

\begin{proof}
    Following \Cref{def:coating}, we show that we can delete an edge (see \Cref{subfig:proof_coating_edge_deletion}) and delete a vertex (see \Cref{subfig:proof_coating_vertex_deletion}) of the skeleton $G$ and obtain a new coating $H'$ from the initial one $H$ without reducing the digirth.

    \begin{enumerate}[label=(\alph*)]
        \item Let $uv \in E(G)$, and let $s \in V(H)$ be the link vertex associated to the edge $uv$. By splitting the vertex $s$ into two vertices $s_1 \in C_u \setminus C_v$ and $s_2 \in C_v \setminus C_u$ and removing the link arcs associated to the edge $uv$, we obtain a graph $H'$ which is a $g$-coating of digirth $g$ of $G - \{ uv \}$.
        \item Let $v \in V(G)$. By deleting all vertices of $H$ that are in $C_v$ and that are not link vertices we obtain a graph $H'$ which is a $g$-coating of digirth $g$ of $G \setminus \{ v \}$.
    \end{enumerate}
\end{proof}

\begin{proposition}\label{prop:coating_properties}
    Let $G$ be a skeleton graph and $H$ a coating of $G$. The following properties hold:
    \begin{enumerate}[(i)]
        \item\label{prop:fvs_greater_than_fG}
        
        $\fv(H) \ge f_G = \lvert F(G) \rvert$

        \item\label{prop:number_of_faces_coating}
        $f_H = f_G + n_G + 2 m_G$. Moreover, if $G$ is connected, then $f_H = 2 + 3 m_G$.
        \item\label{prop:sum_length_cycles} $n_H = \sum_{v \in V(G)} \lvert C_v \rvert - m_G = \sum_{f \in F(G)} \lvert C_f \rvert + m_G$
        \item\label{prop:size_coating} If $H$ is a $g$-coating of $G$, then $n_H = g \cdot n_G - m_G$. If $G$ is connected, then $m_H = g \cdot n_G + 2m_G$.
        \item\label{prop:maximal_degree_skeleton} If $G$ admits a $g$-coating, then $g \ge 2\Delta(G)$ where $\Delta(G)$ denotes the maximum degree of $G$.
    \end{enumerate}
\end{proposition}

\begin{proof}
    Item~\ref{prop:fvs_greater_than_fG} is straightforward since the cycles associated to faces of $G$ are pairwise vertex-disjoint.
    
    To show~\ref{prop:number_of_faces_coating}, observe that:
    \begin{itemize}
        \item each vertex of $G$ induces a face in $H$ (bounded by a clockwise directed cycle),
        \item each face of $G$ induces a face in $H$ (bounded by one or multiple counterclockwise directed cycles),
        \item each edge $uv$ of $G$ gives two triangular faces in $H$ (on each side of the link vertex associated to $uv$).
    \end{itemize}
    
    Thus $f_H = f_G + n_G + 2m_G$. Moreover, if $G$ is connected, then applying Euler's formula we get $f_H = 2 + 3 m_G$.

    For the proof of~\ref{prop:sum_length_cycles}, note that every vertex of $H$ appears in exactly one cycle associated to a vertex of $G$ and in one (union of) cycle(s) associated to a face of $G$ except the link vertices. Link vertices appear in exactly two cycles associated to a vertex of $G$ and in no cycle associated to a face of $G$. There are $m_G$ link vertices in $H$, this proves the equalities.
    
    Item~\ref{prop:size_coating} follows directly from $\ref{prop:sum_length_cycles}$ as $\lvert C_v \rvert = g$ for every $v \in V(G)$. If $G$ is connected then $H$ is also connected and by~\ref{prop:number_of_faces_coating} and Euler's formula we get $m_H = g \cdot n_G + 2m_G$.
    
    Finally observe that if $v$ is degree-$k$ vertex of a skeleton $G$, then the directed cycle $C_v$ associated to $v$ in any coating $H$ of $G$ has length at least $2k$ as link vertices are never adjacent and~\ref{prop:maximal_degree_skeleton} follows.
\end{proof}

\begin{lemma}\label{lem:fvs_without_link_vertices}
Every coating has a minimum feedback vertex set containing no link vertices.
\end{lemma}

\begin{proof}
    Let $S$ be a minimum feedback vertex set of a coating $H$. Suppose $S$ contains a link vertex $s$ associated to an edge $uv$ of the underlying skeleton. Let $a_1$ and $a_2$ (resp. $b_1$ and $b_2$) be the two in-neighbors (resp. out-neighbors) of $s$ (see \Cref{fig:fvs_without_link_vertices}).

    \begin{figure}[htbp]
        \centering
        \includegraphics[width = 0.6\textwidth]{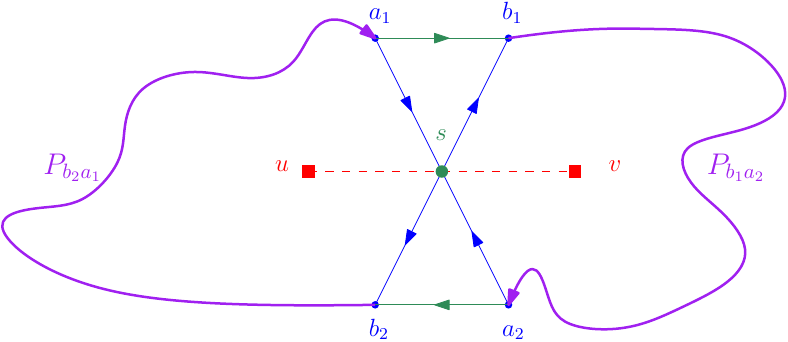}
        \caption{Notations for the proof of \Cref{lem:fvs_without_link_vertices}.}\label{fig:fvs_without_link_vertices}
    \end{figure}
    Let $S_1 = S - s + a_1$ and $S_2 = S - s + a_2$. If $S_1$ is not a feedback vertex set of $H$, then there exists a directed cycle $s  b_1  \dots  a_2 s$ (because $S$ is a feedback vertex set of $H$) that does not contain vertices of $S_1$, in particular there is a directed path $P_{b_1 a_2}$ from $b_1$ to $a_2$ in $H - S$. Similarly if $S_2$ is not a feedback vertex set of $H$, then one can find a directed path $P_{b_2 a_1}$ from $b_2$ to $a_1$ in $H - S$. But $P_{b_1 a_2}$ and $P_{b_2 a_1}$ cannot both exist at the same time in $H - S$ as together with arcs $a_1b_1,a_2b_2$ they would form a directed cycle in $H - S$. Hence either $S_1$ or $S_2$ is a minimum feedback vertex set of $H$. We can iterate the process to replace every link vertex in $S$ by one of its two in-neighbors.
\end{proof}

\begin{corollary}\label{cor:fvs_coating}
    Let $G$ be a plane undirected graph and $H$ a coating of $G$, then $\fv(H) \ge n_G$.
\end{corollary}

\begin{proof}
    From \Cref{lem:fvs_without_link_vertices} there exists a minimum feedback vertex set $S$ of $H$ that does not contain any link vertices. Then the $n_G$ cycles of $H$ associated to vertices of $G$ are hit by $n_G$ different vertices of $S$. Hence $\fv(H) = \lvert S\rvert \ge n_G$.
\end{proof}

\begin{remark}\label{rem:lower_bound_max_nG_fG}
    \Cref{prop:coating_properties}\ref{prop:fvs_greater_than_fG} together with \Cref{cor:fvs_coating} imply that $ \fv(H) \ge \max(n_G, f_G)$. This inequality can be strict as shown by \Cref{fig:example_fvs_larger_than_fG_and_nG}. Indeed, in this figure, the skeleton $G$ has $n_G = 10$ and $f_G = 10$. Every coating $H$ of $G$ has at least 11 vertex-disjoint cycles (the green cycles associated to vertices and faces of $G$) and thus $ \fv(H) \ge 11 > \max(n_G, f_G)$.
    
    \begin{figure}[htbp]
        \centering
        \includegraphics[height=5cm]{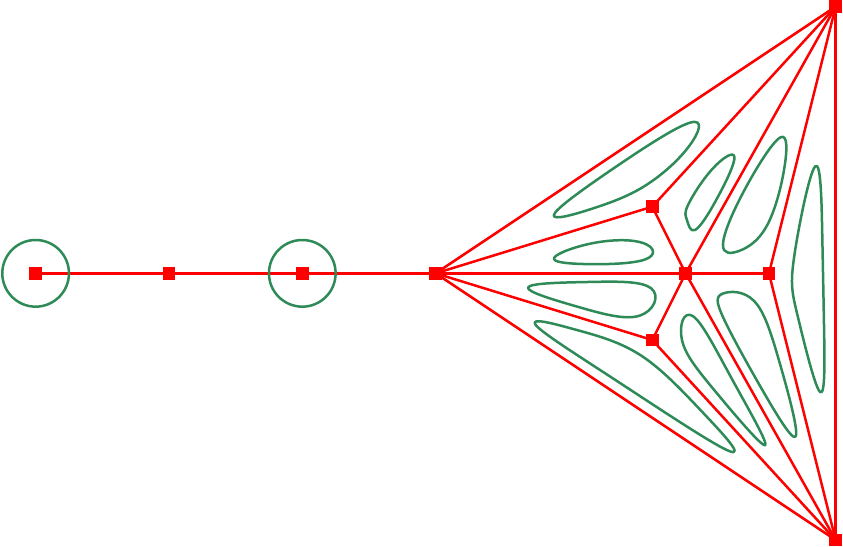}
        \caption{A skeleton graph $G$ (in red) and 11 vertex-disjoint cycles of any coating of $G$ (in green).}\label{fig:example_fvs_larger_than_fG_and_nG}
    \end{figure}
\end{remark}
    
The following remark, although not directly used in the proofs, is of independent interest as it highlights an inductive structure of skeletons and coatings. For example, it can be used to prove \Cref{cor:fvs_coating} using an induction on $m_G$.

\begin{remark}\label{rem:edge_contraction_skeleton}
    Let $H$ be a coating of some skeleton $G$. Let $s$ be the link vertex associated to an edge $uv \in E(G)$.
    \begin{itemize}
        \item If $u \neq v$ then $H' = H - s$ is a coating of $G' = G/uv$ where the edge $uv$ is contracted. (see \Cref{fig:coating_edge_contraction})
        \item If $u = v$ then $H' = H - s$ is a coating of $G' = G_1 \uplus G_2$ where $G_1$ (resp. $G_2$) is the graph induced by the edges lying in the exterior (resp. interior) of the loop $uu$ (see \Cref{fig:contraction_loop}).
    \end{itemize}
    \begin{figure}[htbp]
        \centering
        \includegraphics[width = \textwidth]{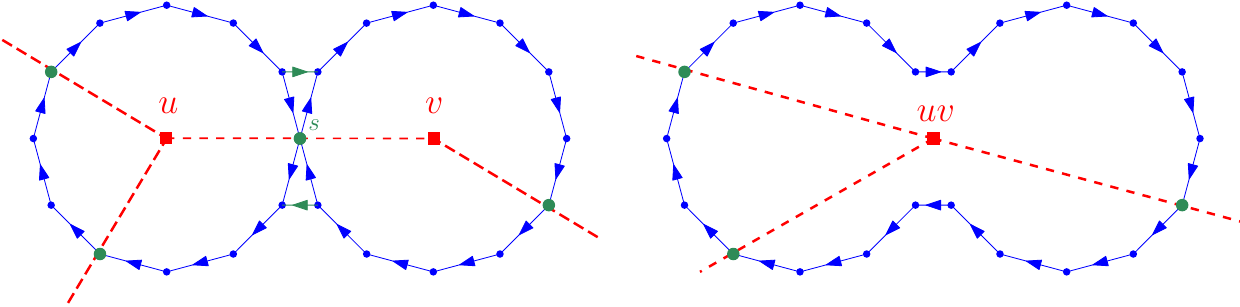}
        \caption{The deletion of a link vertex $s$ of the coating $H$ associated to an edge $uv$ of the skeleton $G$ gives a coating $H' = H - s$ of the skeleton $G'=G/uv$. The skeletons and coatings are drawn in red and blue respectively.}\label{fig:coating_edge_contraction}
    \end{figure}
    \begin{figure}[htbp]
        \centering
        \includegraphics[width = \textwidth]{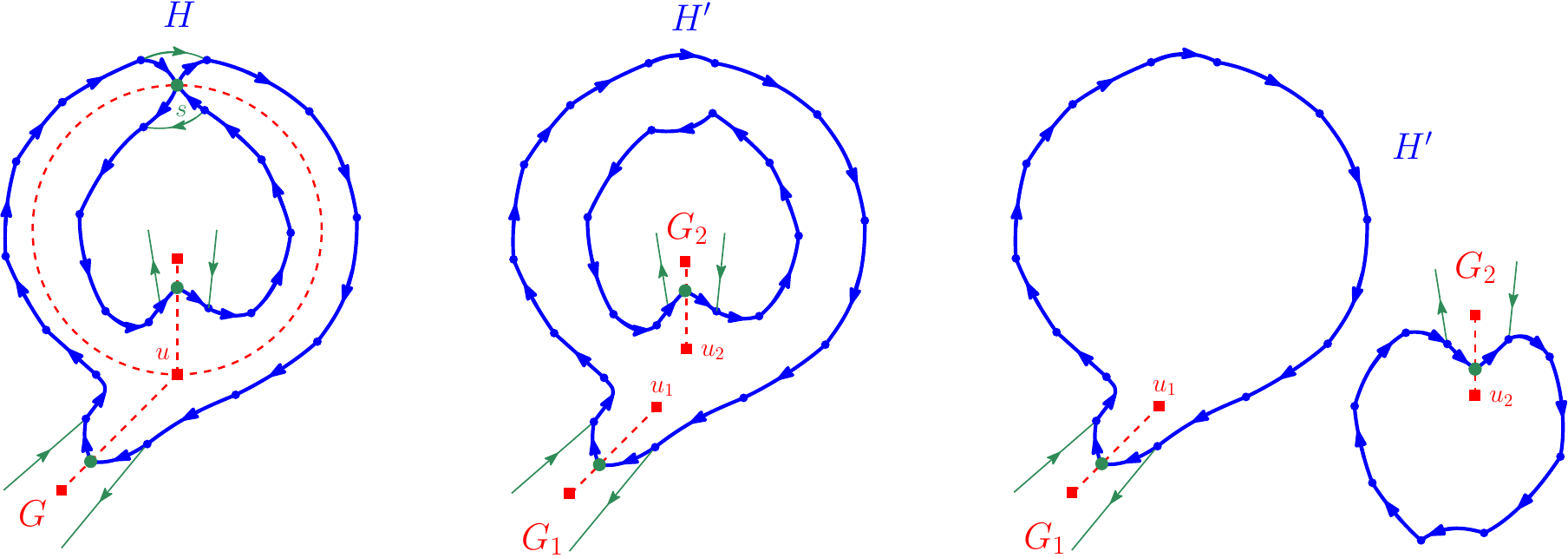}
        \caption{The deletion of a link vertex $s$ of a coating $H$ associated to a loop $uu$ of a skeleton $G$, gives a coating $H'=H-s$ of the skeleton $G' = G_1 \uplus G_2$ obtained from $G$ where the interior and exterior of $uu$ are disconnected. The skeletons and coatings are drawn in red and blue respectively.}\label{fig:contraction_loop}
    \end{figure}
\end{remark}

\begin{lemma}\label{lem:coating_digirth_g_normal_sets}
    If $H$ is a coating of digirth $g$ of some skeleton graph $G$, then $g\cdot\fv(H) \le n_H + m_G$.
\end{lemma}

\begin{proof}
    We use \Cref{thm:maximal_normal_set_greater_than_fvs}. Consider a maximum normal set $\mathcal{N}_{\max}$ of $H$. 
    A vertex $v$ of $H$ that is neither a link vertex nor an extremity of a link arc has degree 2 and thus can be used in at most one cycle of $\mathcal{N}_{\max}$. A link vertex $v$ has degree 4 by construction of $H$ and thus can be used in at most two cycles of $\mathcal{N}_{\max}$. Finally, by the planarity of $H$, a vertex $v$ that is an extremity of a link arc has degree either 3 or 4. By definition of a coating, such a vertex is not alternatingly oriented and thus cannot be used in two alternatingly directed cycles of $\mathcal{N}_{\max}$.

    \begin{figure}[htbp]
        \centering
        \includegraphics[height=4cm]{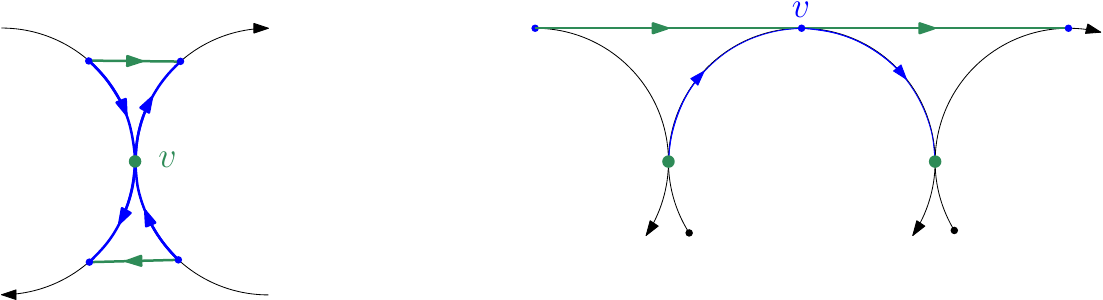}
        \caption{The two types of degree 4 vertices of a coating. Left: link vertex $v$. Right: vertex $v$ adjacent to two link vertices.}\label{fig:degree_4_vertices_in_coatings}
    \end{figure}
   
     Therefore, the only vertices that can be used in two cycles of $\mathcal{N}_{\max}$ are link vertices. As there are exactly $m_G$ link vertices in $H$, we have $\sum_{C \in \mathcal{N}_{\max}} \lvert C \rvert \le n_H +  m_G$. Since $H$ has digirth $g$, we have $g \cdot \lvert \mathcal{N}_{\max} \rvert \le \sum_{C \in \mathcal{N}_{\max}} \lvert C \rvert$ and hence by \Cref{thm:maximal_normal_set_greater_than_fvs} we conclude that $g\cdot \fv(H) \leq n_H + m_G$.
\end{proof}

Now we refine \Cref{cor:fvs_coating} for $g$-coatings:

\begin{theorem}\label{thm:fvs_g_coating_digirth_g}
    If $H$ is a $g$-coating of digirth $g$ of some skeleton graph $G$, then $\fv(H) = n_G$.
\end{theorem}

\begin{proof}
    Recall that $ \fv(H) \ge n_G$ by \Cref{cor:fvs_coating}. It remains to show that $\fv(H) \le n_G$. From \Cref{lem:coating_digirth_g_normal_sets} we know that $g . \fv(H) \le n_H + m_G$. Since $H$ is a $g$-coating, by \Cref{prop:coating_properties}\ref{prop:size_coating} it follows that $n_H + m_G = g .n_G$ and thus $\fv(H) \le n_G$.
\end{proof}

\begin{corollary}\label{cor:computing_fvs_g_coating_of_digirth_g}
    Let $H$ be a $g$-coating of digirth $g$ of some skeleton $G$ and let $\alpha, \beta$ be two constants such that $m_G = \alpha \times n_G - \beta$ with $\alpha \neq g$. Then $\fv(H) = \frac{n_H - \beta}{g - \alpha}$.
\end{corollary}

\begin{proof}
    As $H$ is a $g$-coating, from \Cref{prop:coating_properties}\ref{prop:size_coating} we have
    \begin{displaymath}
        n_H = g \times n_G - m_G = g \times n_G - \alpha \times n_G + \beta = (g- \alpha) \times n_G + \beta
    \end{displaymath}
    By \Cref{thm:fvs_g_coating_digirth_g}, we conclude that $\fv(H) = n_G = \frac{n_H - \beta}{g - \alpha}$.
\end{proof}

\subsection{Coating functions}

In this section, we introduce the notion of \emph{coating functions}, which provide an alternative way to represent a coating of a skeleton graph and are useful for constructing coatings with specific properties.

\begin{definition}[Corner]
    Let $G$ be a plane graph. Let $v$ be a vertex of $G$ and let $e_1, e_2$ be two half-edges incident to $v$ that are consecutive in the clockwise order around $v$. The triple $c = (e_1, v, e_2)$ is called a \emph{corner} of $G$. 
    If $v$ is an isolated vertex, we define the corner as the triple $(\emptyset, v, \emptyset)$. If $v$ has degree 1 with incident half-edge $e$, we define its single corner as $(e, v, e)$. Also, for consistency, if $v$ is a vertex whose only incident edge $e$ is a loop, we define its two corners as ${(e,v,e)}^{int}$ and ${(e,v,e)}^{ext}$.
    
     The set of all the corners of $G$ is denoted by $\mathcal{K}$. Note that each corner $c$ is incident to a unique face $f \in F(G)$ and to a unique vertex $v \in V(G)$. For a vertex $v \in V(G)$ (resp. a face $f \in F(G)$) denote by $\mathcal{K}(v)$ (resp. $\mathcal{K}(f)$) the sets of corners incident to $v$ (resp. incident to $f$).
    
    Let $H$ be a coating of some skeleton graph $G$. Note that every vertex of $H$ is either a link vertex or a vertex that lies in a corner of $G$. Given a directed cycle $D$ of $H$, we denote by $\mathcal{K}(D)$ the set of corners where the vertices of $D$ lie.
\end{definition}

\begin{definition}[Coating Function]
    Given a skeleton plane graph $G$, a \emph{coating function} of $G$ is a function $h \colon \mathcal{K} \to \mathbb{N}^*$ where $h(c)$ denotes the number of vertices of the coating lying strictly within corner $c$, excluding link vertices. This establishes a bijection between coatings of $G$ and coating functions of $G$.
\end{definition}

As an example, \Cref{fig:example_coating_function} shows the coating $H$ of a plane graph $G$ defined by the following coating function $h$:
\[
\begin{array}{c|cccccccccccccc}
c & c_1&c_2&c_3&c_4&c_5&c_6&c_7&c_8&c_9&c_{10}&c_{11}&c_{12}&c_{13}&c_{14}\\ \hline
h(c) &5&1&1&1&6&1&5&3&2&6&2&3&2&2
\end{array}
\]
   
    \begin{figure}[htbp]
        \centering
        \includegraphics[height=5cm]{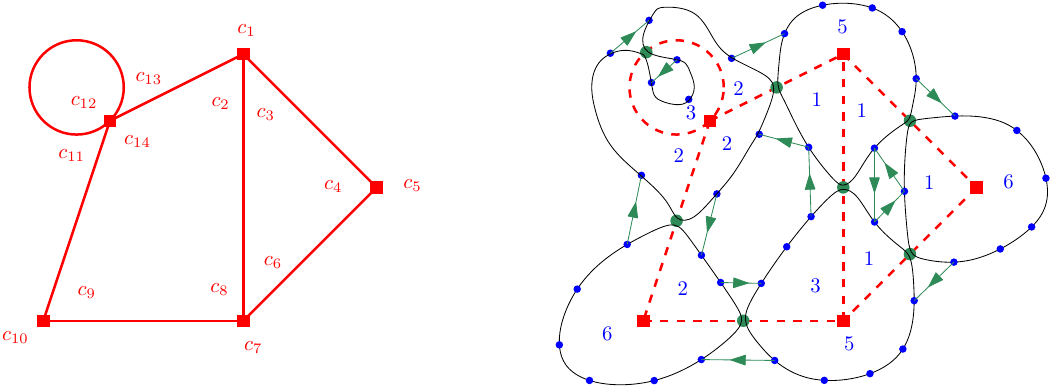}
        \caption{Left: a graph $G$ with its corners. Right: a coating $H$ of $G$ with its associated coating function in blue. The black cycles associated to the vertices of $G$ are oriented clockwise.}\label{fig:example_coating_function}
    \end{figure}

\begin{remark}\label{rem:sum_coating_function}
    Given a plane graph $G$ and a coating function $h$, the following hold:
    \begin{itemize}
        \item $\sum_{c \in \mathcal{K}} h(c) = n_H - m_G$.
        \item The function $h$ defines a $g$-coating if and only if for all $v \in V(G), \sum_{c \in \mathcal{K}(v)} h(c) = g - \deg(v)$.
    \end{itemize}
\end{remark}

The cycles associated to vertices and faces of a skeleton graph in \Cref{def:coating} provide a necessary condition for a coating function to define a coating of digirth $g$.

\begin{observation}\label{obs:necessary_condition_digirth_coating}
    Let $G$ be a plane graph and $h$ a coating function of $G$. If $h$ defines a coating of digirth $g$ then:
    \begin{itemize}
        \item $\forall v \in V(G),  \sum_{c \in \mathcal{K}(v)} h(c) = \lvert C_v \rvert - \deg(v) \ge g - \deg(v)$;
        \item $\forall f \in F(G),  \sum_{c \in \mathcal{K}(f)} h(c) = \lvert C_f \rvert \ge g$.
    \end{itemize}
\end{observation}

We note that the conditions of \Cref{obs:necessary_condition_digirth_coating} to get a coating of digirth $g$ are not sufficient to guarantee digirth $g$, as shown in \Cref{fig:counter_example_digirth_coating}.

\begin{figure}[htbp]
    \centering
    \includegraphics[height=5cm]{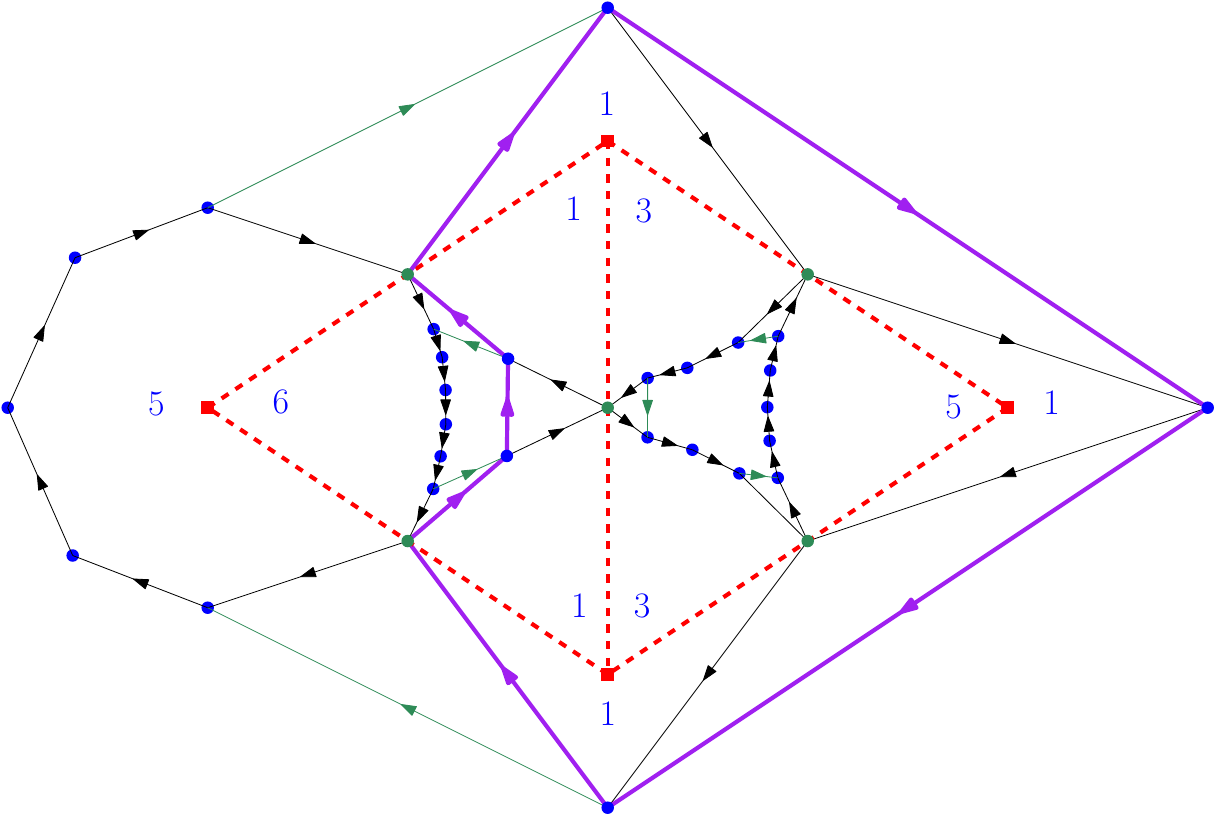}
    \caption{A skeleton $G$ (in red) with a coating function $h$ (in blue) that satisfies the conditions of \Cref{obs:necessary_condition_digirth_coating} for $g=8$ but nevertheless the corresponding coating contains a directed cycle of length 7 (in purple).}\label{fig:counter_example_digirth_coating}
\end{figure}

\begin{remark*}
    Let $G$ be a plane undirected graph and $H_1$ and $H_2$ two coatings of $G$. Then $ \fv(H_1) = \fv(H_2) $, meaning that the size of a minimum feedback vertex set of a coating of $G$ does not depend on the choice of the coating.
    
    Let $F_1$ be a minimal feedback vertex set of $H_1$ that does not use any link vertex (\Cref{lem:fvs_without_link_vertices}) and let $D_2$ be a directed cycle of $H_2$. Note that there exists a directed cycle $D_1$ of $H_1$ such that $\mathcal{K}(D_2) = \mathcal{K}(D_1)$. Then there exists a vertex $v \in F_1$ that lies in a corner $c \in \mathcal{K}(D_2)$. Observe that every directed cycle of a coating passing through a vertex of a corner $c$ must contain all vertices of that corner, therefore $v$ is a vertex of $D_2$. This shows that $F_1$ is a feedback vertex set of $H_2$ and then $\fv(H_1) = \fv(H_2)$.
\end{remark*}

The following application constructs infinite families of digraphs of digirth $g$ whose minimum feedback vertex sets have size $\frac{n}{g-1}$. In particular, for each digirth $g\geq 4$, this yields a slightly better ratio $\frac{\fv(H)}{n}$ than the one obtained in~\cite{KVW17}.

\begin{application}\label{appl:skeletonCk}
    \begin{figure}[htbp]
        \centering
        \includegraphics[height=5cm]{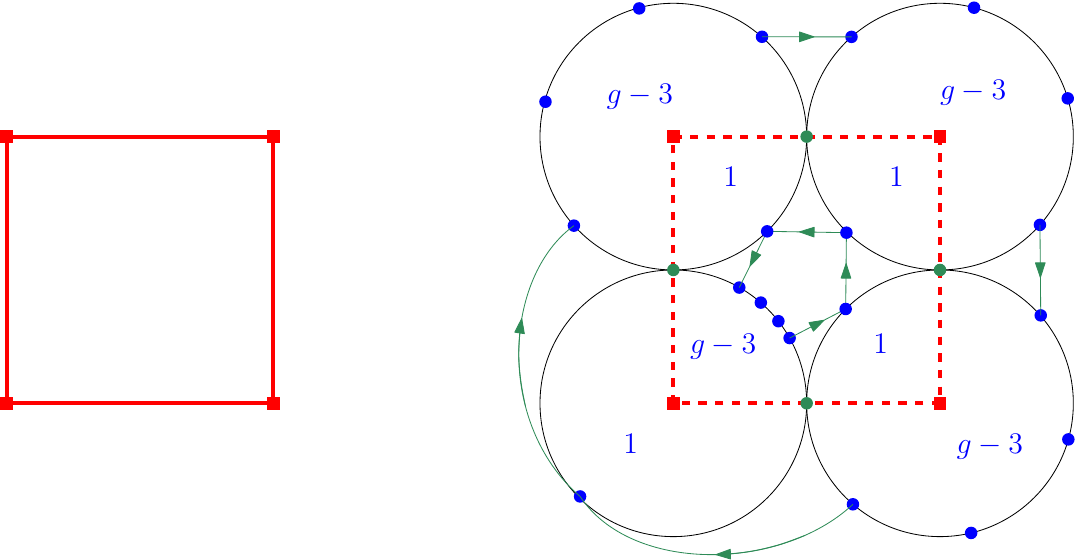}
        \caption{A $g$-coating $H$ of the skeleton graph $G = C_4$ with its associated coating function in blue. The black cycles associated to the vertices of $G$ are oriented clockwise.}\label{fig:coating_of_the_intro}
    \end{figure}
    
    Consider a $g$-coating $H$ of digirth $g \ge 4$ of the skeleton graph $G = C_4$ (undirected 4-cycle) with the following coating function defined in \Cref{fig:coating_of_the_intro}. One can verify that $h$ defines a $g$-coating and that this coating has digirth $g$. Since $m_G = n_G$, by \Cref{cor:computing_fvs_g_coating_of_digirth_g} it follows that $\fv(H) = \frac{n_H}{g-1} = 4$.
    
    For a fixed digirth $g$, one can naturally extend this construction to build an infinite family of planar digraphs $H$ such that $\fv(H)=\frac{n_H}{g-1}$ by considering the skeletons $G_k = C_k$ for all $k \ge 4$.
\end{application}

The following proposition shows that the construction of \Cref{appl:skeletonCk} has the largest ratio $\frac{\fv(H)}{n}$ one can obtain using $g$-coatings of digirth $g$ for $g = 4,5$. It also shows that the ratio cannot be improved for $g=4$ using any type of coatings. 

\begin{proposition}\label{prop:bound_ratio_for_low_digirth}
    Let $H$ be a coating of digirth $g$. Then $\fv(H) \le \frac{4n}{3g}$. Moreover, if $H$ is a $g$-coating then $\fv(H) \le \frac{n_H}{g- \frac{1}{2} \left \lfloor \frac{g}{2} \right \rfloor}$.
\end{proposition}

Note that the two bounds coincide for even values of $g$.

\begin{proof}
   Let $G$ be the skeleton of $H$. From \Cref{lem:coating_digirth_g_normal_sets}, $\fv(H) \le \frac{n_H + m_G}{g}$. By \Cref{prop:coating_properties}\ref{prop:sum_length_cycles} and the fact that $\lvert C_v \rvert \ge 2\deg(v)$ for every $v \in V(G)$:
    \begin{displaymath}
        n_H + m_G = \sum_{v \in V(G)} \lvert C_v \rvert \ge \sum_{v \in V(G)} 2\deg(v) = 4m_G,
    \end{displaymath}
    so $m_G \le \frac{n_H}{3}$, and therefore $\fv(H) \le \frac{n_H + m_G}{g} \le \frac{4n_H}{3g}$.\\
    Now suppose $H$ is a $g$-coating. By \Cref{thm:fvs_g_coating_digirth_g}, $\fv(H) = n_G$. By \Cref{prop:coating_properties}\ref{prop:size_coating}, $n_H = g \cdot n_G - m_G$. Since $\Delta(G) \le \lfloor g/2 \rfloor$ by \Cref{prop:coating_properties}\ref{prop:maximal_degree_skeleton}, we have $m_G = \frac{1}{2}\sum_{v \in V(G)} \deg(v) \le \frac{1}{2}\left\lfloor \frac{g}{2} \right\rfloor \cdot n_G$, and therefore:
    \begin{displaymath}
        n_H \ge \left(g - \tfrac{1}{2}\left\lfloor \tfrac{g}{2} \right\rfloor\right) n_G = \left(g - \tfrac{1}{2}\left\lfloor \tfrac{g}{2} \right\rfloor\right) \fv(H). \qedhere
    \end{displaymath}
\end{proof}

We now prove a series of results allowing us to increase the digirth of a coating $H$ while maintaining the same growth of $\fv(H)$ relative to $n_H$. 

As a direct consequence of \Cref{thm:fvs_g_coating_digirth_g}, it follows that if $G$ is a skeleton that admits a $g$-coating of digirth $g$ (for any value of $g \ge 2$), then for every coating $H$ of $G$ (not necessarily a $g$-coating), $\fv(H) = n_G$.

\begin{theorem}\label{thm:ratio_extension_to_higher_digirth}
    Let $G$ be a plane undirected graph that admits a $g$-coating of digirth $g$. Let $\alpha, \beta$ be two constants such that $m_G = \alpha \cdot n_G - \beta$ with $\alpha < g$. Then for every integer $r \ge 0$, there exists a $g_r$-coating $H_{r}$ of $G$ with digirth $g_r = g+r$ and $\fv(H_{r})  = n_G = \frac{n_{H_r}-\beta}{g_r - \alpha}$.
\end{theorem}

\begin{proof}
    Let $H$ be a $g$-coating of digirth $g$ of $G$ and $h_0$ the coating function of $H$. We construct a family of coating functions $(h_r)_{r \ge 0}$ such that for all $r \ge 0$, $h_r$ defines a coating of digirth $g+r$ of $G$.
    
    By \Cref{lem:fvs_without_link_vertices}, there exists a minimum feedback vertex set $S$ of $H$ with no link vertices. Let $\mathcal{K}$ be the set of all corners of $G$ and denote by $\mathcal{K'}$ the set of corners of $G$ where the vertices of $S$ lie. Construct the coating function $h_r$ as follows:
    \begin{displaymath}
        \forall c \in \mathcal{K} \backslash \mathcal{K'}, \quad h_r(c) = h_0(c) \qquad \text{and} \qquad \forall c \in \mathcal{K'}, \quad h_r(c) = h_0(c) + r
    \end{displaymath}
    
    Let $H_r$ be the coating defined by the coating function $h_r$. As $\lvert S \rvert = n_G$ (by \Cref{thm:fvs_g_coating_digirth_g}) and $S$ does not contain any link vertices, $S \cap C_v$ contains exactly one vertex for every $v \in V(G)$. Then for every $v \in V(G)$ we have $\lvert C_v \rvert = g+r = g_r$, thus $H_r$ is a $g_r$-coating of $G$.
    
    We now show that $H_r$ has digirth $g_r = g+r$.
    Let $D$ be a directed cycle of $H_r$. Let $x$ be the number of link vertices in $D$ and $\mathcal{K}(D) = \lbrace c_1, \dots, c_i \rbrace$ the list of corners of $G$ where the vertices of $D$ lie. Note that $\mathcal{K}(D)$ is a list of corners of a directed cycle $D'$ in $H$ and $\lvert D' \rvert = x + \sum_{c \in \mathcal{K}(D)} h_0(c)$. Then since $S$ is a feedback vertex set of $H$ it follows that $\mathcal{K}(D) \cap \mathcal{K}' \neq \emptyset$ and then
    \begin{displaymath}
        \lvert D \rvert = x + \sum_{c \in \mathcal{K}(D)} h_r(c) \ge x + r + \sum_{c \in \mathcal{K}(D)} h_0(c) = \lvert D' \rvert + r \ge g+r = g_r.
    \end{displaymath}
    
    Moreover using \Cref{rem:sum_coating_function} we get
    \begin{displaymath}
        n_{H_r} =  m_G + \sum_{c \in \mathcal{K}} h_r(c) = m_G + \sum_{c \in \mathcal{K}} h_0(c) + r \cdot \lvert S \rvert = n_H + r \cdot \fv(H).
    \end{displaymath}
    
    Now, since $H_r$ is a $g_r$-coating of digirth $g_r$ and since $m_G = \alpha \cdot n_G - \beta$ with $\alpha < g\leq g_r$, by \Cref{cor:computing_fvs_g_coating_of_digirth_g} we conclude that
    \begin{displaymath}
        \fv(H_r) = \frac{n_{H_r} - \beta}{g_r - \alpha}
    \end{displaymath}
\end{proof}

\subsection{Constructions with large feedback vertex set and small digirth}\label{sec:coating_small_digirth}

In \Cref{appl:skeletonCk} we constructed an infinite family $\mathcal{F}_g$ of planar digraphs with digirth $g \ge 4$ such that every $H \in \mathcal{F}_g$ satisfies $\fv(H) = \frac{n}{g-1}$.
In this section we construct several infinite families of digraphs of digirth $g\ge 6$ with larger ratio $\frac{\fv(H)}{n}$.

\begin{definition}\label{def:recursive_coating}
    Let $\ell$ be a positive integer. A family of coatings $(H_k)_{k \ge 0}$ of skeletons $(G_k)_{k \ge 0}$ is said to be \emph{recursive} if the following hold:
    \begin{enumerate}
        \item $G_0$ is the undirected cycle of length $\ell$. The graph $G_1$ has two distinguished facial $\ell$-cycles $C_{in} = u_1\dots u_\ell$ and $C_{out} = v_1\dots v_\ell$. For $k \ge 2$, $G_k$ is obtained from $G_{k-1}$ by identifying the vertices of $C_{out}$ of $G_{k-1}$ with those of $C_{in}$ of a new copy of $G_1$. Equivalently, $G_k$ is a chain of $k$ copies of $G_1$, with $C_{out}$ of each copy identified with $C_{in}$ of the next.
        \item There exists a coating function $h_{G_1}$ of $G_1$ such that for every $k \ge 0$ and every corner $c$ of $G_k$, the coating function $h_k$ of $H_k$ satisfies $h_k(c) = h_{G_1}(c')$, where $c'$ is the corner of $G_1$ in the copy of $G_1$ containing $c$. In other words, the coating function is the same in every copy of $G_1$.
    \end{enumerate}
\end{definition}

\begin{lemma}\label{lem:digirth_recursive_coating}
    Let $(H_k)_{k \ge 0}$ be a recursive family of coatings of skeletons $(G_k)_{k \ge 0}$ with $H_0$ and $H_1$ having digirth $g$. Suppose moreover that for every pair $u, v$ of link vertices of $H_0$, $d_{H_1}(u,v) \ge d_{H_0}(u,v)$. Then $H_k$ has digirth $g$ for all $k \ge 0$.
\end{lemma}

\begin{proof}
    We first establish that for any pair $(u,v)$ of link vertices on the outer face of $G_k$, $d_{H_{k+1}}(u,v) \ge d_{H_k}(u,v)$. For $k=0$ this is the hypothesis. For $k \ge 1$, building $H_{k+1}$ from $H_k$ by attaching a new copy of the $G_1$-coating is structurally identical to building $H_1$ from $H_0$, so the same hypothesis applies.

    We now prove by induction on $k$ that $H_k$ has digirth $g$. The base cases $k=0$ and $k=1$ are the hypothesis. Assume $H_k$ has digirth $g$, and let $D$ be a directed cycle of $H_{k+1}$. Let $\mathcal{K}(D)$ be the set of corners of $G_{k+1}$ visited by $D$.
    \begin{itemize}
        \item If $\mathcal{K}(D) \subseteq \mathcal{K}(G_k)$, then $D$ is a cycle of $H_k$, so $\lvert D \rvert \ge g$ by the induction hypothesis.
        \item If $\mathcal{K}(D) \subseteq \mathcal{K}(G_{k+1}) \setminus \mathcal{K}(G_k)$, then $D$ lies entirely in the new copy of $G_1$, so $\lvert D \rvert \ge g$ since $H_1$ has digirth $g$.
        \item Otherwise, $D$ visits corners from both parts. Since the $H_k$-part and the new $G_1$-copy share only the link vertices on the outer face of $G_k$, the cycle $D$ must pass through at least two such link vertices $v_1$ and $v_p$. Let $P = v_1 \dots v_p$ be the portion of $D$ traversing the $H_k$-part (with all intermediate vertices in interior corners of $G_k$). Closing $P$ with a shortest path $P'$ from $v_p$ to $v_1$ in $H_k$ yields a cycle $D'$ in $H_k$. By the induction hypothesis $\lvert D'\rvert = p + d_{H_k}(v_p, v_1) \ge g$, and since $d_{H_{k+1}}(v_p, v_1) \ge d_{H_k}(v_p, v_1)$ we conclude $\lvert D \rvert \ge p + d_{H_{k+1}}(v_p, v_1) \ge \lvert D' \rvert \ge g$. \qedhere
    \end{itemize}
\end{proof}

We are now ready to give some constructions for small digirths.

\begin{theorem}\label{thm:general_lower_bound_small_digirth}
    For each item below, for every $g$ at least the stated threshold, there exists an infinite family $\mathcal{F}_g$ of planar digraphs of digirth $g$ such that every $H \in \mathcal{F}_g$ satisfies:
    \begin{enumerate}[(i)]
        \item\label{itm:digirth6} $g \ge 6$: \quad $\displaystyle\fv(H) = \frac{n - 2}{g-\frac{3}{2}}$.\smallskip
        \item\label{itm:digirth8} $g \ge 8$: \quad $\displaystyle\fv(H) = \frac{n - \frac{12}{5}}{g-\frac{8}{5}}$.\smallskip
        \item\label{itm:digirth9} $g \ge 9$: \quad $\displaystyle\fv(H) = \frac{n - \frac{28}{11}}{g - \frac{18}{11}}$.\smallskip
        \item\label{itm:digirth10} $g \ge 10$: \quad $\displaystyle\fv(H) = \frac{n - \frac{8}{3}}{g - \frac{5}{3}}$.\smallskip
        \item\label{itm:digirth11} $g \ge 11$: \quad $\displaystyle\fv(H) = \frac{n - \frac{36}{13}}{g - \frac{22}{13}}$.
    \end{enumerate}
    In particular, for $g = 6, 8, 9, 10, 11$, there exists infinitely many planar digraphs $H$ of digirth $g$ such that $\fv(H) = \frac{n- \frac{4(g-2)}{g+2}}{g - \frac{2g}{g+2}} = \frac{n(g+2) - 4(g-2)}{g^2}$.
\end{theorem}

\Cref{prop:bound_ratio_for_low_digirth} shows that the bound of \Cref{thm:general_lower_bound_small_digirth}\ref{itm:digirth6} is asymptotically the best one can do using $g$-coatings of digirth 6 and 7. However we do not know whether it is possible to obtain a better bound for $g = 7$ using coatings of digirth 7 that are not 7-coatings.

\subsubsection{Proof of \Cref{thm:general_lower_bound_small_digirth}\ref{itm:digirth6}}\label{subsubsec:coating_function_digirth_6}

Consider the recursive coating family $(H_k)_{k \ge 0}$ of skeletons $(G_k)_{k \ge 0}$ defined in \Cref{fig:coating_function_digirth_6}. The graph $G_1$, with $C_{in} = u_0 v_0 w_0 x_0$ and $C_{out} = u_1 v_1 w_1 x_1$ (see \Cref{def:recursive_coating}), is shown in the middle of \Cref{fig:coating_function_digirth_6}; the coating function is indicated in blue.

\begin{figure}[htbp]
    \centering
    \includegraphics[width = 0.95\textwidth]{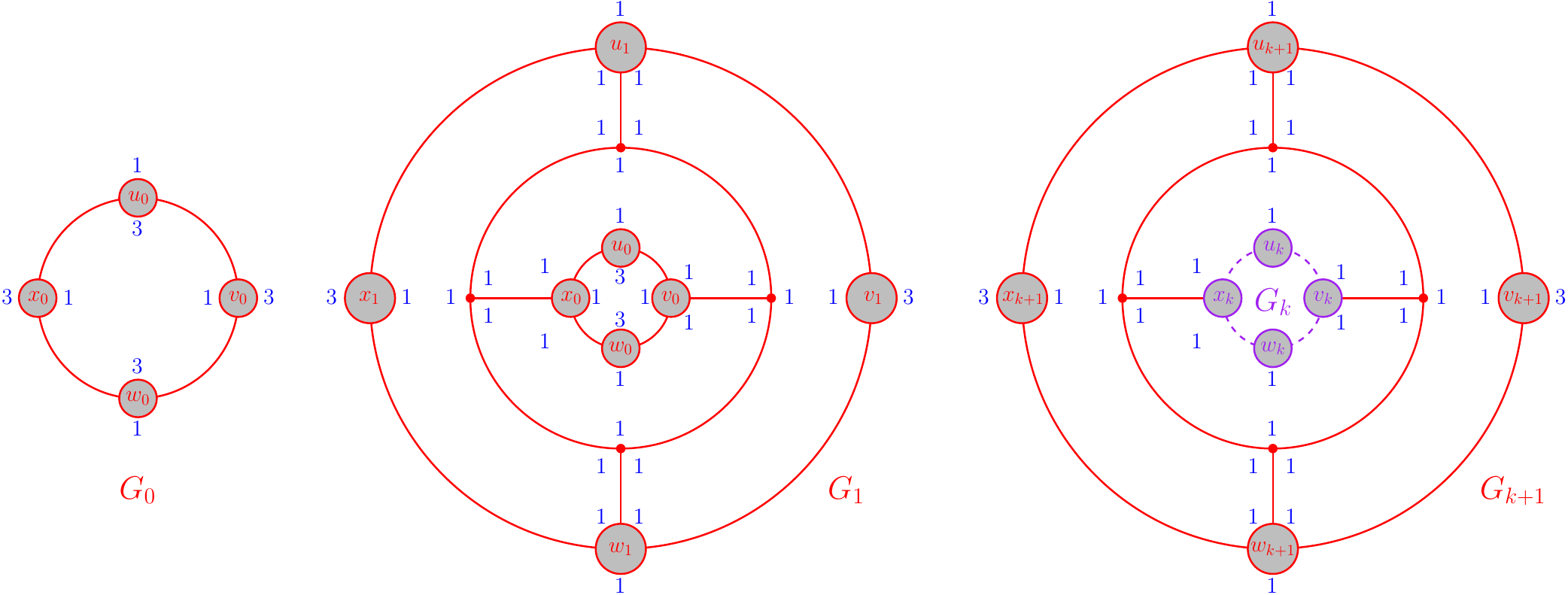}
    \caption{The family $(G_k)_{k \ge 0}$ of \Cref{subsubsec:coating_function_digirth_6} and its coating function in blue.}\label{fig:coating_function_digirth_6}
\end{figure}

For all $k \ge 0$, $H_k$ is a $6$-coating. To prove that $(H_k)_{k \ge 0}$ is a family of coatings of digirth $6$, we apply \Cref{lem:digirth_recursive_coating}, which requires the following.

\begin{observation}\label{obs:digirth6}
    \begin{enumerate}
        \item\label{itm:digirth_H01} $H_0$ and $H_1$ have digirth $6$.
        \item\label{itm:distance_link_vertices} For any two link vertices $y, z$ associated to edges of the cycle $u_0 v_0 w_0 x_0$, we have $d_{H_0}(y,z) = d_{H_1}(y,z)$.
    \end{enumerate}
\end{observation}

\Cref{obs:digirth6}\ref{itm:digirth_H01} can be verified by hand or with a computer, since $H_0$ and $H_1$ have only 20 and 56 vertices respectively.
\Cref{obs:digirth6}\ref{itm:distance_link_vertices} follows by computing $d(y,z)$ for all 16 pairs of link vertices in $H_0$ and $H_1$; the values are given below, where the subscript indicates the associated edge (e.g.\ $y_{uv}$ is the link vertex associated to the edge $u_0v_0$).

\begin{center}
    \begin{tabular}{|c|c|c|c|c|}
        \hline
        $d_{H_0}(y,z) = d_{H_1}(y,z)$  & $y_{uv}$ & $y_{vw}$ & $y_{wx}$ & $y_{xu}$ \\
        \hline
        $z_{uv}$ & 0 & 2 & 5 & 2 \\ 
        \hline
        $z_{vw}$ & 4 & 0 & 4 & 5 \\ 
        \hline
        $z_{wx}$ & 5 & 2 & 0 & 2 \\ 
        \hline
        $z_{xu}$ & 4 & 5 & 4 & 0 \\ 
        \hline
    \end{tabular}
\end{center}

By \Cref{obs:digirth6} and \Cref{lem:digirth_recursive_coating}, $(H_k)_{k \ge 0}$ is a family of coatings of digirth $6$.

The skeleton $G_k$ has $n_{G_k} = 4 + 8k$ vertices and $m_{G_k} = 4 + 12k = \frac{3}{2} n_{G_k} - 2$ edges. Since $H_k$ is a $6$-coating of digirth $6$, \Cref{cor:computing_fvs_g_coating_of_digirth_g} gives $\fv(H_k) = \frac{n_{H_k} - 2}{6 - \frac{3}{2}}$.

By \Cref{thm:ratio_extension_to_higher_digirth}, for every $g \ge 6$ this construction extends to an infinite family $\mathcal{F}_g$ of planar digraphs of digirth $g$ with $\fv(H) = \frac{n-2}{g - \frac{3}{2}}$ for every $H \in \mathcal{F}_g$.

\subsubsection{Proof of \Cref{thm:general_lower_bound_small_digirth}\ref{itm:digirth8}}\label{subsubsec:coating_function_digirth_8}

Consider the recursive coating family $(H_k)_{k \ge 0}$ of skeletons $(G_k)_{k \ge 0}$ defined in \Cref{fig:coating_function_digirth_8}. The graph $G_1$, with $C_{in} = u_0 v_0 w_0 x_0$ and $C_{out} = u_1 v_1 w_1 x_1$ (see \Cref{def:recursive_coating}), is shown in the middle of \Cref{fig:coating_function_digirth_8}; the coating function is indicated in blue.

\begin{figure}[htbp]
    \centering
    \includegraphics[width =\textwidth]{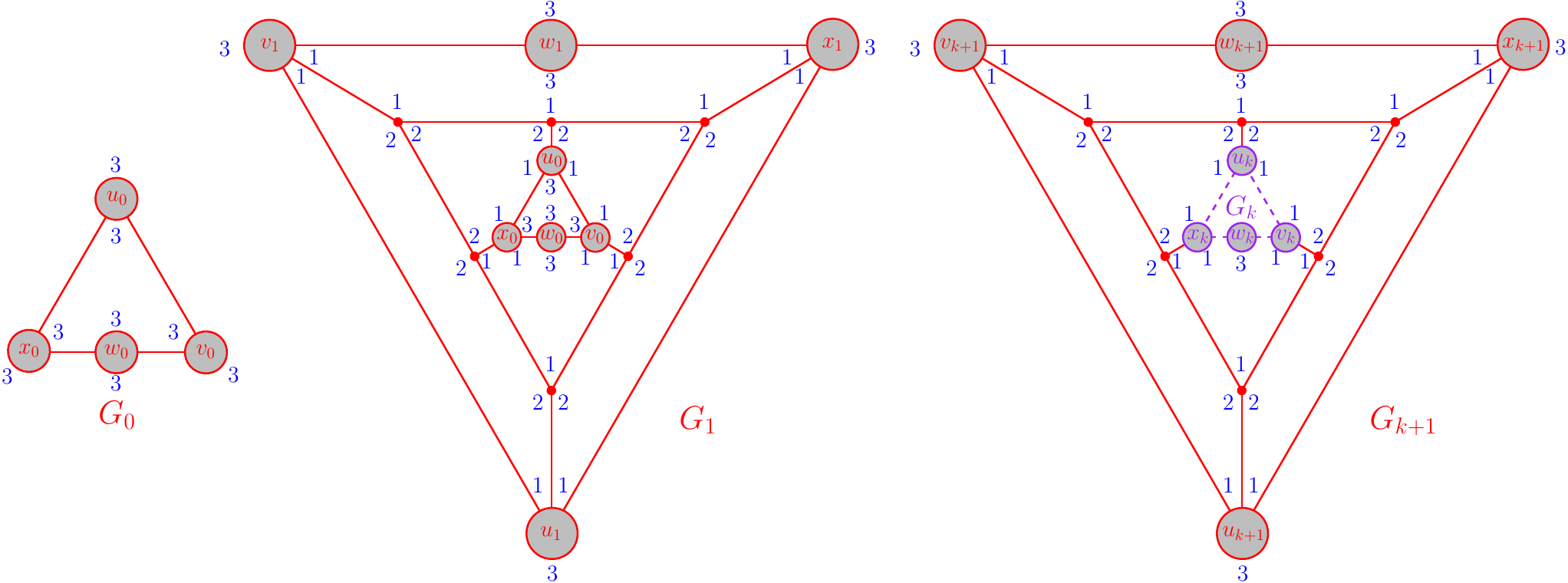}
    \caption{The family $(G_k)_{k \ge 0}$ of \Cref{subsubsec:coating_function_digirth_8} and its coating function in blue.}\label{fig:coating_function_digirth_8}
\end{figure}

For all $k \ge 0$, $H_k$ is an $8$-coating. To prove that $(H_k)_{k \ge 0}$ is a family of coatings of digirth $8$, we apply \Cref{lem:digirth_recursive_coating}, which requires the following.

\begin{observation*}
    \begin{enumerate}
        \item $H_0$ and $H_1$ have digirth $8$.
        \item For any two link vertices $y, z$ associated to edges of the cycle $u_0 v_0 w_0 x_0$, we have $d_{H_0}(y,z) = d_{H_1}(y,z)$.
    \end{enumerate}
\end{observation*}

Both items can be verified by hand or with a computer. By the observation above and \Cref{lem:digirth_recursive_coating}, $(H_k)_{k \ge 0}$ is a family of coatings of digirth $8$.

The skeleton $G_k$ has $n_{G_k} = 4 + 10k$ vertices and $m_{G_k} = 4 + 16k = \frac{8}{5} n_{G_k} - \frac{12}{5}$ edges. Since $H_k$ is an $8$-coating of digirth $8$, \Cref{cor:computing_fvs_g_coating_of_digirth_g} gives $\fv(H_k) = \frac{n_{H_k} - \frac{12}{5}}{8 - \frac{8}{5}}$.

By \Cref{thm:ratio_extension_to_higher_digirth}, for every $g \ge 8$ this construction extends to an infinite family $\mathcal{F}_g$ of planar digraphs of digirth $g$ with $\fv(H) = \frac{n-\frac{12}{5}}{g - \frac{8}{5}}$ for every $H \in \mathcal{F}_g$.

The proofs of \Cref{thm:general_lower_bound_small_digirth}\ref{itm:digirth9}\ref{itm:digirth10}\ref{itm:digirth11} follow the same approach and a similar construction as in the proof of \ref{itm:digirth8}. These proofs are given in Appendix~\ref{appendix:small_digirth}.

\subsection{Perfect $g$-coatings}\label{sec:digirth_coating}

If $H$ is a $g$-coating of digirth $g$ of a skeleton $G$, then every other coating of digirth $g$ of $G$ has at least as many vertices as $H$ because each cycle associated to a vertex $v \in V(G)$ has to be of length at least $g$. Therefore to optimize the ratio $\frac{ \fv(H) }{n_H}$ we look for $g$-coatings of digirth $g$. In particular all the cycles associated to faces have to be of length at least $g$. This leads to the following definition.

\begin{definition}\label{def:perfect_g_coating}
    Let $H$ be a coating of an undirected plane graph $G$. We say that $H$ is a \emph{perfect $g$-coating} of $G$ when all directed cycles associated to vertices and faces of $G$ have length exactly $g$. In other words, from \Cref{def:coating} the coating function $h$ associated to $H$ satisfies:
    \begin{displaymath}
        \forall v \in V(G), \quad \lvert C_v \rvert = \deg(v) + \sum_{c \in \mathcal{K}(v)} h(c) = g \qquad \text{and} \qquad \forall f \in F(G), \quad \lvert C_f \rvert = \sum_{c \in \mathcal{K}(f)} h(c) = g
    \end{displaymath}
\end{definition}

Intuitively, the perfect $g$-coatings are the coatings that minimally satisfy the digirth $g$ conditions of \Cref{obs:necessary_condition_digirth_coating}, although this is not a sufficient condition to have digirth $g$. In fact we will prove in \Cref{thm:digirth_perfect_coating} that having digirth $g$ for a perfect $g$-coating is conditioned by some structural properties of the associated skeleton. Perfect $g$-coatings are the coatings that should give the best ratio $\frac{ \fv(H) }{n_H}$ and this will be proved in \Cref{thm:upper_bound_fvs_coatings}.

\begin{proposition}\label{prop:coating_connected_without_loops}
    Let $H$ be a coating of digirth $g$ of some skeleton $G$. If $H$ is a $g$-coating, then $G$ has no loops. Moreover, if $H$ is perfect, then $G$ is connected.
\end{proposition}

\begin{proof}
    Let $H$ be a $g$-coating of $G$ of digirth $g$. If $G$ had a loop $uu$, then the two link arcs associated to this loop in $H$ would form two vertex-disjoint directed cycles that consist only of vertices associated to $u$ (see \Cref{fig:proof_coating_without_loops}). Since $H$ is a $g$-coating, the sum of their lengths would be at most $g$. A contradiction.
    
    \begin{figure}[htbp]
        \centering
        \includegraphics[height=4cm]{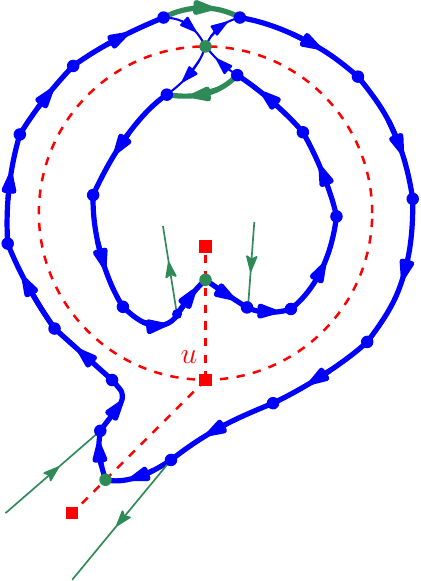}
        \caption{A skeleton (in red) with a loop and two directed cycles of a coating formed by the link arcs associated to this loop (in bold).}\label{fig:proof_coating_without_loops}
    \end{figure}
    
    Suppose now that $H$ is a perfect $g$-coating of digirth $g$. If $G$ is not connected, then it has a face $f \in F(G)$ incident to $k \ge 2$ connected components of $G$. Then the vertices of $H$ that lie inside $f$ (excluding link vertices) form a union of  $k$ directed cycles $C_f^1, \dots, C_f^k$. Then $g = \sum_{c \in \mathcal{K}(f)} h(c) = \lvert C_f \rvert = \sum_{i = 1}^k \lvert C_f^i \rvert \ge 2g$. A contradiction.
\end{proof}

\begin{remark}\label{rem:size_perfect_skeleton}
    If $G$ admits a perfect $g$-coating then $m_G = \frac{2g}{g+2} n_G - \frac{2g}{g+2}$. 
\end{remark}

\begin{proof}
    Let $H$ be a perfect $g$-coating of $G$. Then from \Cref{prop:coating_properties}~\ref{prop:sum_length_cycles} and~\ref{prop:size_coating}:
    \begin{displaymath}
        g \cdot f_G = \sum_{f \in F(G)} \lvert C_f \rvert  = n_H - m_G = g \cdot n_G - 2 m_G
    \end{displaymath}
    From \Cref{prop:coating_connected_without_loops}, $G$ is connected. Then using Euler's formula:
    \begin{displaymath}
        g \cdot (m_G + 2 - n_G) = g \cdot n_G - 2 m_G \qquad \text{and thus} \qquad m_G = \frac{2g}{g+2} n_G - \frac{2g}{g+2}
    \end{displaymath}
\end{proof}

The next result justifies the study of perfect $g$-coatings as they are the coatings with the largest minimum feedback vertex set.

\begin{theorem}\label{thm:upper_bound_fvs_coatings}
    Let $H$ be a coating of digirth $g$ of a skeleton $G$. Then $\fv(H) \le \frac{n_H - \frac{2g}{g+2}}{g- \frac{2g}{g+2}}$. Moreover if $H$ is a perfect $g$-coating of digirth $g$ then the equality holds.
\end{theorem}

\begin{proof}
    First note that if $H$ is perfect, then the equality case comes directly from \Cref{rem:size_perfect_skeleton} and \Cref{cor:computing_fvs_g_coating_of_digirth_g}. Now we show the upper bound. Since $H$ has digirth $g$, from \Cref{lem:coating_digirth_g_normal_sets} we have:
    \begin{displaymath}
            n_H + m_G \ge g \cdot \fv(H). \tag{1}
    \end{displaymath}

    Since $H$ is a coating of digirth $g$, we have $\lvert C_f \rvert \ge g$ for every face $f \in F(G)$. Hence from \Cref{prop:coating_properties}\ref{prop:sum_length_cycles}:
    \begin{displaymath}
        n_H - m_G = \sum_{f \in F(G)} \lvert C_f \rvert \ge g \cdot f_G \ge g \cdot (2 + m_G - n_G).
    \end{displaymath}
    And thus by \Cref{cor:fvs_coating}, we get:
    \begin{displaymath}
        n_H - (g+1)\,m_G \ge 2g  - g\cdot n_G \ge 2g - g\cdot \fv(H).  \tag{2}
    \end{displaymath}
    
    Multiply (1) by $(g+1)$ and add (2) to eliminate the $m_G$ term, we get:
    \begin{displaymath}
        (g+2)\,n_H \ge g^2\,\fv(H) + 2g \qquad \text{hence} \qquad \fv(H) \le \frac{(g+2)\,n_H - 2g}{g^2} = \frac{n_H - \frac{2g}{g+2}}{g - \frac{2g}{g+2}}.
    \end{displaymath}
\end{proof}

\begin{definition}\label{def:fractional_arboricity}
    The \emph{fractional arboricity} of an undirected graph $G$ is defined as follows:
    \begin{displaymath}
        a_f(G) = \underset{\substack{S \text{ subgraph of } G \\ \text{with at least 2 vertices}}}{\max} \frac{m_S}{n_S - 1}.
    \end{displaymath}
\end{definition}

\begin{proposition}\label{prop:bounded_fractional_arboricity}
    If $G$ is a plane undirected graph admitting a $g$-coating of digirth $g$, then $a_f(G) \le \frac{2g}{g+2}$. In particular, every $g$-coating $H$ of digirth $g$ of $G$ satisfies $\fv(H) \ge n_G \ge f_G$.
\end{proposition}

\begin{proof}
    Let $H$ be a $g$-coating of $G$ of digirth $g$. Given a directed cycle $C_f$ associated to a face $f$ of $G$, we have
    $n_H - m_G = \sum_{f \in F(G)} \lvert C_f \rvert$ from \Cref{prop:coating_properties}\ref{prop:sum_length_cycles}. Since $H$ has digirth $g$, and by Euler's formula we get:
    \begin{displaymath}
        n_H - m_G = \sum_{f \in F(G)} \lvert C_f \rvert \ge g \cdot f_G \ge g \cdot (m_G + 2 - n_G).
    \end{displaymath}
    As $H$ is a $g$-coating, we have $n_H = g \times n_G - m_G$ (\Cref{prop:coating_properties}\ref{prop:size_coating}). Then:
    \begin{displaymath}
        g \cdot (m_G + 2 - n_G) \le n_H - m_G = g \cdot n_G - 2 m_G \qquad \text{and hence} \qquad m_G \le \frac{2g}{g+2}(n_G - 1).
    \end{displaymath}
    Now let $S$ be a subgraph of $G$ with at least two vertices. By \Cref{prop:coating_subgraph}, $S$ also admits a $g$-coating of digirth $g$ and thus $m_S \le \frac{2g}{g+2}(n_S - 1)$. Therefore $a_f(G) \le \frac{2g}{g+2}$.
    
    For the second statement, we already know by \Cref{prop:coating_properties}\ref{prop:fvs_greater_than_fG} and \Cref{cor:fvs_coating} that $\fv(H) \ge \max(n_G, f_G)$. Let $G_1 , \dots , G_k$ the connected components of $G$. For $1 \le i \le k$ we have by Euler's formula:
    \begin{displaymath}
        \begin{split}
            f_{G_i} & = m_{G_i} + 2 - n_{G_i} \le \frac{2g}{g+2}(n_{G_i} - 1) + 2 - n_{G_i} = \left( \frac{2g}{g+2} - 1 \right) n_{G_i} + 2 - \frac{2g}{g+2} \\
            &\le \frac{g-2}{g+2}n_{G_i} + \frac{4}{g+2} \le \frac{g-2}{g+2}n_{G_i} + \frac{4}{g+2}n_{G_i} = n_{G_i}.
        \end{split}
    \end{displaymath}
    Then summing the inequality over all connected components we get:
    \begin{displaymath}
            f_G \le f_{G_1} + \cdots + f_{G_k} \le n_{G_1} + \cdots + n_{G_k} = n_G.
    \end{displaymath}
\end{proof}

\begin{remark*}
    \Cref{prop:bounded_fractional_arboricity} also shows that a skeleton that admits a $g$-coating of digirth $g$ has to be 3-degenerate. Indeed, $\frac{2g}{g+2}<2$ for $g\ge 3$ and so for every subgraph $S$ of $G$ of at least 2 vertices, we have $\sum_{v \in V(S)} \deg(v) = 2 m_S < 4(n_S - 1)$. If $S$ has only one vertex, recall that $G$ has no loops by \Cref{prop:coating_connected_without_loops}. Therefore, $S$ has at least one vertex of degree strictly less than 4.
\end{remark*}

Now we come to the main result of this subsection that relates the digirth of a perfect $g$-coating with some structural properties of the associated skeleton.

\begin{theorem}\label{thm:digirth_perfect_coating}
    Let $H$ be a perfect $g$-coating of a skeleton $G$. $H$ has digirth $g$ if and only if $G$ is connected, has no loops and satisfies $a_f(G) = \frac{2g}{g+2}$.
\end{theorem}

\begin{proof}
   The forward direction follows immediately from \Cref{prop:coating_connected_without_loops} (connectivity and no loops), \Cref{prop:bounded_fractional_arboricity} ($a_f(G) \le \frac{2g}{g+2}$), and \Cref{rem:size_perfect_skeleton} (the bound is tight for a perfect $g$-coating). It remains to prove the converse.
    
    Suppose $G$ is connected, has no loops and satisfies $a_f(G) = \frac{2g}{g+2}$. Let $D$ be any directed cycle of $H$, and let $D^-$ be the closed region of the plane to the right of $D$ (i.e.\ the interior of $D$ if $D$ is clockwise, and the exterior otherwise). Set $G' = G \cap D^-$  and $H' = H \cap D^-$ (see \Cref{fig:illustration_claim_subcoating} for an example) and denote by $\mathcal{K}$ (resp. $\mathcal{K'}$) the set of corners of $G$ (resp. of $G'$).
    
    \begin{figure}[htbp]
        \centering
        \includegraphics[width=0.8\textwidth]{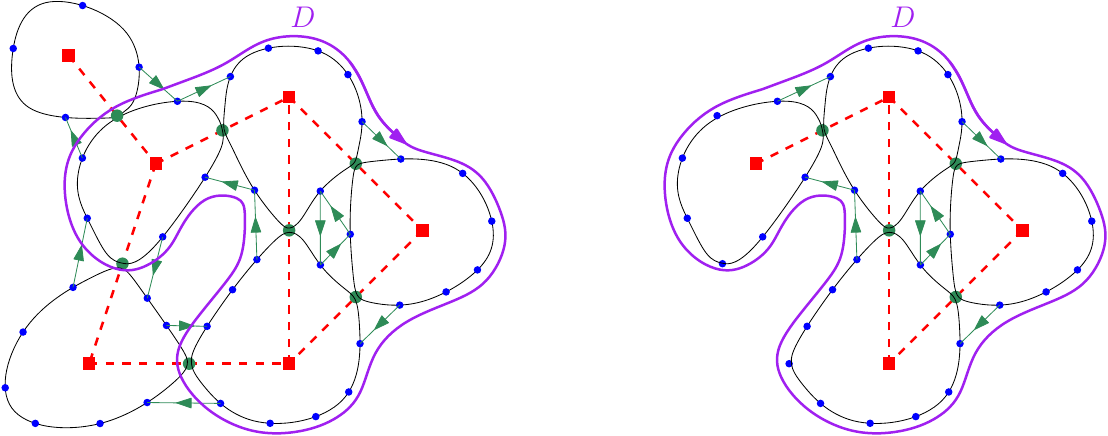}
        \caption{Left: a coating $H$ of a skeleton $G$ and a directed cycle $D$ (in purple) of $H$. Right: The subgraphs $G \cap D^-$ and $H \cap D^-$.}\label{fig:illustration_claim_subcoating}
    \end{figure}

\textbf{$G'$ is connected.} The cycle $D$ corresponds to a closed walk on $G$ via a sequence of corners $(c_1, \dots, c_k)$ incident to vertices $(v_1, \dots, v_\ell)$ of $G'$: consecutive corners either share the same vertex of $G$ (turning around it) or are connected by an edge of $G$ (crossing into $D^-$). Hence all of $v_1, \dots, v_\ell$ lie in one connected component of $G'$. Since $G$ is connected, any path in $G$ between two vertices of $G'$ that exits $D^-$ must exit through some $v_i$ and re-enter through some $v_j$; rerouting it with a path from $v_i$ to $v_j$ staying in $G'$ (as $v_i$ and $v_j$ are in the same connected component of $G'$) yields a path staying in $G'$. Thus $G'$ is connected.
    
\textbf{$H'$ is a coating of $G'$.} For $c \in \mathcal{K}$ and $c' \in \mathcal{K}'$, write $c \subseteq c'$ if the corner $c = (e_1, v, e_2)$ of $G$ lies inside the corner $c' = (e_1', v, e_2')$ of $G'$ (i.e.\ in clockwise order around $v$: $e_1', e_1, e_2, e_2'$). Define $h' : \mathcal{K}' \to \mathbb{N}$ by
        \begin{displaymath}
            h'(c') = \sum_{\substack{c \in \mathcal{K} \\ c \subseteq c'}} h(c) + \bigl(\lvert\{c \in \mathcal{K} : c \subseteq c'\}\rvert - 1\bigr),
        \end{displaymath}
        where the second term counts the link vertices of $H$ absorbed into corner $c'$. One verifies that $H'$ is a coating of $G'$ with coating function $h'$; in particular $h'(c) = h(c)$ for every $c \in \mathcal{K} \cap \mathcal{K}'$.

    Call $f_{ext}$ the outer face of $G'$ bounded by $D$. We call outer corners the corners $c \in \mathcal{K}$ that are included in a corner $c' \in \mathcal{K'}(f_{ext})$ (i.e. the corners of $G$ that are included in a corner of $G'$ incident to $f_{ext}$) and we denote $\mathcal{K}(f_{ext})$ the set of these corners (abusing the notation $\mathcal{K}(f)$ as $f_{ext}$ is not necessarily a face of $G$).
    \begin{displaymath}
        \begin{split}
            \lvert D \rvert &= \sum_{c \in \mathcal{K}(f_{ext})} h(c)  + \sum_{\substack{e \in E(G) \\ e \text{ crosses the cycle } D}} 1  \\
            &= \sum_{v \in V(G')} \sum_{ c \in \mathcal{K}(v)} h(c) - \sum_{v \in V(G')} \sum_{c \in \mathcal{K}(v) \backslash \mathcal{K}(f_{ext})} h(c) + \sum_{v \in V(G')} (\deg_G(v) - \deg_{G'}(v) ) \\
            &= \sum_{v \in V(G')} (\deg_G(v) + \sum_{c \in \mathcal{K}(v)} h(c)) - \sum_{\substack{f \in F(G') \\ f \neq f_{ext}}} \sum_{c \in \mathcal{K}(f)} h(c) - 2m_{G'}.
        \end{split}
    \end{displaymath}
    Note that all the faces of $F(G')$ except $f_{ext}$ are faces of $F(G)$. Then using the hypotheses we obtain:
    \begin{displaymath}
        \begin{split}
            \lvert D \rvert &= \sum_{v \in V(G')} g  - \sum_{\substack{f \in F(G') \\ f \neq f_{ext}}} g - 2m_{G'} = g \times n_{G'} - g \times (f_{G'} - 1) - 2m_{G'} \\
            &= g\big(n_{G'} - (1 + m_{G'} - n_{G'})\big) - 2 m_{G'} \qquad \text{(since $G'$ is planar and connected)} \\
            &= 2g\, n_{G'} - (g+2)\,m_{G'} - g \\
            &\ge 2g\, n_{G'} - (g+2)\,\frac{2g}{g+2}(n_{G'} - 1) - g = g.
        \end{split}
    \end{displaymath}
    Thus every directed cycle of $H$ has length at least $g$, so $H$ has digirth $g$.
\end{proof}


\subsection{A construction of perfect $g$-coatings of digirth $g$}\label{sec:construction_Gkl}
We construct infinite families $\mathcal{F}_g$ of perfect $g$-coatings of digirth $g$ for various values of $g$. By \Cref{thm:upper_bound_fvs_coatings} these graphs would satisfy
\begin{displaymath}
    \forall H \in \mathcal{F}_g, \quad \fv(H) = \frac{n - \frac{2g}{g+2}}{g- \frac{2g}{g+2}} \quad \text{and then} \quad \frac{\fv(H)}{n} \overset{n \to + \infty}{\longrightarrow} \frac{1}{g- \frac{2g}{g+2}}
\end{displaymath}

\Cref{thm:upper_bound_fvs_coatings} also ensures that this is the best bound we can reach using coatings.

Fix $k \ge 1$ and define the family of skeleton graphs $(G_{\ell}^{k})_{ \ell \ge 1}$ as follows:
\begin{itemize}
    \item Define three base blocks $S$, $B$ and $A^k$. The block $S$, also called \emph{the starting block}, is the cube (see \Cref{fig:BlockS_construction_Gkl}). The block $B$ is a cube where we subdivide two edges and add one degree 2 vertex as illustrated in \Cref{fig:BlockB_construction_Gkl}. The block $A^k$ is composed of $k-1$ cubes stacked one inside another (see \Cref{fig:BlockA_construction_Gkl}), where the vertices of the outer face of the $i^{th}$ cube coincide with the vertices of the inner face of the ${(i+1)}^{th}$ cube: we say that these vertices form \emph{the layer $i$}.
    
    \item $G_{\ell}^{k}$ is built by stacking one inside the other multiple base blocks as illustrated in \Cref{fig:Assembly_Gkl}. Inside the central 4-face of the block $S$, stack one inside the other $\ell$ copies $A_1^k,\ldots,A_l^k$ of the block $A^k$. Between each two consecutive blocks $A_i^k$, $A_{i+1}^k$, insert a copy $B_i$ of the block $B$. The vertices of layer $k-1$ of the block $S$ are identified with vertices of layer $k-1$ of the block $A_1^k$. The vertices of layer 0 of the block $A_i^k$ are identified with vertices of layer 0 of the block $B_i$. The vertices of layer $k-1$ of the block $B_i$ are identified with vertices of layer $k-1$ of the block $A_{i+1}^k$.
\end{itemize}

\begin{figure}[htbp]
    \centering
    \begin{subfigure}[b]{0.49\textwidth}
        \centering
        \includegraphics[scale=0.75]{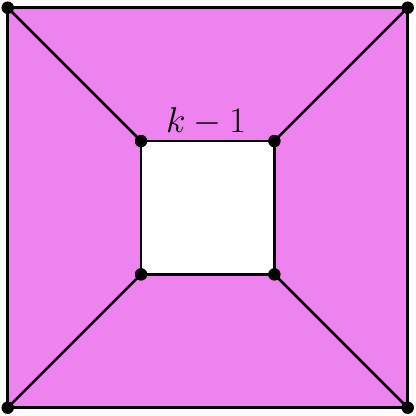}
        \caption{Block $S$ with one layer: $k-1$.}\label{fig:BlockS_construction_Gkl}
    \end{subfigure}
    \hfill
    \begin{subfigure}[b]{0.49\textwidth}
        \centering
        \includegraphics[scale=0.64]{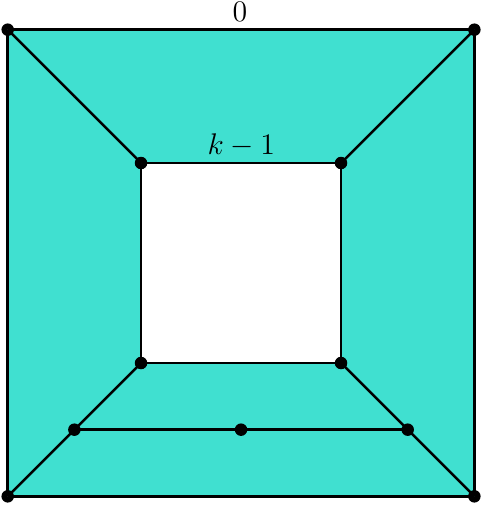}
        \caption{Block $B$ with two layers: 0 and $k-1$.}\label{fig:BlockB_construction_Gkl}
    \end{subfigure}
    
    \bigskip
    
    \begin{subfigure}[b]{0.49\textwidth}
        \centering
        \includegraphics[scale=0.55]{Figures/construction_Gkl_blockA}
        \caption{Block $A^k$ with $k$ layers: $\{0,\ldots,k-1\}$.}\label{fig:BlockA_construction_Gkl}
    \end{subfigure}
    \hfill
    \begin{subfigure}[b]{0.49\textwidth}
        \centering
        \includegraphics[scale=0.35]{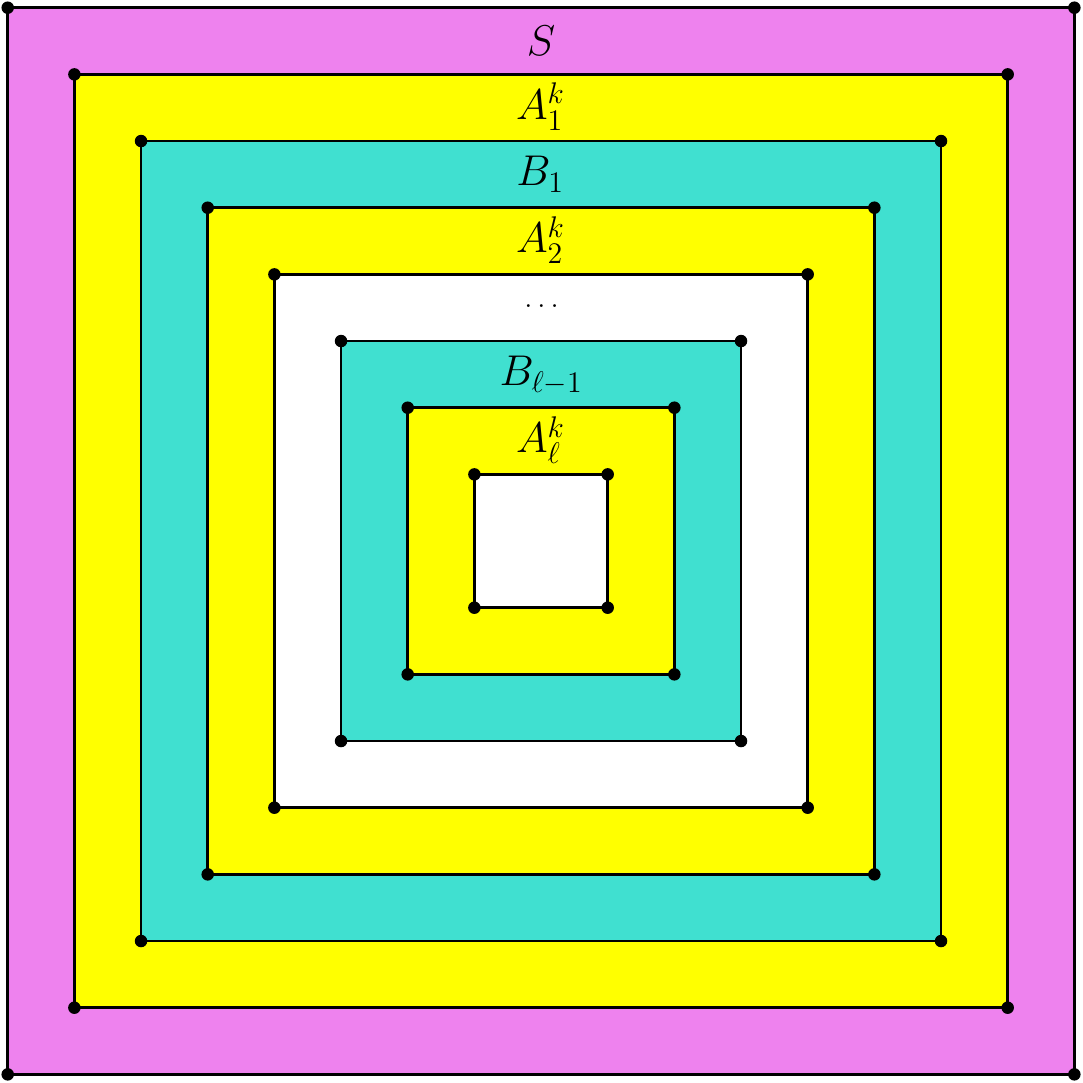}
        \caption{Assembly of the base blocks to build $G_{\ell}^{k}$.}\label{fig:Assembly_Gkl}
    \end{subfigure}
    \caption{Base blocks and construction of the family $(G_{\ell}^{k})_{\ell \ge 1}$.}\label{fig:Construction_Gkl}
\end{figure}

\begin{remark}\label{rem:size_Gkl}
    In each block of type $A^k$ there are $4k$ vertices and $8k-4$ edges. Each block of type $B$ adds $3$ vertices and $8$ edges. The starting block $S$ adds $4$ vertices and $8$ edges. In $G_{\ell}^{k}$ there are $\ell$ blocks of type $A^k$, $\ell - 1$ blocks of type $B$ and one starting block. Hence for $\ell \ge 1$ we have:
    \begin{displaymath}
        \lvert V(G_{\ell}^{k})\rvert = \ell(4k+3) +1 \quad \text{and}\quad \lvert E(G_{\ell}^{k})\rvert = \ell(8k + 4)
    \end{displaymath}
    In particular, by \Cref{def:fractional_arboricity} $a_f(G_{\ell}^{k}) \ge \frac{\lvert E(G_{\ell}^{k})\rvert}{ \lvert V(G_{\ell}^{k})\rvert - 1} = \frac{8k +4}{4k+3} = \frac{2\cdot (8k +4)}{(8k+4)+2}$.
\end{remark}

\begin{proposition}\label{prop:perfect_coating_Gkl}
    Fix $k \ge 1$ and set $g = 8k + 4$. For every $\ell \ge 1$, the skeleton graph $G_{\ell}^{k}$ admits a perfect $g$-coating.
\end{proposition}

\begin{proof}
    We provide the coating function $h$ for the copies of blocks $A^k$ and $B$ and for the remaining corners of $G_{\ell}^{k}$ as described in \Cref{fig:coating_function_Gkl}.
    \begin{figure}[htbp]
        \centering
        \begin{subfigure}[b]{0.48\textwidth}
            \centering
            \includegraphics[scale=0.45]{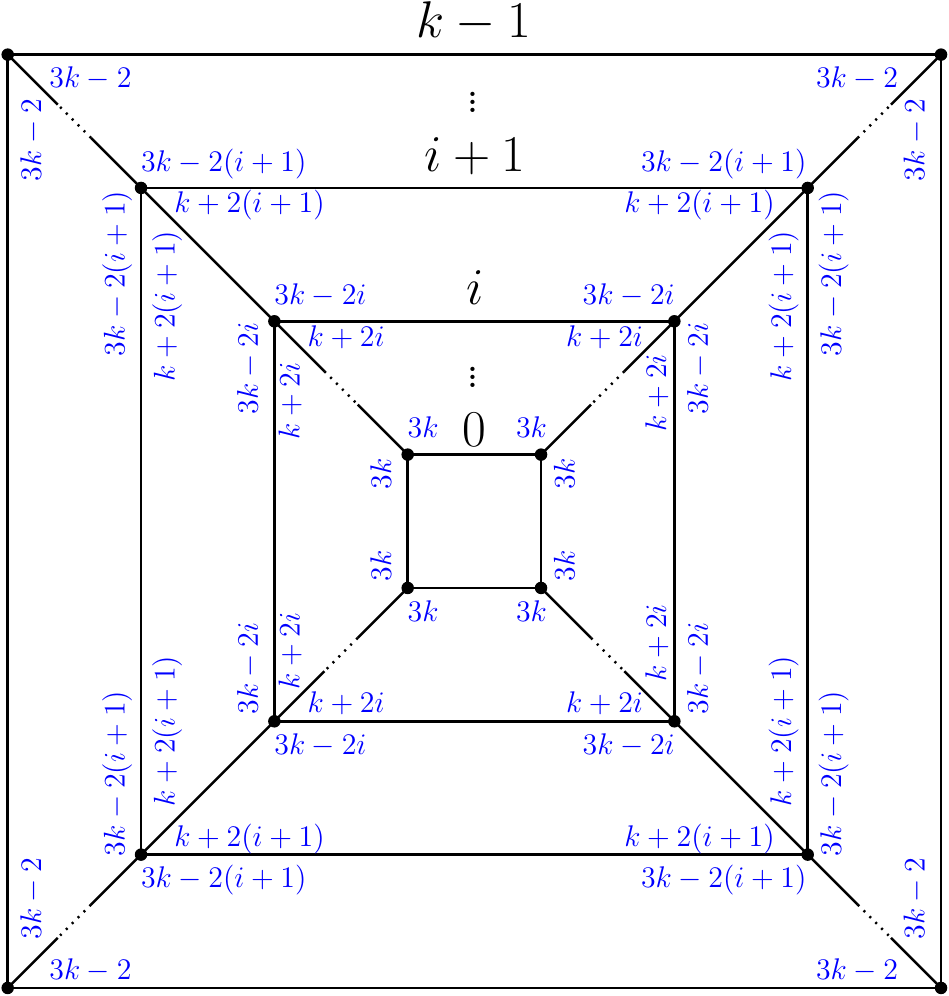}
            \caption{Corners of block $A^k$.}\label{fig:coating_function_blocA}
        \end{subfigure}
        \hfill
        \begin{subfigure}[b]{0.48\textwidth}
            \centering
            \includegraphics[scale=0.5]{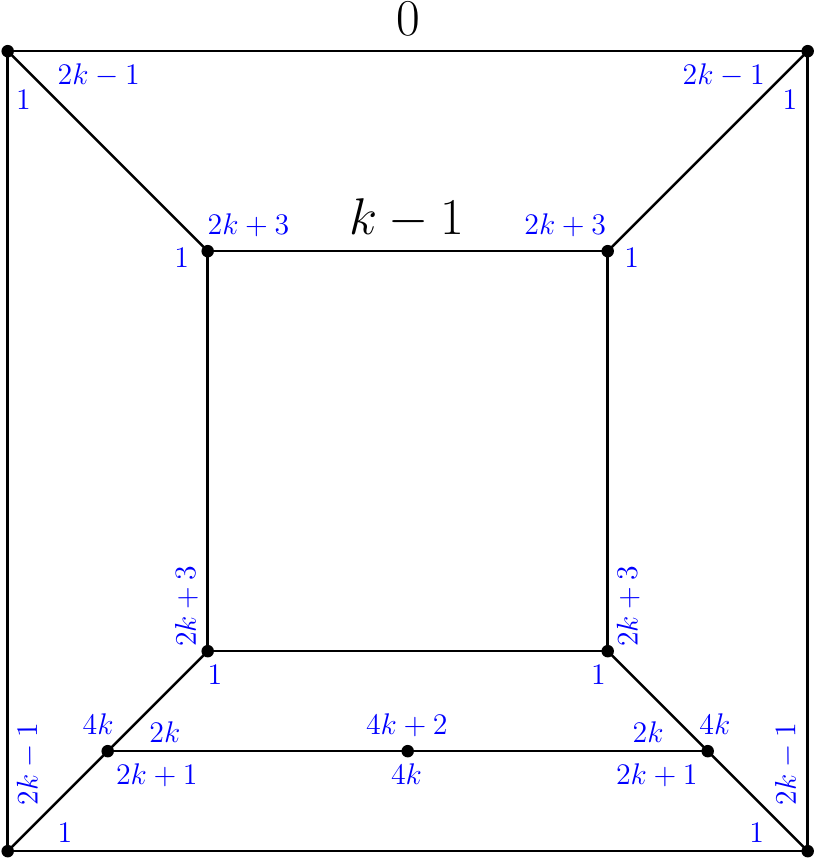}
            \caption{Corners of block $B$.}\label{fig:coating_function_blocB}
        \end{subfigure}
        \begin{subfigure}[b]{0.75\textwidth}
            \centering
            \includegraphics[scale=0.5]{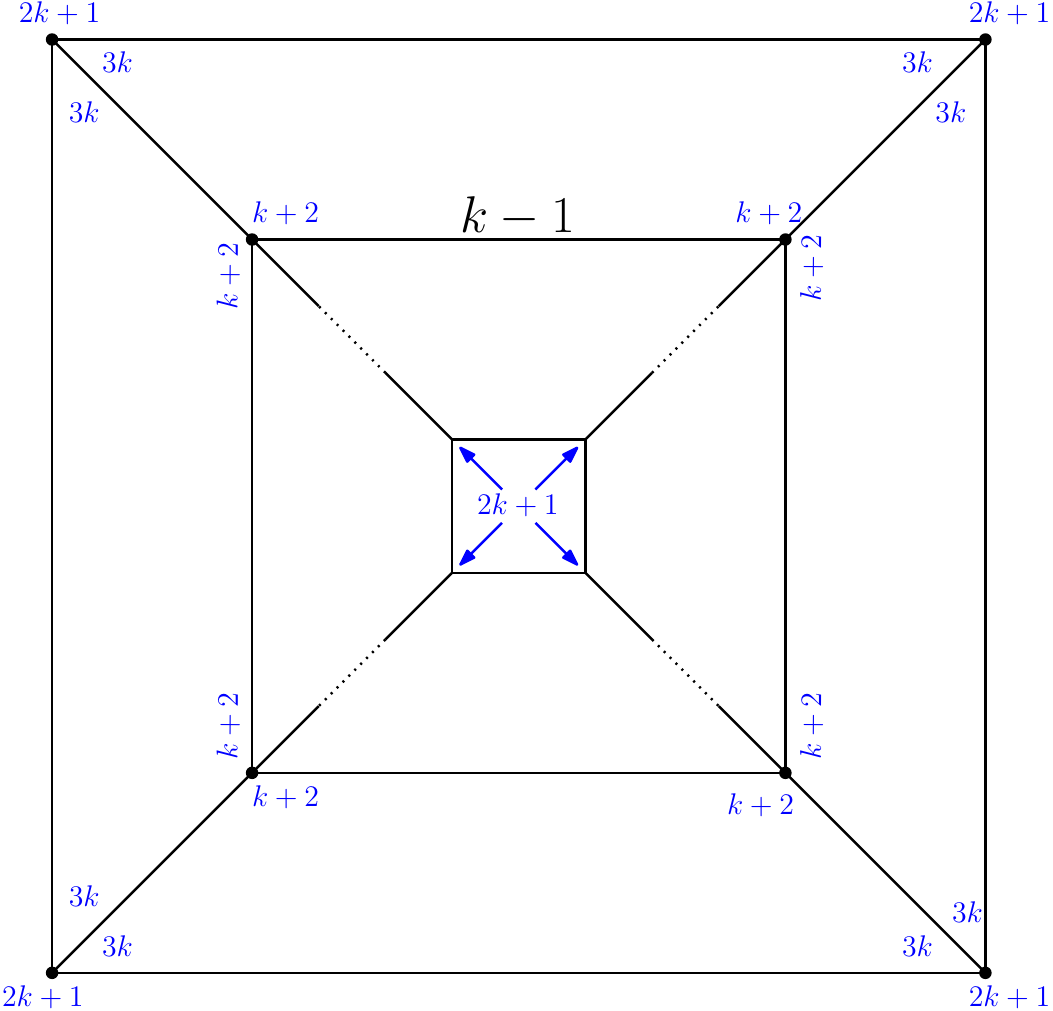}
            \caption{The remaining corners: starting block $S$ and central face of block $A^k_\ell$.}\label{fig:coating_function_outer_inner_faces}
        \end{subfigure}
        \caption{A coating function for $G_{\ell}^{k}$ (in blue) yielding a perfect $g$-coating for $g = 8k+4$.}\label{fig:coating_function_Gkl}
    \end{figure}

    One can verify that this coating function satisfies the conditions of \Cref{def:perfect_g_coating} that is the $h(c)$ around a face sum to $g = 8k+4$ and the $h(c)$ around a vertex $v$ sum to $g - \deg(v)$.
\end{proof}

For the rest of the section we set $g = 8k+4$. We now want to show that the perfect $g$-coating constructed above has digirth $g$. From \Cref{thm:digirth_perfect_coating} it remains to show that $a_f(G_{\ell}^{k})=\frac{2g}{g + 2}$. To compute the fractional arboricity of a planar graph, we use the notion introduced by Scheinerman and Ullman~\cite{SU11}, which we refer to as an \emph{arborization}, following the terminology of Bonamy et al.~\cite{BKKP20}.

\begin{definition}
    Let $G$ be an undirected graph and $c>0$ a constant. A \emph{$c$-arborization} of $G$ is a weighted family of forests $(\mathcal{A}, w)$ with $\mathcal{A}$ a family of forests and $w: \mathcal{A} \to \mathbb{R}^+$ a weight function such that:
    \begin{itemize}
        \item $\forall F \in \mathcal{A}$, we have $E(F) \subset E(G)$ and $F$ is acyclic (forest induced by edges of $G$).
        \item $\forall e \in E(G), \displaystyle \sum_{F\in \mathcal{A},\ \text{s.t. } e \in E(F)} w(F) \ge 1$ (each edge is covered by forests of total weight at least 1).
        \item $\underset{F \in \mathcal{A}}{\sum} w(F) = c$ (the sum of the forest weights equals $c$).
    \end{itemize}
\end{definition}

This definition is motivated by the following result which relates fractional arboricity and arborizations.

\begin{theorem}[Scheinerman and Ullman {\cite[p.~83]{SU11}}]\label{thm:fractional_arboricity}
    Let $G$ be a planar graph. Let $\nu(G)$ be the minimal value of $c$ such that $G$ admits a $c$-arborization. Then $a_f(G) = \nu(G)$.
\end{theorem}

In order to prove that $a_f(G_{\ell}^{k}) \le \frac{2g}{g + 2}$, we construct a $c$-arborization of $G_{\ell}^{k}$ with $c=\frac{16k + 8}{8k + 6}=\frac{2g}{g + 2}$. Then \Cref{rem:size_Gkl} will directly imply that $a_f(G_{\ell}^{k}) = \frac{2g}{g + 2}$.

We color the half-edges of $G_{\ell}^{k}$ with 4 colors (black, red, blue and green) starting with the outer face and the block $S$ and continue on each layer. In the following coloring procedure, coloring a vertex $v$ with a color $c$ means coloring all the half edges incident to $v$ with the color $c$.
\begin{enumerate}
    \item Color the vertices of the outer face starting from the top right vertex with colors blue, red, green and black in this order (see \Cref{subfig:coloring_Gkl_Sblock}). We say that the these vertices form a blue-red-green-black square. The vertices of layer $k-1$ of the block $S$ form a green-black-blue-red square.
    \item Color the layers of the blocks $A^k_i$ (for $i$ from $\ell$ to $1$) by alternating green-black-blue-red squares and blue-red-green-black squares. The layer $k-1$ being already colored, start with layer $k-2$. By doing so, along the diagonals in the Northwest–Southeast direction, blue and green vertices alternate and along the diagonals in the Southwest–Northeast direction, black and red vertices alternate (see \Cref{subfig:coloring_Gkl_Akblock}).
    \item Color the half edges inside each block $B_i$ (for $i$ from $\ell-1$ to $1$) following \Cref{subfig:coloring_Gkl_Bblock}. Note, there are two possible colorings for the block $B_i$ depending whether the layer 0 of the block $A_i$ is a blue-red-green-black square or a green-black-blue-red square.
\end{enumerate}

Moreover, in \Cref{fig:coloring_Gkl} 8 edges of the block $S$ and 8 edges of the block $B$ are labeled with $S_1, \dots, S_8$ (two for each color class). For a color $c$ (out of red, green, black and blue), the two edges with the label written in color $c$ are called \emph{special edges for the color $c$}.

\begin{figure}[htbp]
    \begin{subfigure}[b]{0.48\textwidth}
        \centering
        \includegraphics[scale=0.72]{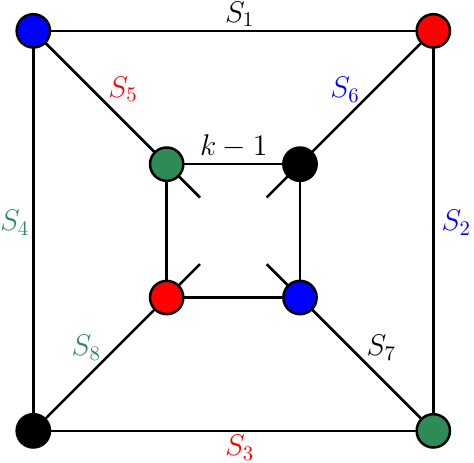}
        \caption{The coloring of the block $S$ of $G^k_{\ell}$.}\label{subfig:coloring_Gkl_Sblock}
    \end{subfigure}
    \begin{subfigure}[b]{0.48\textwidth}
        \centering
        \includegraphics[scale=0.72]{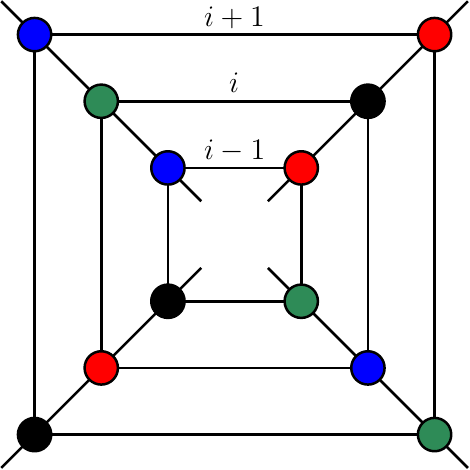}
        \caption{The coloring of the block $A^k$ of $G^k_{\ell}$.}\label{subfig:coloring_Gkl_Akblock}
    \end{subfigure}
    
    \begin{subfigure}{\textwidth}
        \begin{subfigure}{0.48\textwidth}
            \centering
            \includegraphics[scale=0.72]{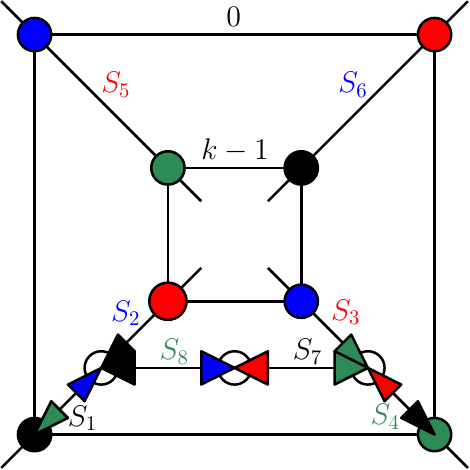}
        \end{subfigure}
        \begin{subfigure}{0.48\textwidth}
            \centering
            \includegraphics[scale=0.72]{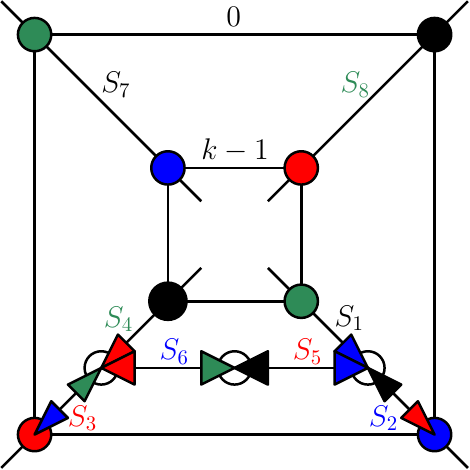}
        \end{subfigure}
        \caption{The coloring of the block $B$ of $G^k_{\ell}$.}\label{subfig:coloring_Gkl_Bblock}
    \end{subfigure}
    \caption{Coloring of the half-edges of $G_{\ell}^{k}$. The special edges of each of the four colors are those labelled $S_i$ for $1 \le i \le 8$. The color of label $S_i$ corresponds to the color the edge is special for.}\label{fig:coloring_Gkl}
\end{figure}

Denote by $E_c$ the set of edges of $G_{\ell}^{k}$ having one half-edge colored $c$ (see \Cref{fig:red_forest_general_case} for an example). Note that the two special edges for the color $c$ are never in $E_c$.

\begin{figure}[htbp]
    \begin{subfigure}[b]{0.48\textwidth}
        \centering
        \includegraphics[scale=0.72]{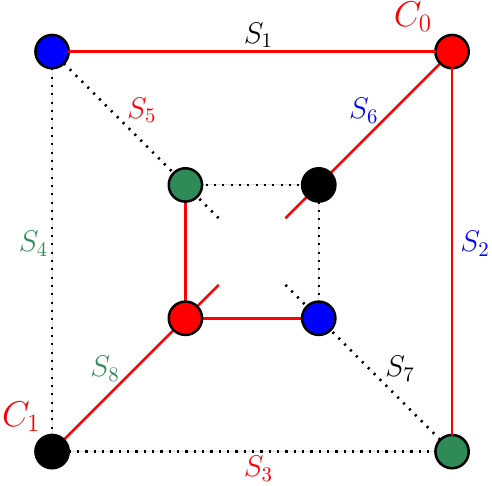}
        \caption{Block $S$}\label{subfig:red_E_forest_Sblock}
    \end{subfigure}
    \begin{subfigure}[b]{0.48\textwidth}
        \centering
        \includegraphics[scale=0.74]{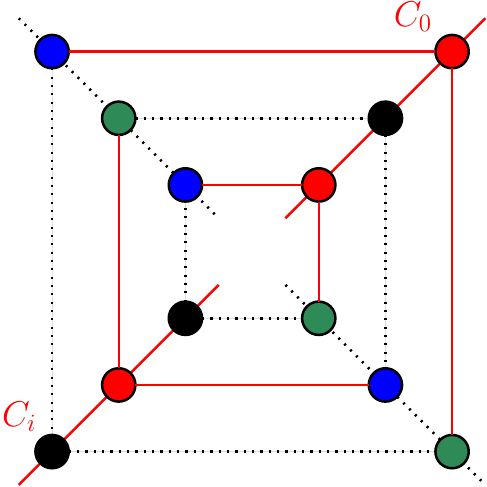}
        \caption{Block $A_i^k$}\label{subfig:red_E_forest_Aikblock}
    \end{subfigure}
    
    \begin{subfigure}{\textwidth}
        \begin{subfigure}{0.48\textwidth}
            \centering
            \includegraphics[scale=0.72]{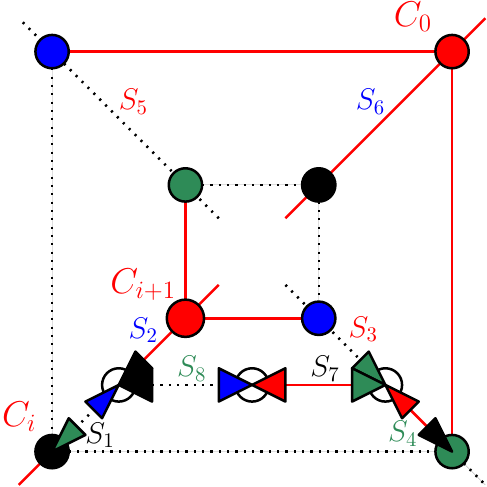}
        \end{subfigure}
        \begin{subfigure}{0.48\textwidth}
            \centering
            \includegraphics[scale=0.72]{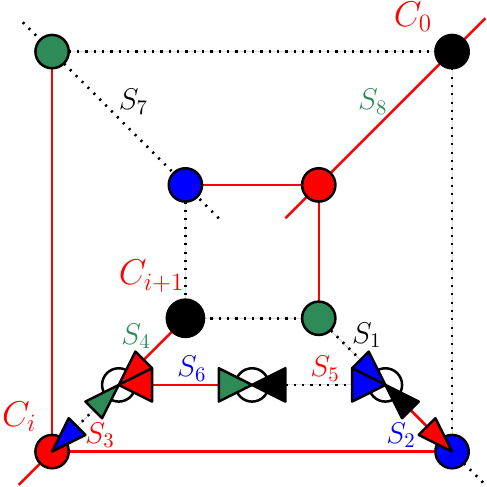}
        \end{subfigure}
        \caption{Block $B_i$}\label{subfig:red_E_forest_Bblock}
    \end{subfigure}
    \caption{Forest $E_{red}$ in the graph $G_{\ell}^{k}$.}\label{fig:red_forest_general_case}
\end{figure}

\begin{observation}\label{obs:properties_Ec}
    \begin{enumerate}[(i)]
        \item\label{obs:size_Ec_Glk} There is no edge of $G_{\ell}^{k}$ whose two half-edges have the same color. Thus every edge of $G_{\ell}^{k}$ belongs to exactly two of the sets $E_c$. Moreover, there are equally many half-edges of each color in $G_{\ell}^{k}$. Hence, from \Cref{rem:size_Gkl}
        \begin{displaymath}
            \lvert E_{black}\rvert = \lvert E_{red}\rvert = \lvert E_{blue}\rvert = \lvert E_{green}\rvert = \frac{\lvert E(G_{\ell}^{k})\rvert}{2} = \ell \cdot (4k +2)
        \end{displaymath}
        \item\label{obs:connected_components_Ec} Let $c$ be a color among red, blue, black, and green. $E_c$ contains no cycles and has exactly $\ell + 1$ connected components. One of these components, denoted $C_0$, consists of the edges located in the North part of $G^k_{\ell}$ and intersects all blocks. Each block $A_i^k$, for $i \in \llbracket 1, \ell \rrbracket$, defines one distinct connected component $C_i$ located in the South part of $G^k_{\ell}$. Moreover:
        \begin{itemize}
            \item Every edge of block $A_i^k$ that is not in $E_c$ connects the component $C_i$ to the component $C_0$.
            \item Each special edge for color $c$ located in block $B_i$ has one endpoint in $C_{i+1}$ and the other in $C_{i}$ or in $C_0$ (in particular, it is not an edge between $C_{i}$ and $C_{0}$).
            \item The special edges for color $c$ located in block $S$ connect the component $C_1$ to the component $C_0$.
        \end{itemize}
        
        \item\label{obs:identification_edge_Gkl} Every edge in $G_\ell^{k}$ can be identified with a unique edge of the block $S$ or $A_1^k$. \begin{itemize}
            \item Every colored block $A_i^k$ is either identical to $A_1^k$ or to its central reflection. So every edge of $A_i^k$ can be identified with a unique edge of $A_1^k$.
            \item A special edge in $B_i$ is identified with the special edge of $S$ that has the same letter in \Cref{fig:coloring_Gkl}.
        \end{itemize} 
    \end{enumerate}
\end{observation}

Using these properties we derive the following procedure to construct large trees contained in $G_{\ell}^{k}$.
\begin{enumerate}
    \item Choose a color $c$ among (black, red, blue, green).
    \item Choose an edge $x$ among the edges of block $A_1^k$ and $S$ that is not in $E_c$. 
    \item Build the tree $T_{c,x}$ by adding to the forest $E_c$ the edge $x$ as well as all edges identified with $x$ (see \Cref{obs:properties_Ec}\ref{obs:identification_edge_Gkl}).
\end{enumerate}
Then $T_{c,x}$ is a tree of $G_{\ell}^{k}$. Indeed, by adding the edge $x$ we connect two connected components of $E_c$. Moreover, by adding all edges identified with $x$ in the other blocks we connect all other connected components of $E_c$ without creating cycles (by \Cref{obs:properties_Ec}\ref{obs:connected_components_Ec}).

This allows to construct a $c$-arborization of $G_{\ell}^{k}$ with $c = \frac{2g}{g+2} = \frac{16k + 8}{8k + 6}$:
\begin{itemize}
    \item Set $w(T_{c,x}) = \frac{1}{8k + 6}$ if $x$ is an edge of block $A_1^k$ not in $E_c$.
    \item Set $w(T_{c,x}) = \frac{2}{8k + 6}$ if $x$ is a special edge for color $c$.
\end{itemize}
Denote by $\mathcal{A}$ the family of weighted trees thus constructed.

\begin{proposition}\label{prop:arborization_Gkl}
    The family $\mathcal{A}$ defined above is a $\frac{16k + 8}{8k + 6}$-arborization of $G_{\ell}^{k}$ and thus $a_f(G_\ell^k) = \frac{2g}{g+2}$ (with $g = 8k+4$).
\end{proposition}

\begin{proof}
    Fix a color $c$. In block $A_1^k$ there are $ \frac{8k-4}{2} = 4k - 2$ edges that are not in $E_c$. In the block $S$ there are 2 special edges for color $c$. Hence the total weight of the trees where the chosen color is $c$ equals:
    \begin{displaymath}
        w(c) = \sum_{x} w(T_{c,x}) =  (4k - 2) \times \frac{1}{8k + 6} + 2 \times \frac{2}{8k + 6} = \frac{4k + 2}{8k + 6}
    \end{displaymath}
    From \Cref{obs:properties_Ec}\ref{obs:size_Ec_Glk}, there are equally many half-edges of each color. Therefore the sum of the weights of the trees in $\mathcal{A}$ is:
    \begin{displaymath}
        w(\mathcal{A}) = \sum_{c} \sum_{x} w(T_{c,x}) = 4 \times \frac{4k + 2}{8k + 6} = \frac{16k + 8}{8k + 6}
    \end{displaymath}
    For $e \in E(G_{\ell}^{k})$, we want to show that $e$ is covered by trees of total weight at least 1.
    \begin{itemize}
        \item All trees $T_{c,x}$ where $c$ is one of the two colors of the half-edges of $e$ contain $e$. This corresponds to trees of total weight $2 \times \frac{4k + 2}{8k + 6} = \frac{8k + 4}{8k + 6}$.
        \item To reach the remaining weight $\frac{2}{8k+6}$, distinguish two cases depending on whether $e$ lies in a block $A_i^k$ or is a special edge.
        \begin{itemize}[label = $\circ$]
            \item If $e$ is in a block $A_i^k$: let $c$ and $c'$ be the two colors that do not appear among the half-edges of $e$ and let $x_e$ be the edge of block $A_1^k$ identified with $e$. Then $e$ is used in $T_{c,x_e}$ and in $T_{c',x_e}$. This contributes to two trees of weight $\frac{1}{8k + 6}$.
            \item If $e$ is a special edge: let $c$ be the color for which $e$ is special and let $x_e$ be the special edge of the starting block identified with $e$. Then $e$ is used in $T_{c,x_e}$. This contributes to one tree of weight $\frac{2}{8k + 6}$.
        \end{itemize}
    \end{itemize}
    Thus $e$ is contained in trees of $\mathcal{A}$ of total weight at least 1. Hence $\mathcal{A}$ is a $\frac{16k + 8}{8k + 6}$-arborization of $G_{\ell}^{k}$.
\end{proof}

Finally we obtain the following result:

\begin{lemma}\label{lem:existence_perfect_g_coating_digirth_g_weak}
    Let $k \ge 1$ and $g = 8k+4$. There exists an infinite family of perfect $g$-coatings of digirth $g$.
\end{lemma}

\begin{proof}
    For all $\ell \ge 1$, the graph $G_\ell^k$ admits a perfect $g$-coating $H_\ell$ (\Cref{prop:perfect_coating_Gkl}). Moreover we proved that $a_f(G_\ell^k) = \frac{2g}{g+2}$ (\Cref{prop:arborization_Gkl}). Then by \Cref{thm:digirth_perfect_coating}, $H_\ell$ has digirth $g$.
\end{proof}

One can extend \Cref{lem:existence_perfect_g_coating_digirth_g_weak} to every even values of $g \ge 12$ with a modification of the skeleton $G_{\ell}^k$.

\begin{lemma}\label{lem:existence_perfect_g_coating_digirth_g}
    For all $g \ge 12$, there exists an infinite family of perfect $g$-coatings of digirth $g$.
\end{lemma}

The proof of \Cref{lem:existence_perfect_g_coating_digirth_g} is similar to what we developed in this section and is applied to an extension of the skeletons $(G_\ell^k)_{\ell \ge 1}$. The proof is given in the Appendix~\ref{appendix:proof_lemma_Gklr}.

\begin{theorem}\label{thm:general_lower_bound_high_digirth}
    For $g \ge 12$ there exists an infinite family of planar digraphs $\mathcal{F}_g$ such that for each $H \in \mathcal{F}_g$, $H$ has digirth $g$ and $\fv(H) = \frac{n - \frac{2g}{g+2}}{g- \frac{2g}{g+2}}$.
\end{theorem}

\begin{proof}
    For $g \ge 12$, from \Cref{lem:existence_perfect_g_coating_digirth_g} there exists an infinite family $\mathcal{F}_g$ of perfect $g$-coatings of digirth $g$. From the equality case of \Cref{thm:upper_bound_fvs_coatings}: $\forall H \in \mathcal{F}_g, \fv(H) = \frac{n - \frac{2g}{g+2}}{g- \frac{2g}{g+2}}$.
\end{proof}

Using \Cref{thm:general_lower_bound_small_digirth} and \Cref{thm:general_lower_bound_high_digirth}, we derive a lower bound for $\tau_g$ for all $g \ge 6$.

\begin{corollary}\label{cor:lower_bound_tau}
    Let $g \ge 6$. We have $\tau_g \ge \frac{g+2}{g^2}$ if $g \neq 7$ and $\tau_7 \ge \frac{1}{7 - \frac{3}{2}} = \frac{2}{11}$.
\end{corollary}

\begin{proof}
    First observe that for $g \in \{ 6,8,9,10,11 \}$, \Cref{thm:general_lower_bound_small_digirth} gives an infinite family $\mathcal{F}_g$ of planar digraphs of digirth $g$ having $\fv(H) = \frac{n - \beta_g}{g - \frac{2g}{g+2}}$ with $\beta_g$ being a constant as:
    \begin{center}
        \renewcommand{\arraystretch}{1.5}
        \begin{tabular}{|c|c|c|c|c|c|}
        \hline
            $g$ & 6 & 8 & 9 & 10 & 11 \\
            \hline
            $\frac{2g}{g+2}$ & $\frac{3}{2}$ & $\frac{8}{5}$ & $\frac{18}{11}$ & $\frac{5}{3}$ & $\frac{22}{13}$ \\
            \hline
        \end{tabular}
        \renewcommand{\arraystretch}{1}
    \end{center}
    
    Then from \Cref{thm:general_lower_bound_high_digirth} (if $g \ge 12$) or \Cref{thm:general_lower_bound_small_digirth} (if $g = 6,8,9,10,11$) there exists an infinite family $\mathcal{F}_g$ of planar digraphs of digirth $g$ such that every $H \in \mathcal{F}_g$ satisfies $\fv(H) = \frac{n - \beta_g}{g - \frac{2g}{g+2}}$ for a fixed constant $\beta_g$. Then for all $H \in \mathcal{F}_g$:
    \begin{displaymath}
        \tau_g \ge \frac{\fv(H)}{n} \overset{n \to +\infty}{\longrightarrow} \frac{1}{g-\frac{2g}{g+2}} = \frac{g+2}{g^2}
    \end{displaymath}
    For the case $g = 7$, from \Cref{thm:general_lower_bound_small_digirth}\ref{itm:digirth6} there exists an infinite family $\mathcal{F}_g$ of planar digraphs of digirth $7$ such that every $H \in \mathcal{F}_g$ has $\fv(H) = \frac{n - 2}{7 - \frac{3}{2}}$. Then $\tau_7 \ge \frac{1}{7-\frac{3}{2}} = \frac{2}{11}$.
\end{proof}

\clearpage
\section{Conclusion}
To obtain an upper bound on $\fv(G)$ for any planar digraph $G$ with digirth $g\geq 3$, we introduced the concept of a normal set of cycles and established the inequalities $\fv(G)\leq \lvert\mathcal{N}_{\max}(G)\rvert\leq \frac{n-2}{g-2}$. On the other hand, \Cref{thm:upper_bound_FVS} suggests that this upper bound might be improved, since equality is attained only when $G$ is a directed cycle.

Let $\mathcal{N}$ be a normal set of size $k = \fv(G)$ that minimizes $q(\mathcal{N}) = m_\mathcal{N} - n_\mathcal{N}$ (\Cref{thm:maximal_normal_set_greater_than_fvs} ensures that it always exists). The proof of \Cref{thm:upper_bound_FVS} shows that whenever a vertex belongs to more than one cycle of $\mathcal{N}$, it follows that $\fv(G[\mathcal{N}]) < k$. Intuitively, if $v$ is a vertex used in $d$ cycles of $\mathcal{N}$ (so $\deg_{G[\mathcal{N}]}(v) = 2d$), one needs at least $2(d-1)$ segments (paths that cut faces of $G[\mathcal{N}]$) to create a normal set $\mathcal{N}'$ of size $k-1$ from the edges of $G[\mathcal{N}] - v$. One could for instance look at a normal set of cycles of size $k-1$ of $G - v$ and ``superimpose it on $\mathcal{N}$'' to try to show this. Continuing in this way for all vertices used in multiple cycles of $\mathcal{N}$, we believe it should be possible to find $\sum_{v \in V(\mathcal{N})} (\deg_{G[\mathcal{N}]}(v) - 2) = 2\, q(\mathcal{N})$ segments. By the same reasoning as at the end of the proof of \Cref{prop:tightness_upper_bound}, we believe it should be possible to show that each segment contributes an energy of at least $\frac{1}{g}$. We thus conjecture the following.

\begin{conjecture}\label{conj:normal_family_energy}
    Let $G$ be a plane digraph of digirth $g$ and $\mathcal{N}$ a normal set of cycles of $G$ of size $\fv(G)$ that minimizes the quantity $q(\mathcal{N}) = m_\mathcal{N} - n_\mathcal{N}$. Then $\mathbf{E_{tot}}(\mathcal{N}) \ge \frac{2}{g}(m_\mathcal{N} - n_\mathcal{N})$.
\end{conjecture}

The constructions of \Cref{sec:construction_Gkl} and Appendix~\ref{appendix:proof_lemma_Gklr} produce planar digraphs with $\fv(G) = \frac{(g+2)n - 2g}{g^2}$ for $g \ge 12$ (\Cref{thm:general_lower_bound_high_digirth}). \Cref{conj:normal_family_energy} would imply that this lower bound is best possible for these values of $g$:

\begin{proposition}\label{prop:energy_conjecture_imply_tightness}
    If \Cref{conj:normal_family_energy} holds, then $\fv(G) \le \frac{(g+2)n - 2g}{g^2}$ for every plane digraph $G$ of digirth $g$.
\end{proposition}

\begin{proof}
    Let $\mathcal{N}$ be a normal set of cycles of size $k = \fv(G)$ such that $\mathbf{E_{tot}}(\mathcal{N}) \ge \frac{2}{g}(m_\mathcal{N} - n_\mathcal{N})$. From \Cref{thm:energy}:
    \begin{displaymath}
        \begin{split}
            \mathbf{E_{tot}}(\mathcal{N}) &= n-2 - (g-2) \lvert \mathcal{N} \rvert \\
            &= (n-n_\mathcal{N}) + (n_\mathcal{N} - m_\mathcal{N}) + m_\mathcal{N} - g \lvert \mathcal{N} \rvert + 2 \lvert \mathcal{N} \rvert - 2 \\
            &= (n - n_\mathcal{N}) + (n_\mathcal{N} - m_\mathcal{N}) + \sum_{C \in \mathcal{N}} ( \lvert C \rvert - g) + 2k - 2
        \end{split}
    \end{displaymath}
    Thus, by hypothesis:
    \begin{displaymath}
        \mathbf{E_{tot}}(\mathcal{N}) + \frac{g}{2} \mathbf{E_{tot}}(\mathcal{N}) \ge (n- n_\mathcal{N}) + \sum_{C \in \mathcal{N}} (\lvert C \rvert - g) + 2k - 2 \ge 2k - 2
    \end{displaymath}
    Then $\mathbf{E_{tot}}(\mathcal{N}) \ge \frac{4k-4}{g+2}$ and therefore $k(g-2) = n-2 - \mathbf{E_{tot}}(\mathcal{N}) \le n-2 - \frac{4k-4}{g+2}$.
    
    Hence $k (g-2 + \frac{4}{g+2}) \le n-2 + \frac{4}{g+2}$, which leads to $k \le \frac{(g+2)n - 2g}{g^2}$.
\end{proof}

Another approach to improving the upper bound on the minimum feedback vertex set would be to bound $q(\mathcal{N})$ from above. We outline our ideas in the following propositions.

\begin{proposition}
    Let $G$ be a plane digraph of digirth $g$ that admits a normal set of cycles $\mathcal{N}$ of size $k=\fv(G)$ such that $q(\mathcal{N}) \le \frac{2g}{g+2}(k-1)$. Then $\fv(G) \le \frac{(g+2)n - 2g}{g^2}$.
\end{proposition}

\begin{proof}
    Using the definition of $\mathbf{E_2}(\mathcal{N})$ and Euler's formula applied to $G[\mathcal{N}]$, we obtain:
    \begin{displaymath}
        \begin{split}
            \mathbf{E_2}(\mathcal{N}) &= \frac{1}{g} \sum_{F \in F(\mathcal{N})} (\ell_F - g) = \frac{1}{g} ( 2 \sum_{C \in \mathcal{N}} \lvert C \rvert - g \cdot f_\mathcal{N}) \\
            &\ge \frac{1}{g} (2gk - g(1 + c_\mathcal{N} + q(\mathcal{N}))) = 2k - c_\mathcal{N} -1 -q(\mathcal{N})
        \end{split}
    \end{displaymath}
    Then, using the hypothesis on $q(\mathcal{N})$, it follows that
    \begin{displaymath}
        \mathbf{E_{tot}}(\mathcal{N}) \ge \mathbf{E_2}(\mathcal{N}) + \mathbf{E_4}(\mathcal{N}) \ge 2k-2-q(\mathcal{N}) \ge \frac{4k-4}{g+2}
    \end{displaymath}
    The end of the proof is identical to the end of the proof of \Cref{prop:energy_conjecture_imply_tightness}.
\end{proof}

For small values of $g$, we believe it should be possible to impose an additional constraint on normal sets. We say that a normal set of cycles $\mathcal{N}$ has \emph{low density} if $2m_\mathcal{N} \le 3n_\mathcal{N}$ or equivalently $q(\mathcal{N}) \le \frac{m_\mathcal{N}}{3}$. We conjecture that the following holds.

\begin{conjecture}\label{conj:low_density}
    Any plane digraph $G$ admits a normal set of cycles with low density of size $\fv(G)$.
\end{conjecture}

If true, \Cref{conj:low_density} would imply another upper bound for the minimum feedback vertex sets, improving \Cref{thm:upper_bound_FVS} for small values of $g$:
\begin{proposition}\label{prop:low_density_fvs}
    Let $G$ be a plane digraph of digirth $g$ that admits a normal set of cycles with low density of size $\fv(G)$. Then $\fv(G) \le \frac{3n}{2g}$.
\end{proposition}

\begin{proof}
    Let $\mathcal{N}$ be a normal set of cycles with low density of size $\fv(G)$. Then
    \begin{displaymath}
        g \cdot \fv(G) \le \sum_{C \in \mathcal{N}} \lvert C \rvert  = m_\mathcal{N} \le \frac{3}{2} n_\mathcal{N} \le \frac{3}{2}n \qquad \text{and hence} \qquad \fv(G) \le \frac{3n}{2g}.
    \end{displaymath}
\end{proof}

Interestingly, the authors of~\cite{ELM17} also observed that the bound $\fv(G)\leq \frac{3n}{2g}$ would be a corollary of a conjecture of Goemans and Williamson~\cite{GW98}. That conjecture asserts that $\fv(G) \le \frac{3}{2}\fv^*(G)$ for every planar digraph $G$, where $\fv^*(G)$ is the fractional relaxation of $\fv(G)$: the minimum value of $\sum_{v \in V(G)} w(v)$ over all weight functions $w \colon V(G) \to [0,1]$ such that $\sum_{v \in C} w(v) \ge 1$ for every directed cycle $C$ of $G$.

For small values of $g$, \Cref{conj:low_density} together with \Cref{prop:low_density_fvs} would directly answer several open questions in the affirmative:
\begin{itemize}
    \item For $g=2$: $\fv(G) \le \frac{3n}{4}$, which would imply the existence of an independent set of size $\frac{n}{4}$ in every undirected planar graph, without using the Four-Color Theorem.
    \item For $g=3$: $\fv(G) \le \frac{n}{2}$, which is Albertson's question (weakening of \Cref{conj:albertson-berman}).
    \item For $g=4$: $\fv(G) \le \frac{3n}{8}$.
    \item For $g=5$: $\fv(G) \le \frac{3n}{10}$.
\end{itemize}
Note that for $g=4$ and $g=5$, the same upper bounds are also conjectured for undirected planar graphs of girth $g$ (see~\cite{AW87,KLS10}). For $g \ge 6$, the bound given by~\Cref{thm:upper_bound_FVS} is better than the one implied by~\Cref{conj:low_density}.

To prove \Cref{conj:low_density}, we note that the strategy used to show \Cref{thm:maximal_normal_set_greater_than_fvs} does not seem straightforward, since the notion of essential vertex and \Cref{lem:essential_vertex} do not readily adapt to normal sets of low density. Specifically, there exist plane digraphs such that no vertex is contained in every maximum normal set of low density. For example, in the digraph of \Cref{subfig:stacked-octahedron}, the size of a maximum normal set of cycles with low density is $4$, and no vertex belongs to every normal set of cycles with low density of size $4$.

To obtain our lower bounds, we developed and used the structure of coatings and skeletons to construct digraphs of digirth $g \ge 12$ with a minimum feedback vertex set of size $\frac{n- \frac{2g}{g+2}}{g-\frac{2g}{g+2}}$ (\Cref{thm:general_lower_bound_high_digirth}), approaching our upper bound of $\frac{n-2}{g-2}$ and proving that $\tau_g \ge \frac{g+2}{g^2}$ for these values of $g$. We also gave constructions for digirth $g \in \{6,8,9,10,11\}$ reaching a similar bound (\Cref{thm:general_lower_bound_small_digirth}) which imply that $\tau_g \ge \frac{g+2}{g^2}$. For $g=7$, we currently only have $\tau_7 \ge \frac{2}{11}$, which is slightly lower than $\frac{9}{49}$ (the value of $\frac{g+2}{g^2}$ at $g=7$). 

However, the concept of coatings and skeletons cannot be used for lower values of digirth. Indeed, \Cref{prop:bound_ratio_for_low_digirth} states that every coating $H$ of digirth $g$ satisfies $\fv(H) \le \frac{4n}{3g}$, which implies that the construction given in \Cref{appl:skeletonCk} is the best we can obtain using coatings for $g = 4$. \Cref{prop:bound_ratio_for_low_digirth} also shows that $\fv(H) \le \frac{n}{g - \frac{1}{2} \left\lfloor \frac{g}{2} \right \rfloor}$ when $H$ is a $g$-coating. This latter bound shows that the coatings given in \Cref{appl:skeletonCk} and \Cref{thm:general_lower_bound_small_digirth}\ref{itm:digirth6} are asymptotically the best $g$-coatings we can construct for $g = 5$ and $g = 7$, respectively. It remains open whether one can improve the lower bounds on $\tau_5$ and $\tau_7$ using coatings that are not $g$-coatings. We believe that solving this question would require a finer understanding of the value of the minimum feedback vertex set of a coating than the lower bound $\fv(H) \ge \max(n_G,f_G)$ of \Cref{rem:lower_bound_max_nG_fG}. It is also natural to ask for which values of $g \le 7$ the lower bound $\tau_g \ge \frac{g+2}{g^2}$ still holds. Since this inequality holds for $g=6$, it might also hold for $g=7$, although it is unclear whether this can be shown using coatings. Finally, for $g = 3$, $\tau_3 \ge \frac{5}{9}$ would be surprising, since it would disprove \Cref{conj:albertson-berman} by a wide margin and would also disprove \Cref{conj:NL,conj:low_density}.

It would also be interesting to look for digraphs with fixed digirth and genus. The notions of coatings and skeletons generalize naturally to graphs embedded on surfaces of higher genus. Let $G$ be an undirected graph embedded on an orientable surface $S$ of genus $\gamma$. Then one can define coatings for this skeleton in the same way as in \Cref{def:coating}. A coating $H$ of $G$ would then be a directed graph embedded on $S$. Most of the results of \Cref{sec:skeleton_and_coating} remain valid. In particular, \Cref{cor:fvs_coating} ensures that $\fv(H) \ge n_G$ and \Cref{rem:edge_contraction_skeleton} provides a useful inductive structure for these objects (deleting a link vertex in $H$ is the same as contracting an edge in $G$). The proof of \Cref{thm:digirth_perfect_coating}, which shows that a perfect $g$-coating of $G$ has digirth $g$, should still hold provided we adapt the definition of fractional arboricity (\Cref{def:fractional_arboricity}) to incorporate the genus $\gamma$ as follows: $a_f^{(\gamma)}(G) = \underset{\substack{G' \text{ subgraph of } G \\ \text{with at least 2 vertices}}}{\max} \frac{m_{G'}}{n_{G'} - (1 - \gamma)}$. In this case, perfect $g$-coatings would have $\fv(H) = \frac{n - \frac{2g}{g+2}(1 - \gamma)}{g - \frac{2g}{g+2}}$. For the upper bound, the concept of normal sets of cycles should also be extended to graphs of genus $\gamma$ to show that $\fv(G) \le \frac{n-2(1-\gamma)}{g-2}$. Define $\tau_g^{(\gamma)}$ analogously to $\tau_g$, but for digraphs of fixed genus $\gamma$. Interestingly, these two results would imply the same lower and upper bounds for $\tau_g^{(\gamma)}$ as in the planar case: $\frac{g+2}{g^2} \le \tau_g^{(\gamma)} \le \frac{1}{g-2}$. This observation raises the question of whether $\tau_g^{(\gamma)} = \tau_g$ for every genus $\gamma$.

\bibliographystyle{plain}
\bibliography{biblio}

\clearpage
\section*{Appendix}
\renewcommand{\thesubsection}{\Alph{subsection}}
\section{Proof of \Cref{thm:general_lower_bound_small_digirth}}\label{appendix:small_digirth}

\subsection{\Cref{thm:general_lower_bound_small_digirth}\ref{itm:digirth9}: case $g=9$}\label{subsubsec:coating_function_digirth_9}

Consider the recursive coating family $(H_k)_{k \ge 0}$ of the skeleton ${(G_k)}_{k \ge 0}$ defined in \Cref{fig:coating_function_digirth_9}. The coating function is defined in blue in the figure.

For every integer $k \ge 0$, \Cref{fig:coating_function_digirth_9} gives a coating function on a skeleton graph $G_k$ that defines a 9-coating $H_k$. 
\begin{figure}[htbp]
    \centering
    \includegraphics[width = 0.8\textwidth]{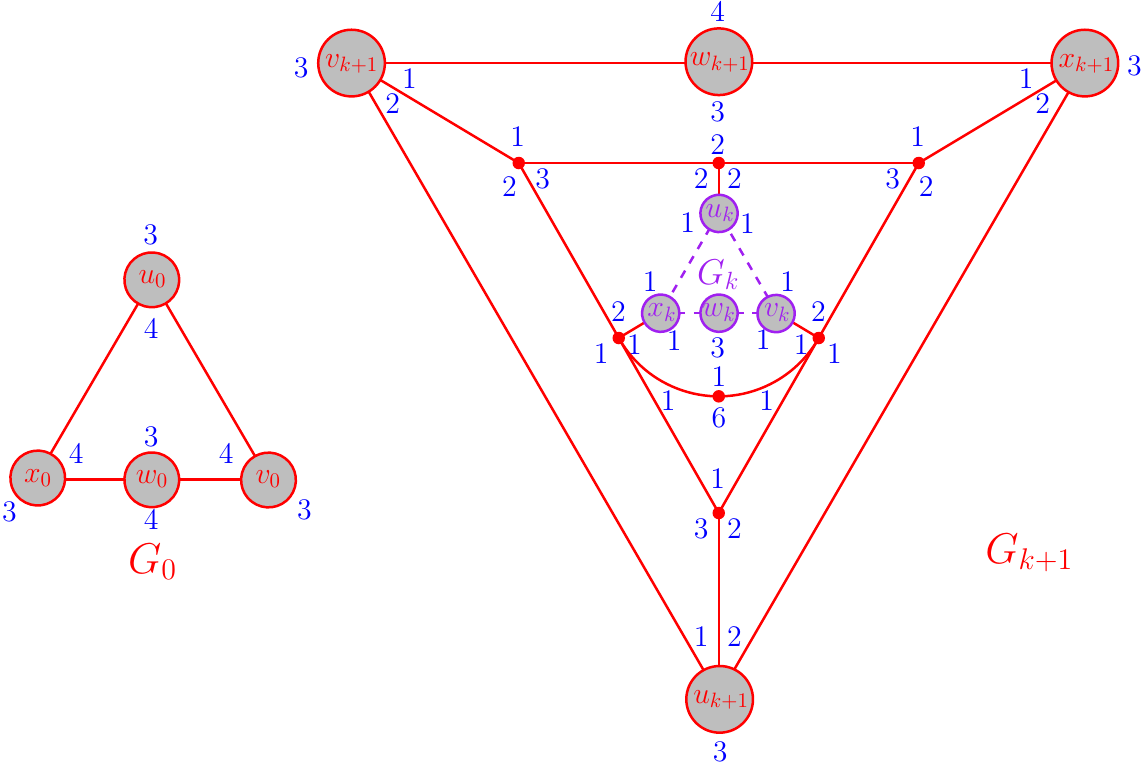}
    \caption{The family $(G_k)_{k \ge 0}$ of \Cref{subsubsec:coating_function_digirth_9} and its coating function in blue.}
    \label{fig:coating_function_digirth_9}
\end{figure}

For all $k \ge 0$, $H_k$ is a 9-coating. To prove that $(H_k)_{k \ge 0}$ is a family of coatings of digirth 9, we use \Cref{lem:digirth_recursive_coating}.

\begin{observation*}
    \begin{enumerate}
        \item $H_0$ and $H_1$ have digirth 9.
        \item For two link vertices $y,z$ associated with edges of the cycle $u_0 v_0 w_0 x_0$, we have $d_{H_0}(y,z) = d_{H_1}(y,z)$.
    \end{enumerate}
\end{observation*}

These observations can be verified by hand or with a computer. Hence we can apply \Cref{lem:digirth_recursive_coating} and $(H_k)_{k \ge 0}$ is a family of coatings of digirth 9.

The skeleton $G_k$ has $n_{G_k} = 4 + 11k$ vertices and $m_{G_k} = 4 + 18k = \frac{18}{11} n_{G_k} - \frac{28}{11}$ edges. As $H_k$ is a 9-coating of digirth 9, from \Cref{cor:computing_fvs_g_coating_of_digirth_g} it follows that $\fv(H_k) = \frac{n_{H_k} - \frac{28}{11}}{9 - \frac{18}{11}}$.

Using \Cref{thm:ratio_extension_to_higher_digirth}, one can extend this construction for every digirth $g \ge 9$ and build an infinite family $\mathcal{F}_g$ of digraphs of digirth $g$ such that every $H \in \mathcal{F}_g$ satisfies $\fv(H) = \frac{n - \frac{28}{11}}{g - \frac{18}{11}}$.

\subsection{\Cref{thm:general_lower_bound_small_digirth}\ref{itm:digirth10}: case $g=10$}\label{subsubsec:coating_function_digirth_10}

Consider the recursive coating family $(H_k)_{k \ge 0}$ of the skeleton ${(G_k)}_{k \ge 0}$ defined in \Cref{fig:coating_function_digirth_10}. The coating function is defined in blue in the figure.

For every integer $k \ge 0$, \Cref{fig:coating_function_digirth_10} gives a coating function on a skeleton graph $G_k$ that defines a 10-coating $H_k$. 
\begin{figure}[htbp]
    \centering
    \includegraphics[width = 0.8\textwidth]{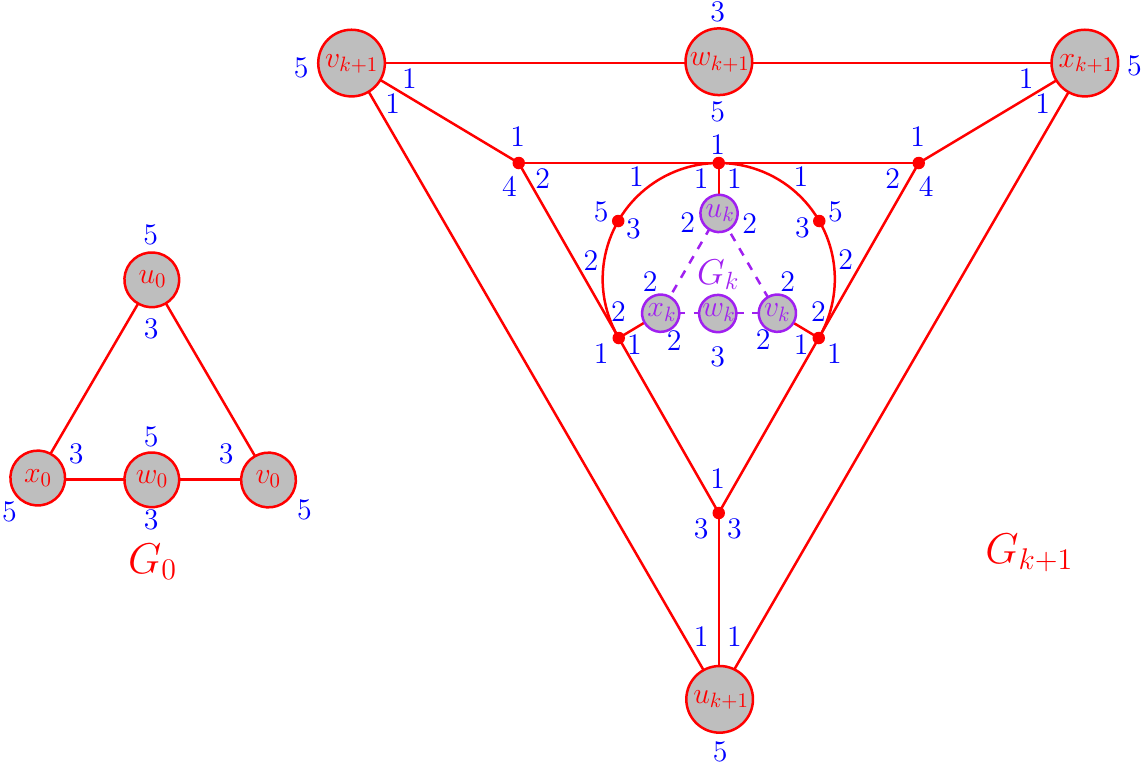}
    \caption{The family $(G_k)_{k \ge 0}$ of \Cref{subsubsec:coating_function_digirth_10} and its coating function in blue.}
    \label{fig:coating_function_digirth_10}
\end{figure}

For all $k \ge 0$, $H_k$ is a 10-coating. To prove that $(H_k)_{k \ge 0}$ is a family of coatings of digirth 10, we use \Cref{lem:digirth_recursive_coating}.

\begin{observation*}
    \begin{enumerate}
        \item $H_0$ and $H_1$ have digirth 10.
        \item For two link vertices $y,z$ associated with edges of the cycle $u_0 v_0 w_0 x_0$, we have $d_{H_0}(y,z) = d_{H_1}(y,z)$.
    \end{enumerate}
\end{observation*}

These observations can be verified by hand or with a computer. Hence we can apply \Cref{lem:digirth_recursive_coating} and $(H_k)_{k \ge 0}$ is a family of coatings of digirth 10.

The skeleton $G_k$ has $n_{G_k} = 4 + 12k$ vertices and $m_{G_k} = 4 + 20k = \frac{5}{3} n_{G_k} - \frac{8}{3}$ edges. As $H_k$ is a 10-coating of digirth 10, from \Cref{cor:computing_fvs_g_coating_of_digirth_g} it follows that $\fv(H_k) = \frac{n_{H_k} - \frac{8}{3}}{10 - \frac{5}{3}}$.

Using \Cref{thm:ratio_extension_to_higher_digirth}, one can extend this construction for every digirth $g \ge 10$ and build an infinite family $\mathcal{F}_g$ of digraphs of digirth $g$ such that every $H \in \mathcal{F}_g$ satisfies $\fv(H) = \frac{n - \frac{8}{3}}{g - \frac{5}{3}}$.

\subsection{\Cref{thm:general_lower_bound_small_digirth}\ref{itm:digirth11}: case $g=11$}\label{subsubsec:coating_function_digirth_11}

Consider the recursive coating family $(H_k)_{k \ge 0}$ of the skeleton ${(G_k)}_{k \ge 0}$ defined in \Cref{fig:coating_function_digirth_11}. The coating function is defined in blue in the figure.

For every integer $k \ge 0$, \Cref{fig:coating_function_digirth_11} gives a coating function on a skeleton graph $G_k$ that defines an 11-coating $H_k$. 
\begin{figure}[htbp]
    \centering
    \includegraphics[width = 0.8\textwidth]{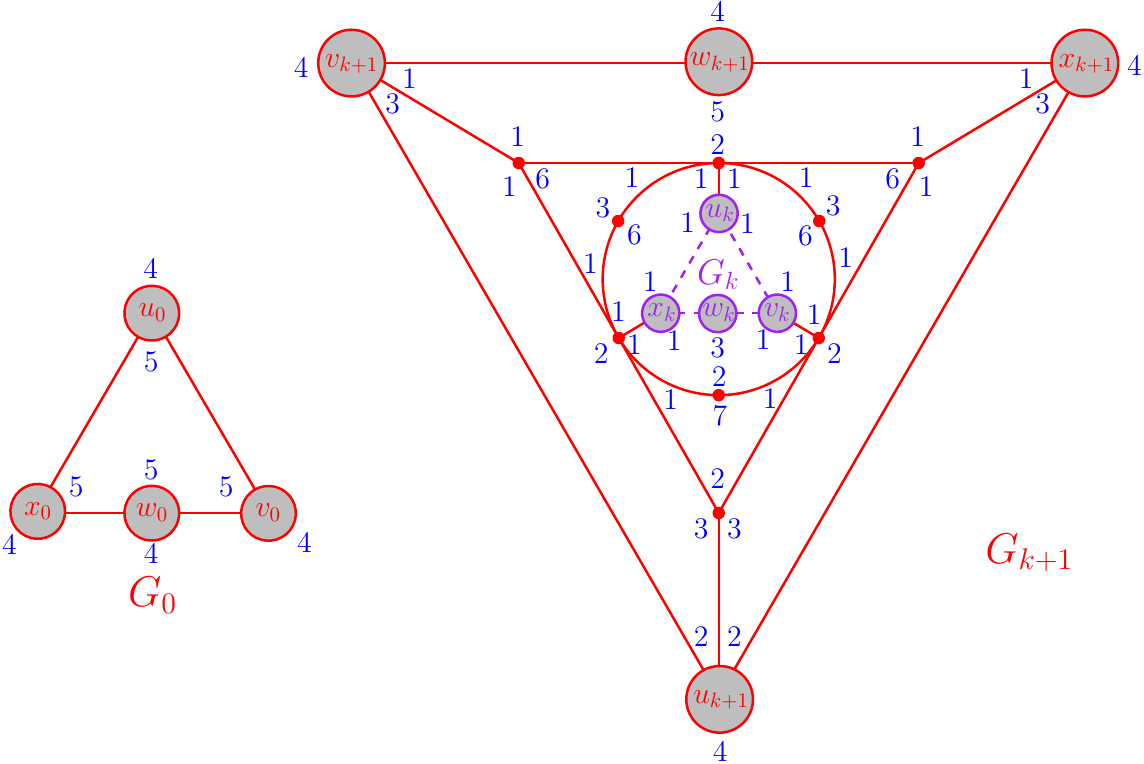}
    \caption{The family $(G_k)_{k \ge 0}$ of \Cref{subsubsec:coating_function_digirth_11} and its coating function in blue.}
    \label{fig:coating_function_digirth_11}
\end{figure}

For all $k \ge 0$, $H_k$ is an 11-coating. To prove that $(H_k)_{k \ge 0}$ is a family of coatings of digirth 11, we use \Cref{lem:digirth_recursive_coating}.

\begin{observation*}
    \begin{enumerate}
        \item $H_0$ and $H_1$ have digirth 11.
        \item For two link vertices $y,z$ associated with edges of the cycle $u_0 v_0 w_0 x_0$, we have $d_{H_0}(y,z) = d_{H_1}(y,z)$.
    \end{enumerate}
\end{observation*}

These observations can be verified by hand or with a computer. Hence we can apply \Cref{lem:digirth_recursive_coating} and $(H_k)_{k \ge 0}$ is a family of coatings of digirth 11.

The skeleton $G_k$ has $n_{G_k} = 4 + 13k$ vertices and $m_{G_k} = 4 + 22k = \frac{22}{13} n_{G_k} - \frac{36}{13}$ edges. As $H_k$ is an 11-coating of digirth 11, from \Cref{cor:computing_fvs_g_coating_of_digirth_g} it follows that $\fv(H_k) = \frac{n_{H_k} - \frac{36}{13}}{11 - \frac{22}{13}}$.

Using \Cref{thm:ratio_extension_to_higher_digirth}, one can extend this construction for every digirth $g \ge 11$ and build an infinite family $\mathcal{F}_g$ of digraphs of digirth $g$ such that every $H \in \mathcal{F}_g$ satisfies $\fv(H) = \frac{n - \frac{36}{13}}{g - \frac{22}{13}}$.


\section{Proof of \Cref{lem:existence_perfect_g_coating_digirth_g}}\label{appendix:proof_lemma_Gklr}

Fix $k \ge 1$ and $r = 2a+b$ with $a \in \lbrace 0,1,2,3 \rbrace$ and $b \in \lbrace 0,1 \rbrace$. Define the family of skeleton graphs $(G_{\ell}^{k,r})_{ \ell \ge 1}$ as follows:
\begin{itemize}
    \item Define the three base blocks $A^k$, $B^x$, and $S^x$. The block $A^k$ is the same as the one used in \Cref{sec:construction_Gkl} (see \Cref{fig:BlockA_construction_Gklr}). The block $B^x$ is obtained from the block $B$ of \Cref{fig:BlockB_construction_Gkl} by adding $x$ vertices as described in \Cref{fig:BlockB_construction_Gklr} (in fact $B^0$ coincides with the block $B$ used in \Cref{sec:construction_Gkl}). The block $S^x$ is obtained from the block $S$ of \Cref{fig:BlockS_construction_Gkl} by adding $x$ vertices as described in \Cref{fig:BlockS_construction_Gklr} (in fact $S^0$ coincides with the block $S$ used in \Cref{sec:construction_Gkl}).
    
    \item We construct $G_{\ell}^{k,r}$ the same way as in \Cref{fig:Construction_Gkl}. Start with a block $S^a$ and inside its central 4-face, stack one inside the other $2\ell$ copies $A_1^k, A_2^k, \dots, A_{2\ell}^k$ of the block $A^k$. Between each two consecutive blocks $A_{2i-1}^k$ and $A_{2i}^k$, insert a copy $B_{2i-1}^{a+b}$ of the block $B^{a+b}$ and between each two consecutive blocks $A_{2i}^k$ and $A_{2i+1}^k$, insert a copy $B_{2i}^{a}$ of the block $B^a$. In terms of blocks, one can write:
    \begin{displaymath}
        G_\ell^{k,r} = (S^a + A^k + B^{a+b} + A^k) + (\ell - 1) \times (B^a + A^k + B^{a+b} + A^k)
    \end{displaymath}
    
    The vertices of layer $k-1$ of the block $S^a$ are identified with vertices of layer $k-1$ of the block $A_1^k$. The vertices of layer 0 of the block $A_{2i}^k$ (resp. $A_{2i-1}^k$) are identified with vertices of layer 0 of the block $B_{2i}^a$ (resp. $B_{2i-1}^{a+b}$). The vertices of layer $k-1$ of the block $B_{2i}^a$ (resp. $B_{2i-1}^{a+b}$) are identified with vertices of layer $k-1$ of the block $A_{2i+1}^k$ (resp. $A_{2i}^k$). The construction is illustrated in \Cref{fig:Assembly_Gklr}.
\end{itemize}

\begin{figure}[htbp]
    \centering
    \begin{subfigure}[b]{0.48\textwidth}
        \centering
        \includegraphics[scale=0.6]{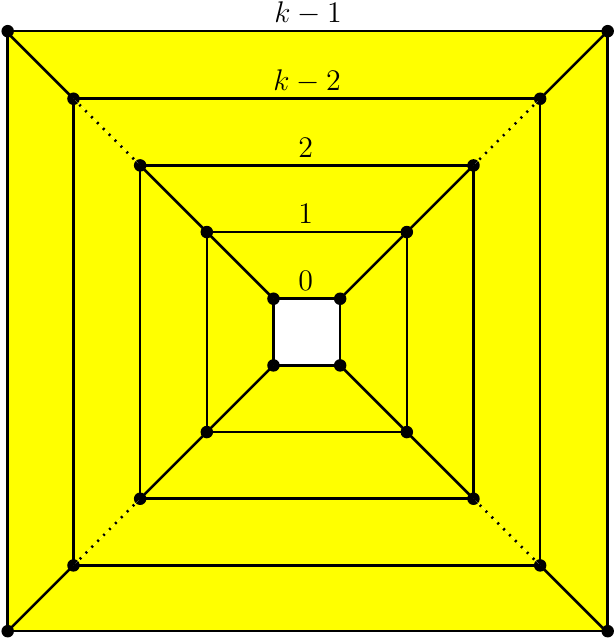}
        \caption{Block $A^k$ with $k$ layers: $\{0,\ldots,k-1\}$.}\label{fig:BlockA_construction_Gklr}
    \end{subfigure}
    \hfill
    \begin{subfigure}[b]{0.48\textwidth}
        \centering
        \includegraphics[scale=0.7]{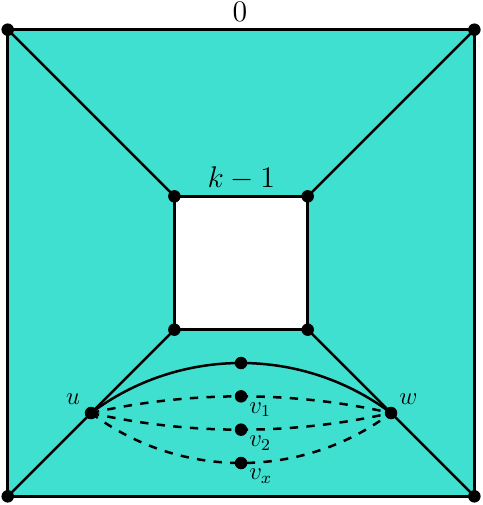}
        \caption{Block $B^x$ with 2 layers: 0 and $k-1$. The vertices $u$ and $w$ have $x+1$ common neighbors.}\label{fig:BlockB_construction_Gklr}
    \end{subfigure}
    \begin{subfigure}[b]{0.48\textwidth}
        \centering
        \includegraphics[scale=0.6]{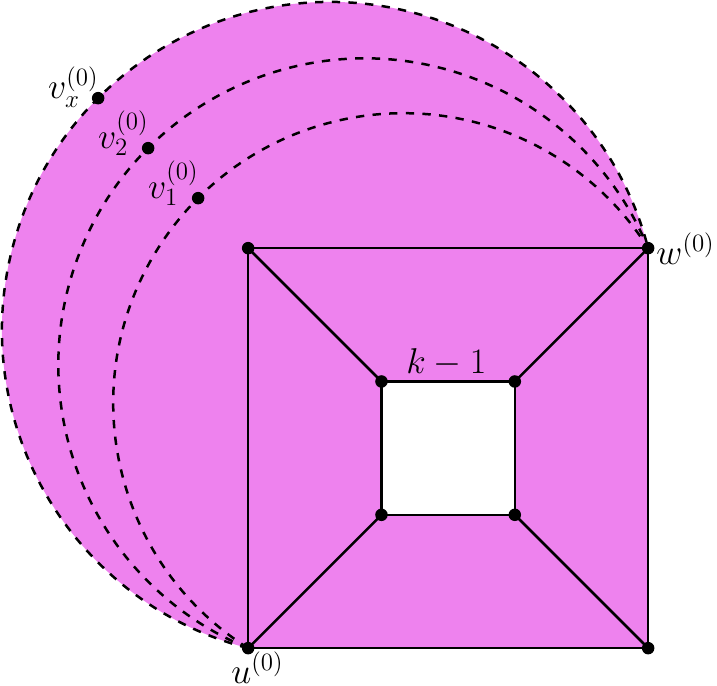}
        \caption{Block $S^x$ with the layer $k-1$. The vertices $u^{(0)}$ and $w^{(0)}$ have $x+2$ common neighbors.}\label{fig:BlockS_construction_Gklr}
    \end{subfigure}
    \hfill
    \begin{subfigure}[b]{0.48\textwidth}
        \centering
        \includegraphics[scale=0.33]{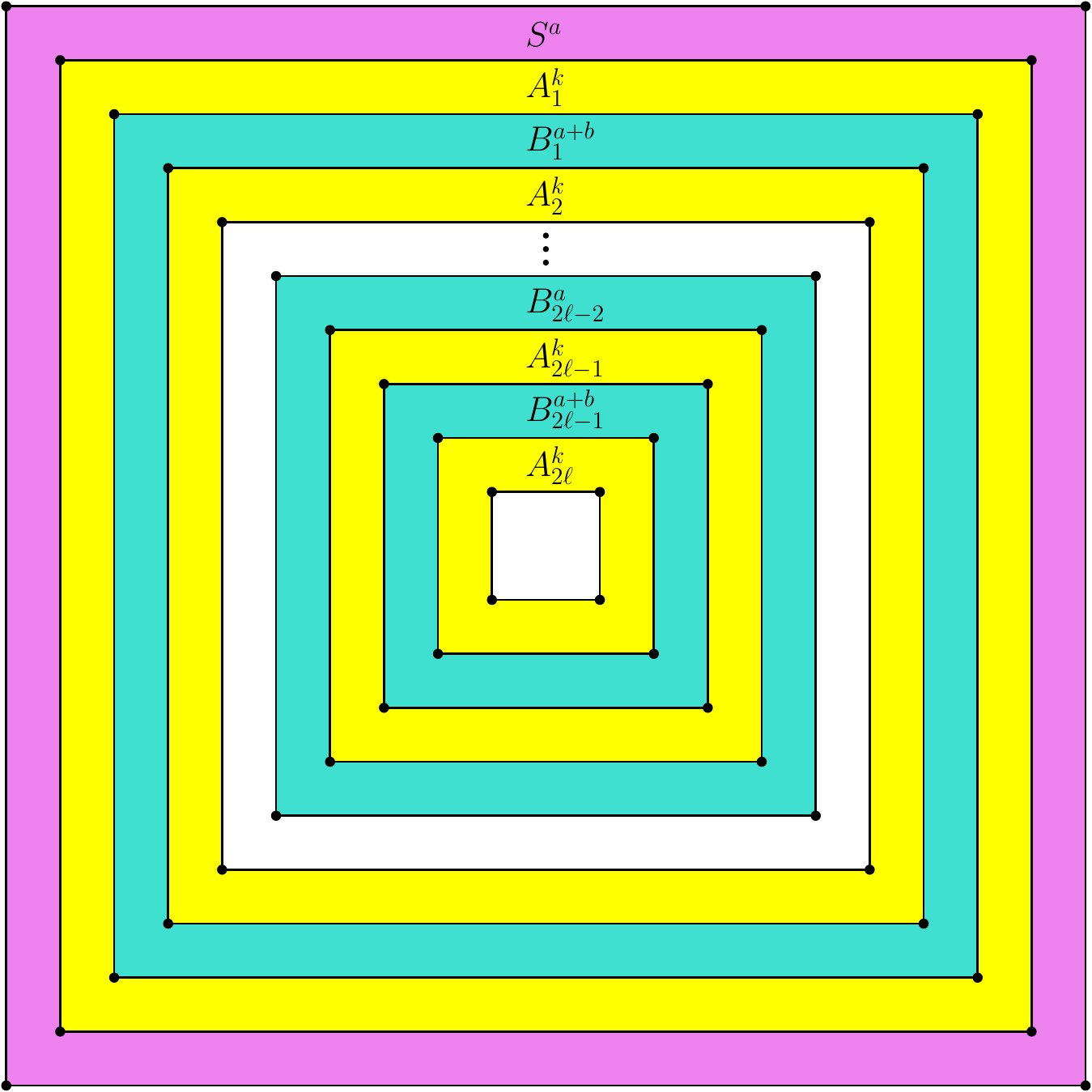}
        \caption{Assembly of the base blocks to build $G_{\ell}^{k,r}$.}\label{fig:Assembly_Gklr}
    \end{subfigure}
    \caption{Base blocks and construction of the family $(G_{\ell}^{k,r})_{\ell \ge 1}$.}\label{fig:Construction_Gklr}
\end{figure}

\begin{remark}\label{rem:size_Gklr}
    In each block of type $A^k$ there are $4k$ vertices and $8k-4$ edges. Each block of type $B^x$ ($x=a$ or $a+b$) adds $3+x$ vertices and $8+2x$ edges. The starting block $S^a$ adds $4+a$ vertices and $8+2a$ edges. In $G_{\ell}^{k,r}$ there are $2\ell$ blocks of type $A^k$, $\ell$ blocks of type $B^{a+b}$, $\ell - 1$ blocks of type $B^a$ and one block $S^a$. Hence for $\ell \ge 1$ we have:
    \begin{displaymath}
        \lvert V(G_{\ell}^{k,r})\rvert = \ell(8k+6+r) +1 \quad \text{and}\quad \lvert E(G_{\ell}^{k,r})\rvert = \ell(16k + 8 +2r)
    \end{displaymath}
    In particular, by \Cref{def:fractional_arboricity}: $a_f(G_{\ell}^{k,r}) \ge \frac{\lvert E(G_{\ell}^{k,r})\rvert}{ \lvert V(G_{\ell}^{k,r})\rvert - 1} = \frac{16k +8 + 2r}{8k+6+r} = \frac{2g}{g +2}$ with $g = 8k+4 + r$.
\end{remark}

\begin{proposition}\label{prop:perfect_coating_Gklr}
    Fix $k \ge 1$ and $r = 2a+b$ with $a \in \lbrace 0,1,2,3 \rbrace$ and $b \in \{0,1\}$. Set $g = 8k + 4 + r$. For every $\ell \ge 1$, the skeleton graph $G_{\ell}^{k,r}$ admits a perfect $g$-coating.
\end{proposition}

\begin{proof}
    We provide the coating function $h$ for the copies of blocks $A^k$, $B^a$ and $B^{a+b}$ and for the remaining corners of $G_{\ell}^{k,r}$ as described in \Cref{fig:coating_function_Gklr}. This coating function is derived from the coating function $h_0$ of \Cref{fig:coating_function_Gkl}. We only add an extra quantity $r$, $r-b$ or $b$ on some corners of $G_{\ell}^{k,r}$.
    \begin{figure}[htbp]
        \centering
        \begin{subfigure}[b]{0.48\textwidth}
            \centering
            \includegraphics[scale=0.4]{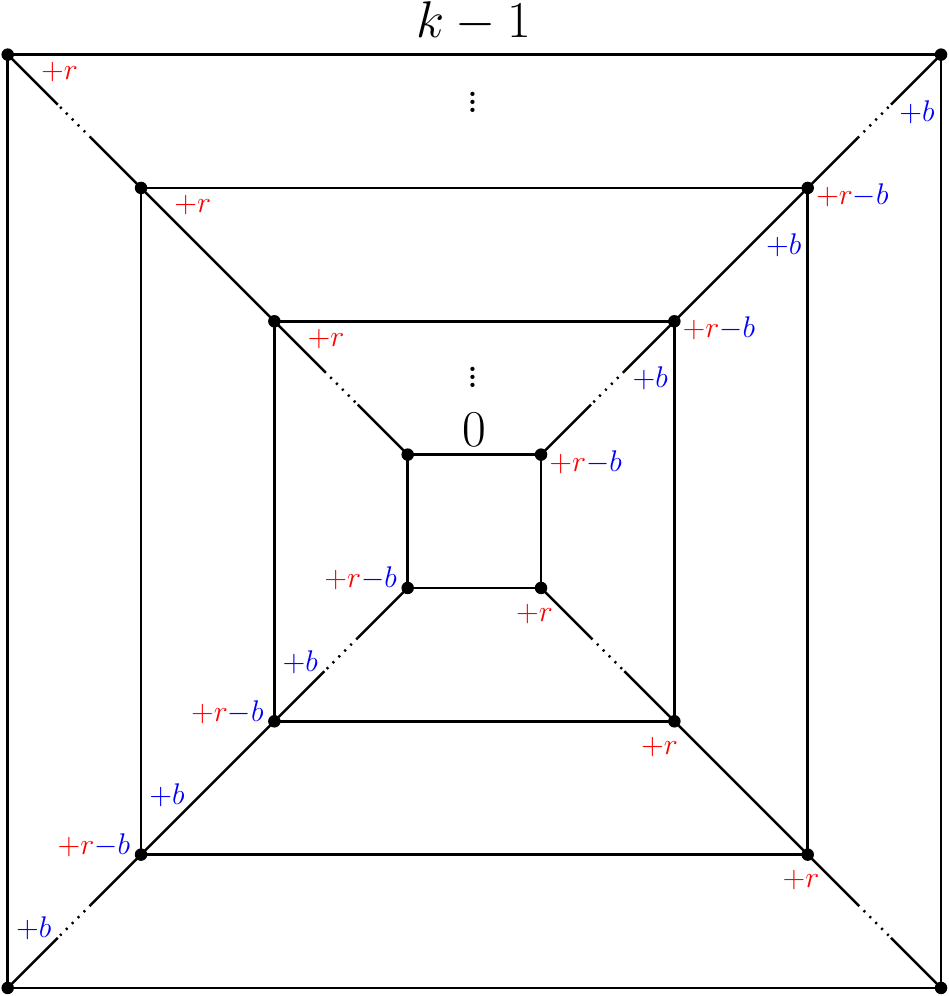}
            \caption{Corners of block $A_{2i-1}^k$.}\label{fig:coating_function_blockA_odd}
        \end{subfigure}
        \hfill
        \begin{subfigure}[b]{0.48\textwidth}
            \centering
            \includegraphics[scale=0.45]{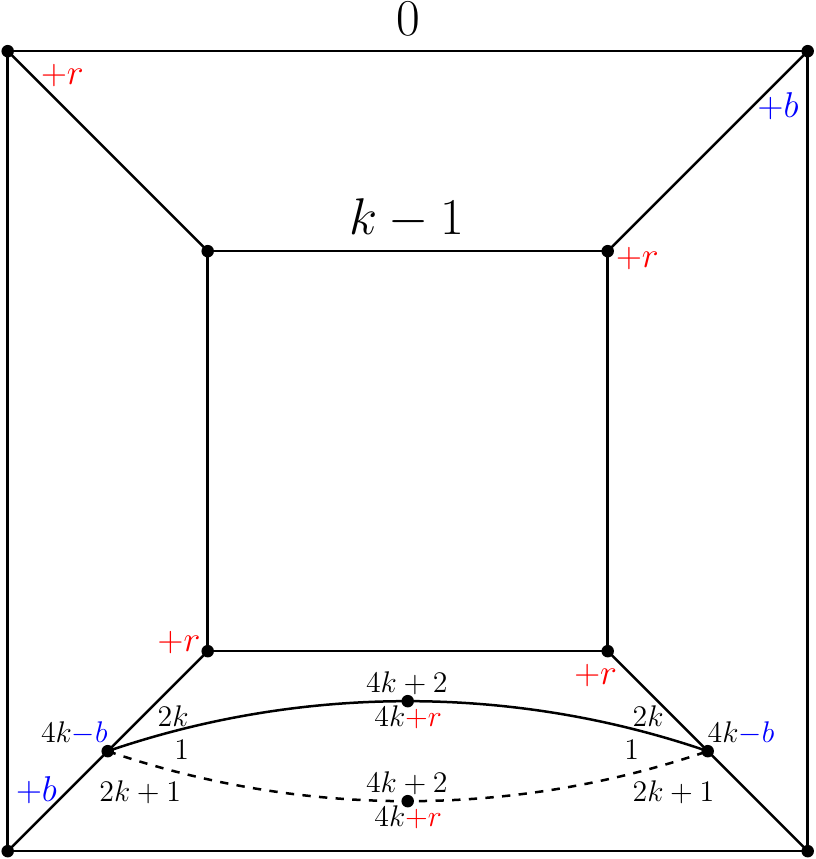}
            \caption{Corners of block $B_{2i-1}^{a+b}$.}\label{fig:coating_function_blockB_a+b}
        \end{subfigure}
        \begin{subfigure}[b]{0.48\textwidth}
            \centering
            \includegraphics[scale=0.4]{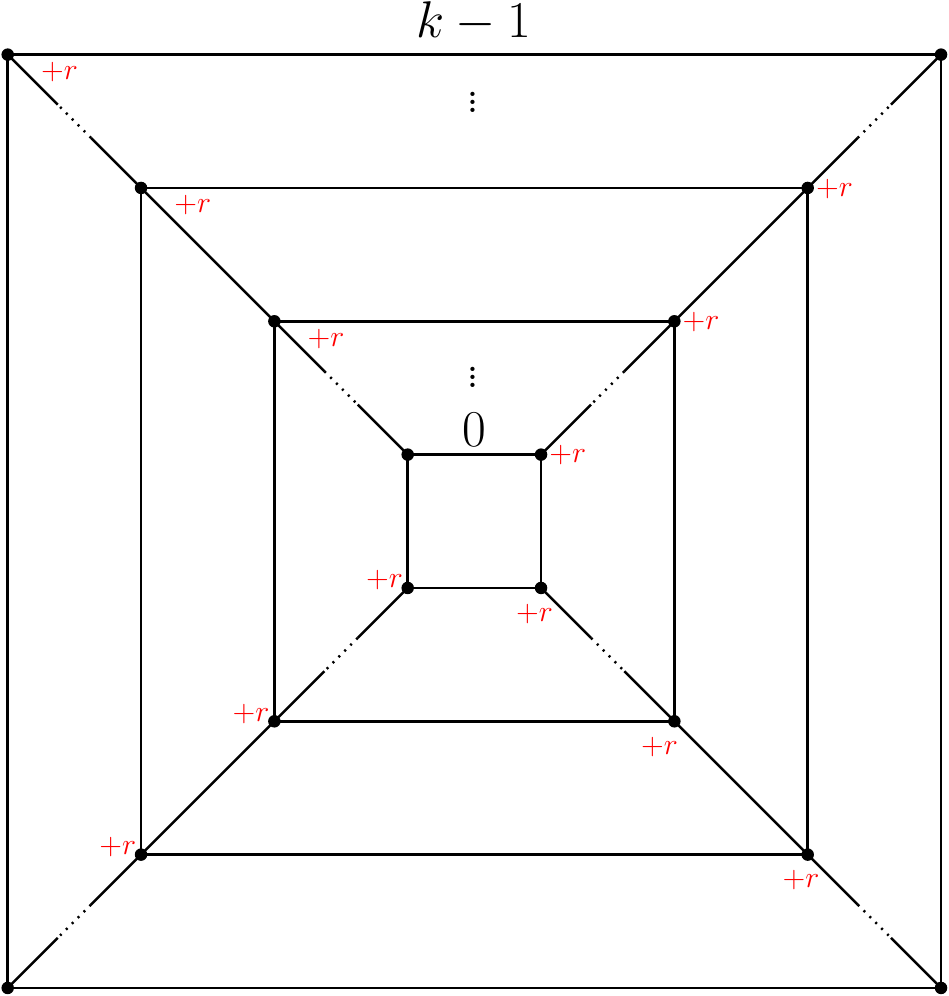}
            \caption{Corners of block $A_{2i}^k$.}\label{fig:coating_function_blockA_even}
        \end{subfigure}
        \hfill
        \begin{subfigure}[b]{0.48\textwidth}
            \centering
            \includegraphics[scale=0.45]{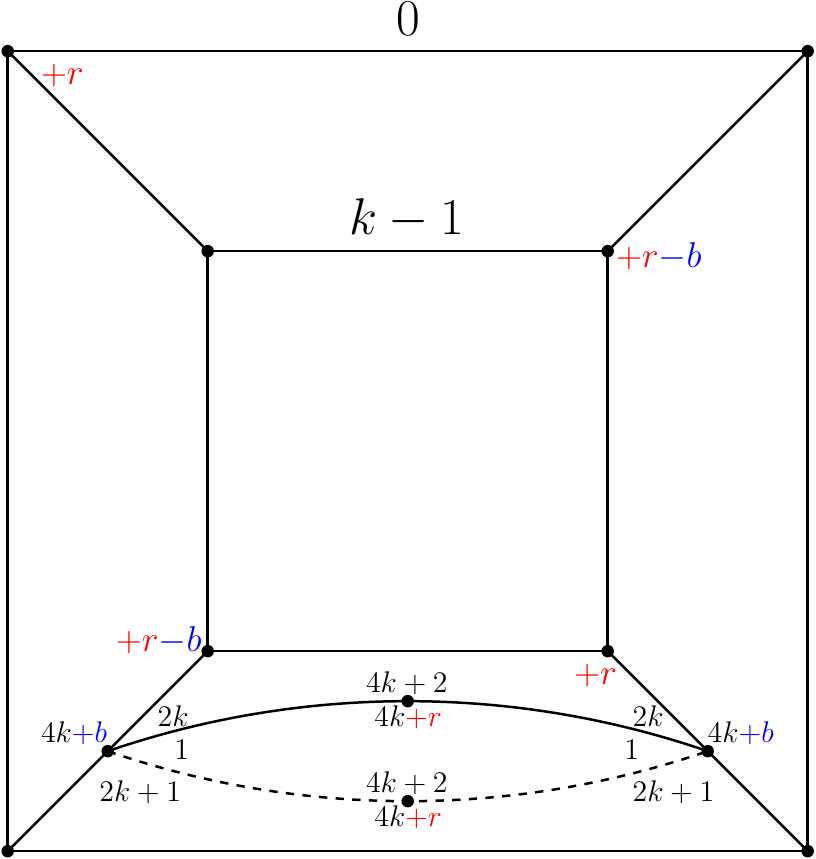}
            \caption{Corners of block $B_{2i}^{a}$.}\label{fig:coating_function_blockB_a}
        \end{subfigure}
        \begin{subfigure}[b]{0.75\textwidth}
            \centering
            \includegraphics[scale=0.4]{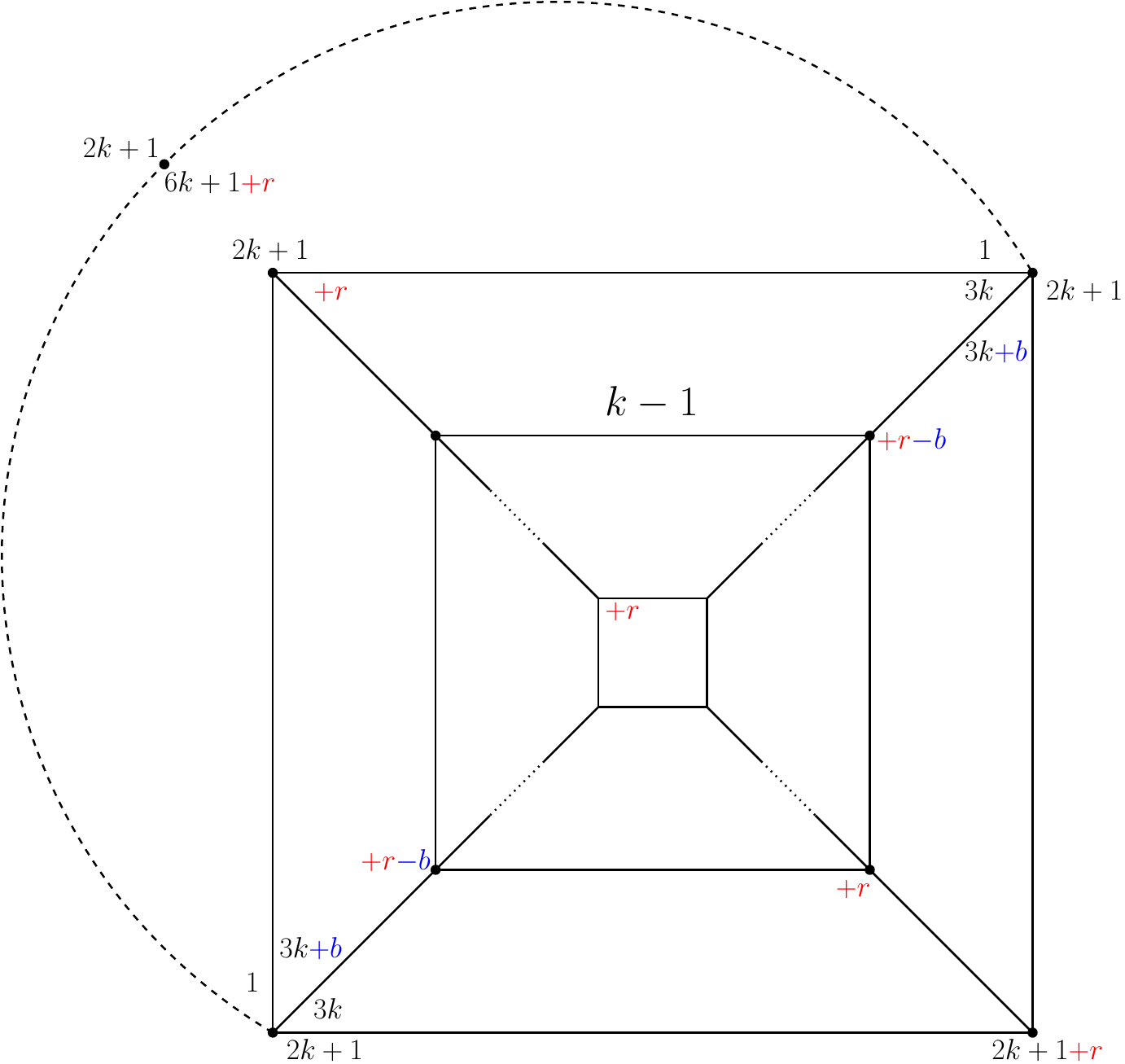}
            \caption{The remaining corners: starting block $S^a$ and central face of block $A^k_{2\ell}$.}\label{fig:coating_function_outer_inner_faces_Gklr}
        \end{subfigure}
        \caption{A coating function $h$ for $G_{\ell}^{k,r}$ yielding a perfect $g$-coating. We only write the coating function around the newly added vertices and edges. For the other corners $h(c)$ is either $h_0(c)$ or $h_0(c) + x$ with $x \in \lbrace r,r-b,b \rbrace$ and $h_0$ being the coating function defined in \Cref{fig:coating_function_Gkl}.}\label{fig:coating_function_Gklr}
    \end{figure}
    One can verify that around each face a quantity of exactly $r$ is added.
    \begin{displaymath}
        \forall f \in F(G_\ell^{k,r}), \quad \sum_{c \in \mathcal{K}(f)} h(c) = r + \sum_{c \in \mathcal{K}(f)} h_0(c) = r +8k+4 = g.
    \end{displaymath}
    Similarly, one can verify that around each vertex a quantity of exactly $r$ is added, except for the vertices $u^{(0)}, w^{(0)}$ in block $S^a$ and $(u^{(i)}, w^{(i)})_{1 \le i \le 2\ell -1}$, which correspond to the $\ell$ copies of vertices $u,w$ in block $B^{a+b}$ ($i$ odd) and the $\ell-1$ copies in block $B^{a}$ ($i$ even). Then:
    \begin{displaymath}
        \forall v \in V(G_\ell^{k,r}) \backslash \{ u^{(i)}, w^{(i)} \mid 0 \le i \le 2\ell -1 \}, \quad \deg(v) + \sum_{c \in \mathcal{K}(v)} h(c) = r + \deg(v) + \sum_{c \in \mathcal{K}(v)} h_0(c) = r + 8k+4 = g
    \end{displaymath}
    Finally note that $\deg(u^{(2i)}) = \deg(w^{(2i)}) = 3+a$ and $\deg(u^{(2i+1)}) = \deg(w^{(2i+1)}) = 3+a+b$ for all $0 \le i \le \ell-1$ and then for $v \in \{ u^{(0)}, w^{(0)} \}$:
    \begin{displaymath}
        \deg(v) + \sum_{c \in \mathcal{K}(v)} h(c) = 3+a + a \times 1 + 3k + (3k+b) + (2k+1) = 8k+4+2a+b = g.
    \end{displaymath}
    For $v \in \{ u^{(2i)}, w^{(2i)} \mid 1 \le i \le \ell -1 \}$:
    \begin{displaymath}
        \deg(v) + \sum_{c \in \mathcal{K}(v)} h(c) = 3+a + a \times 1 + 2k + (4k+b) + (2k+1) = 8k+4+2a+b = g.
    \end{displaymath}
    And for $v \in \{ u^{(2i+1)}, w^{(2i+1)} \mid 0 \le i \le \ell-1 \}$:
    \begin{displaymath}
        \deg(v) + \sum_{c \in \mathcal{K}(v)} h(c) = 3+a+b + (a+b) \times 1 + 2k + (4k-b) + (2k+1) = 8k+4+2a+b = g.
    \end{displaymath}
    Then $h$ defines a perfect $g$-coating.
\end{proof}

For the rest of the section we fix $g = 8k+4+r$ (with $r = 2a+b$).

We now want to show that the perfect $g$-coating constructed above has digirth $g$. From \Cref{thm:digirth_perfect_coating} it remains to show that the fractional arboricity of $G_{\ell}^{k,r}$ is $\frac{2g}{g + 2}$. To prove it we use \Cref{thm:fractional_arboricity} and construct a $\frac{2g}{g+2}$-arborization of $G_{\ell}^{k,r}$, then \Cref{rem:size_Gklr} will directly give that $a_f(G_\ell^{k,r}) = \frac{2g}{g+2}$.

Similarly to what we have done in \Cref{sec:construction_Gkl}, we color the half-edges of $G_{\ell}^{k,r}$ with 4 colors (black, red, blue and green). This coloring is an extension of the coloring of \Cref{fig:coloring_Gkl} where we color the new half-edges as shown in \Cref{fig:coloring_Gklr}. The two special edges for each color are the same as in \Cref{sec:construction_Gkl}.

\begin{figure}[htbp]
    \centering
    \includegraphics[width=\textwidth]{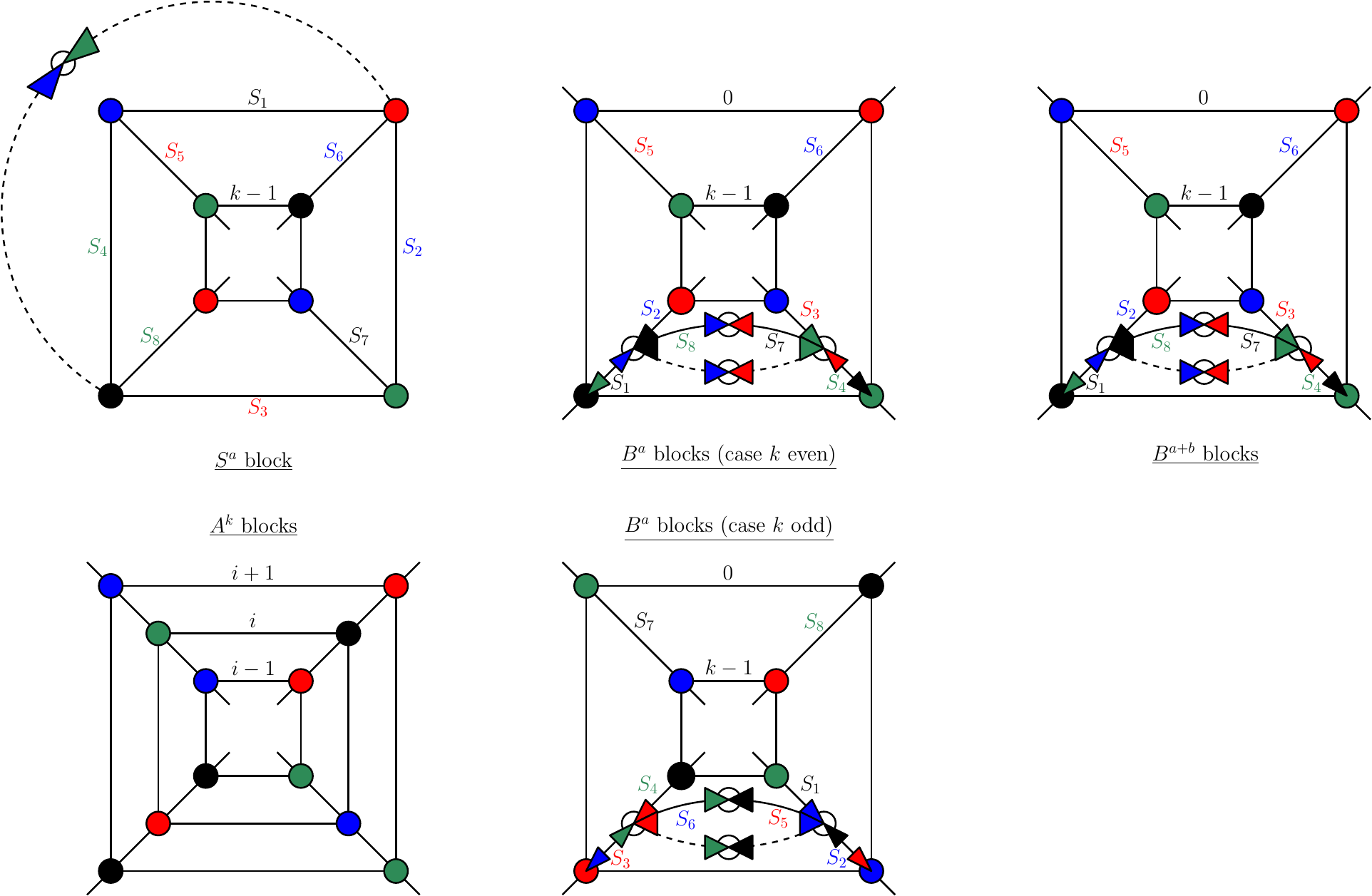}
    \caption{Coloring of the half-edges of $G_{\ell}^{k,r}$. We also give the two special edges for each color in $S^a$ and in the blocks $B^x$ ($x = a$ or $a+b$)}\label{fig:coloring_Gklr}
\end{figure}

For $c$ a color among (black, red, blue, green), denote by $E_c$ the set of edges of $G_{\ell}^{k,r}$ having one half-edge colored $c$ (see \Cref{fig:red_forest} for an example). Note that the two special edges for the color $c$ are never in $E_c$.

\begin{figure}[htbp]
    \centering
    \includegraphics[width=\textwidth]{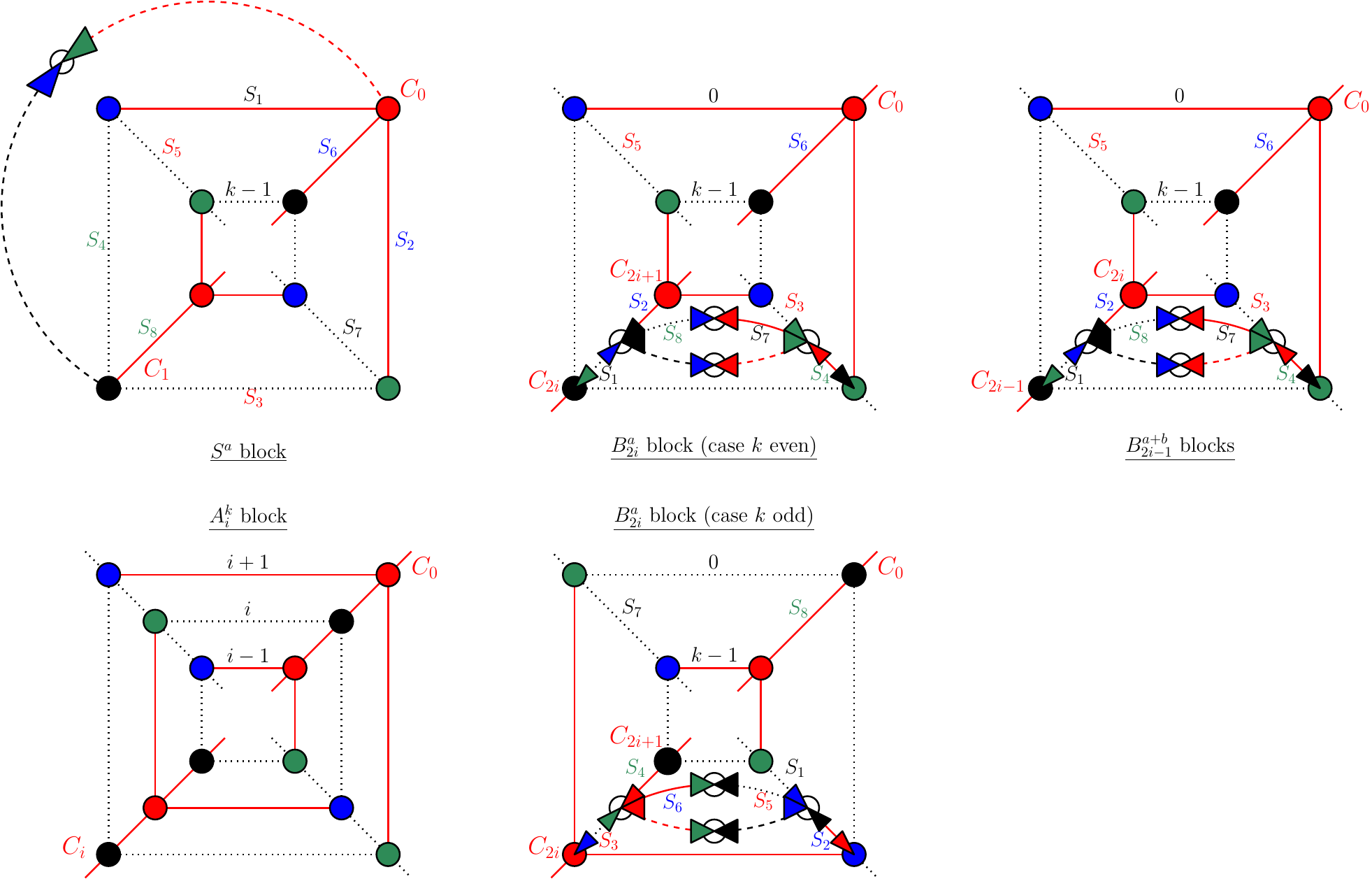}
    \caption{Forest $E_{red}$ in the graph $G_{\ell}^{k,r}$.}\label{fig:red_forest} 
\end{figure}

The properties of \Cref{obs:properties_Ec} still hold and we add some more to deal with the edges incident to the added vertices $(v_j^{(i)})_{\substack{1 \le j \le a+b \\ 0 \le i \le 2\ell - 1}}$.

\begin{observation}\label{obs:properties_Ec_Gklr}
    \begin{enumerate}[(i)]
        \item\label{obs:size_Ec_Glkr} There is no edge of $G_{\ell}^{k,r}$ whose two half-edges have the same color. Thus every edge of $G_{\ell}^{k,r}$ belongs to exactly two of the sets $E_c$. Moreover there are equally many half-edges of each color in $G_{\ell}^{k,r}$. Hence from \Cref{rem:size_Gklr}
        \begin{displaymath}
            \lvert E_{black}\rvert = \lvert E_{red}\rvert = \lvert E_{blue}\rvert = \lvert E_{green}\rvert = \frac{\lvert E(G_{\ell}^{k,r})\rvert}{2} = \ell \cdot (8k +4+r)
        \end{displaymath}
        Furthermore, in each block $A^k$ there are exactly $k$ edges whose two half-edges are $(blue, black)$ and $k$ edges whose two half-edges are $(red, green)$ (one per layer).
        \item\label{obs:connected_components_Ec_Gklr} $E_c$ contains no cycles and has exactly $2\ell + 1$ connected components. There is one component given by the edges located in the North part of the graph which crosses all blocks, denoted $C_0$, and one connected component per $A^k$ block located in the South part of the graph. Denote by $C_i$ the connected component of $E_c$ associated with block $A_i^k$ for $i \in \llbracket 1, 2\ell \rrbracket$. Moreover:
        \begin{itemize}
            \item Any edge of block $A_i^k$ that is not in $E_c$ connects the component $C_i$ to the component $C_0$.
            \item For $1 \le i \le 2\ell - 1$, the special edges for color $c$ located in block $B_i^x$ connect the component $C_{i+1}$ to the component $C_{i}$ or to $C_0$ (in particular this is not an edge between $C_{i}$ and $C_{0}$).
            \item The special edges for color $c$ located in block $S^a$ connect the component $C_1$ to the component $C_0$.
            \item The edges in block $S^a$ incident to a vertex $v_j^{(0)}$ (for $1 \le j \le a$) that are not in $E_c$ connect the component $C_1$ to the component $C_0$.
            \item For $1 \le i \le 2\ell - 1$, the edges in block $B_i^x$ incident to a vertex $v_j^{(i)}$ (for $1 \le j \le a$) that are not in $E_c$ connect the component $C_{i+1}$ to the component $C_{i}$ or to $C_0$ (in particular this is not an edge between $C_{i}$ and $C_{0}$).
            \item If $b = 1$, then for every $1 \le i \le \ell$, the edge in block $B_{2i-1}^{a+b}$ incident to a vertex $v_{a+b}^{(2i-1)}$ that is not in $E_c$ connects the component $C_{2i}$ to the component $C_{0}$ when $c \in \{ \text{red}, \text{blue} \}$ and to the component $C_{2i-1}$ when $c \in \{ \text{black}, \text{green} \}$.
        \end{itemize}
        
        \item\label{obs:identification_edge_Gklr} Every edge in $G_\ell^{k,r}$ can be identified with a unique edge of one of the blocks $S^a, A_1^k, B_1^{a+b}$ or $A_2^k$. \begin{itemize}
            \item For $1 \le i \le \ell$, every colored block $A_{2i-1}^k$ is identical to $A_1^k$ and every colored block $A_{2i}^k$ is identical to $A_2^k$. So every edge of $A_{2i-1}^k$ (resp. $A_{2i}^k$) can be identified with a unique edge of $A_1^k$ (resp. $A_2^k$).
            \item For $1 \le i \le \ell$, every colored block $B_{2i-1}^{a+b}$ is identical to $B_1^{a+b}$. So every edge of $B_{2i-1}^{a+b}$ can be identified with a unique edge of $B_1^{a+b}$.
            \item For $1 \le i \le \ell-1$, every edge in block $B_{2i}^a$ incident to a vertex $v_j^{(2i)}$ (for $1 \le j \le a$) can be identified with a unique edge of $S^a$ incident to the vertex $v_j^{(0)}$ (by matching the color of the half-edges).
            \item For $1 \le i \le \ell-1$, a special edge in $B_{2i}^a$ is identified with the special edge of $S^a$ that has the same label in \Cref{fig:coloring_Gklr}.
        \end{itemize} 
        
        \item\label{obs:edge_counterpart_Gklr} Every edge in $S^a$ or $A_1^k$ has a counterpart in $B_1^{a+b}$ or $A_2^k$.\begin{itemize}
            \item The block $A_2^k$ is either identical to $A_1^k$ (if $k$ is even) or to its central reflection (if $k$ is odd).
            \item We identify every edge of $S^a$ with an edge of $B_1^{a+b}$ the same way as in the last observation. By doing this every edge of $S^a$ has a counterpart. (In case $b = 1$, only the edges incident to $v_{a+b}^{(1)}$ do not have a counterpart.)
        \end{itemize}
    \end{enumerate}
\end{observation}

Using these properties we derive the following procedure to construct large trees contained in $G_{\ell}^{k,r}$.
\begin{enumerate}
    \item Choose a color $c$ among (black, red, blue, green).
    \item Choose an edge $x_1$ among the edges of block $A_1^k$ and $S^a$ that is not in $E_c$ and that is not a special edge for another color $c' \neq c$.
    \item Define $x_2$ as the counterpart in block $B_1^{a+b}$ or $A_2^k$ of the edge $x_1$ (see \Cref{obs:properties_Ec_Gklr}\ref{obs:edge_counterpart_Gklr}).
    \item Build the tree $T_{c,x_1,x_2}$ by adding to the forest $E_c$ the edges $x_1, x_2$ as well as all edges identified with them (see \Cref{obs:properties_Ec_Gklr}\ref{obs:identification_edge_Gklr}).
\end{enumerate}

In case $b = 1$, denote by $e_{a+b}^{BB}$ (resp. $e_{a+b}^{RG}$) the Black-Blue (resp. Red-Green) edge incident to $v_{a+b}^{(1)}$ (in block $B_1^{a+b}$). We add the following trees:
\begin{enumerate}
    \item Choose a color $c$ among (black, red, blue, green).
    \item \begin{itemize}
        \item If $c = \text{black}$, set $x_1 = e_{a+b}^{RG}$ and choose $x_2$ among the edges in $A_2^k$ whose two half-edges are red and green.
        \item If $c = \text{blue}$, set $x_2 = e_{a+b}^{RG}$ and choose $x_1$ among the edges in $A_1^k$ whose two half-edges are red and green.
        \item If $c = \text{green}$, set $x_1 = e_{a+b}^{BB}$ and choose $x_2$ among the edges in $A_2^k$ whose two half-edges are black and blue.
        \item If $c = \text{red}$, set $x_2 = e_{a+b}^{BB}$ and choose $x_1$ among the edges in $A_1^k$ whose two half-edges are black and blue.
    \end{itemize}
    \item Build the tree $T_{c,x_1,x_2}$ by adding to the forest $E_c$ the edges $x_1, x_2$ as well as all edges identified with them (see \Cref{obs:properties_Ec_Gklr}\ref{obs:identification_edge_Gklr}).
\end{enumerate}

Then $T_{c,x_1, x_2}$ is a tree of $G_{\ell}^{k,r}$. Indeed, by adding the edge $x_1$ we connect two connected components $C_0$ and $C_1$ of $E_c$. By adding the edge $x_2$ we connect the connected component $C_2$ to either $C_1$ or $C_0$. Moreover, by adding all edges identified with $x_1$ we connect the connected component $C_{2i-1}$ to either $C_{2i-2}$ or $C_0$ and by adding all edges identified with $x_2$ we connect the connected component $C_{2i}$ to either $C_{2i-1}$ or $C_0$ for each $i = 2, \dots, \ell$ (by \Cref{obs:properties_Ec_Gklr}\ref{obs:connected_components_Ec_Gklr}). Thus we connect all connected components of $E_c$ without creating cycles.

This allows us to construct a $\frac{2g}{g+2}$-arborization of $G_{\ell}^{k,r}$, where $\frac{2g}{g+2} = \frac{16k+8 +4a+2b}{g+2}$:
\begin{itemize}
    \item Set $w(T_{c,x_1, x_2}) = \frac{1 - \frac{b}{2k}}{g+2}$ if $x_1 \in E(A_1^k)$ and $x_2 \in E(A_2^k)$ and their half-edges are (black, blue) or (red, green).
    \item Set $w(T_{c,x_1, x_2}) = \frac{1}{g+2}$ if $x_1 \in E(A_1^k)$ and $x_2 \in E(A_2^k)$ and their half-edges are not (black, blue) or (red, green).
    \item Set $w(T_{c,x_1, x_2}) = \frac{1}{g+2}$ if $x_1 \in E(S^a)$ and $x_1$ is not a special edge for color $c$.
    \item Set $w(T_{c,x_1, x_2}) = \frac{2}{g+2}$ if $x_1 \in E(S^a)$ and $x_1$ is a special edge for color $c$.
    \item Set $w(T_{c,x_1, x_2}) = \frac{b}{k(g+2)}$ if $x_1$ or $x_2$ is in $\{ e_{a+b}^{BB}, e_{a+b}^{RG} \}$.
\end{itemize}
Denote by $\mathcal{A}$ the family of weighted trees thus constructed.

\begin{proposition}\label{prop:arborization_Gklr}
    The family $\mathcal{A}$ defined above is a $\frac{2g}{g+2}$-arborization of $G_{\ell}^{k,r}$ and thus $a_f(G_{\ell}^{k,r}) = \frac{2g}{g + 2}$ (with $g = 8k+4+2a+b$).
\end{proposition}

\begin{proof}
    Fix $c = \text{black}$. In block $A_1^k$ there are $ \frac{8k-4}{2} = 4k - 2$ edges that are not in $E_c$. In the block $S^a$ there are 2 special edges for color $c$ and $a$ edges incident to the vertices $(v_j^{(0)})_{1 \le j \le a}$ that are not in $E_c$. Finally, by \Cref{obs:properties_Ec_Gklr}\ref{obs:size_Ec_Glkr}, there are $k$ edges in block $A_1^k$ whose two half-edges are red and green. Hence the total weight of the trees where the chosen color is black equals:
    \begin{displaymath}
        \begin{split}
            w( \text{black} ) &= \sum_{x_1,x_2} w(T_{c,x_1,x_2}) = k \times \frac{1 - \frac{b}{2k}}{g+2} + (4k - 2 + a - k) \times \frac{1}{g+2} + 2 \times \frac{2}{g+2} + k \times \frac{b}{k(g+2)} \\
            &= \frac{1}{g+2} \times (k- \frac{b}{2} + 3k-2+4+a+b) = \frac{4k + 2 + a + \frac{b}{2}}{g+2}= \frac{g}{2(g+2)}
        \end{split}
    \end{displaymath}
    The same computation applies to any of the four choices of color (up to permuting the colors), hence $w(c) = \frac{g}{2(g+2)}$ for every color $c \in \{\text{black},\text{red},\text{blue},\text{green}\}$. Therefore the sum of the weights of the trees in $\mathcal{A}$ is:
    \begin{displaymath}
        w(\mathcal{A}) = \sum_{c} \sum_{x_1,x_2} w(T_{c,x_1,x_2}) = 4 \times  \frac{g}{2(g+2)} = \frac{2g}{g+2}
    \end{displaymath}
    For $e \in E(G_{\ell}^{k,r})$, we want to show that $e$ is covered by trees of total weight at least 1.
    
    All trees $T_{c,x_1,x_2}$ where $c$ is one of the two colors of the half-edges of $e$ contain $e$. This corresponds to trees of total weight $2 \times \frac{g}{2(g+2)} = \frac{g}{g+2}$. By \Cref{obs:properties_Ec_Gklr}\ref{obs:identification_edge_Gklr}, $e$ can be identified with a unique edge $f_e$ in block $S^a$, $A_1^k$, $B_1^{a+b}$ or $A_2^k$. To reach the remaining weight $\frac{2}{g+2}$, distinguish multiple cases depending on the nature of $f_e$. Denote by $c_1, c_2$ the color of the two half-edges of $f_e$.
    \begin{itemize}
        \item If $f_e$ is in $A_1^k$ or $A_2^k$ then $e$ is used in $T_{c, x_1, x_2}$ when $c \notin \{c_1, c_2\}$ and $f_e \in \{ x_1 , x_2 \}$. If $b= 0$ or if $\{c_1, c_2\} \neq \{ \text{black}, \text{blue} \}$ and $\{c_1, c_2\} \neq \{ \text{red}, \text{green} \}$ then it contributes to two trees of weight $\frac{1}{g+2}$.
        
        If $b=1$ and $\{c_1, c_2\} = \{ \text{black}, \text{blue} \}$ or $\{c_1, c_2\} = \{ \text{red}, \text{green} \}$ then it contributes to two trees of weight $\frac{1 - \frac{b}{2k}}{g+2}$. In this case $e$ is also used in the trees $T_{c,x_1, x_2}$ where $\lbrace x_1, x_2 \rbrace = \lbrace f_e, e_{a+b}^{RG} \}$ or $\lbrace x_1, x_2 \rbrace = \lbrace f_e, e_{a+b}^{BB} \}$. This contributes to one tree of weight $\frac{b}{k(g+2)}$ (for example if $x_1$ is in $A_1^k$ and $(c_1, c_2) = (\text{black},\text{blue})$ then $e$ is used in $T_{\text{red}, f_e, e_{a+b}^{BB}}$). Thus the total contribution is $\frac{2}{g+2}$.
        
        \item Suppose $f_e$ is in $S^a$ or $B_{1}^{a+b}$ and $f_e \not \in \{ e_{a+b}^{BB}, e_{a+b}^{RG} \}$. If $f_e$ is a special edge for some color $c$ then $e$ is used in $T_{c,x_1, x_2}$ where $f_e \in \{x_1, x_2 \}$. This contributes to one tree of weight $\frac{2}{g+2}$.
        
        If $f_e$ is not a special edge then $e$ is used in $T_{c,x_1, x_2}$ when $c \notin \{c_1, c_2\}$ and $f_e \in \{x_1, x_2 \}$. This contributes to two trees of weight $\frac{1}{g+2}$.
        
        \item Suppose $b = 1$ and $f_e \in \{ e_{a+b}^{BB}, e_{a+b}^{RG} \}$, say $f_e = e_{a+b}^{BB}$ here. Then $e$ is used in the trees of type $T_{\text{red},x_1,f_e}$ where $x_1$ is a black-blue edge in $A_1^k$ and the trees of type $T_{\text{green},f_e,x_2}$ where $x_2$ is a black-blue edge in $A_2^k$. From \Cref{obs:properties_Ec_Gklr}\ref{obs:size_Ec_Glkr}, this contributes to $2k$ trees of weight $\frac{b}{k(g+2)}$. Thus the total contribution is $\frac{2}{g+2}$.
    \end{itemize}
    Thus $e$ is contained in trees of $\mathcal{A}$ of total weight at least 1. Hence $\mathcal{A}$ is a $\frac{2g}{g+2}$-arborization of $G_{\ell}^{k,r}$.
    
    \Cref{thm:fractional_arboricity} together with \Cref{rem:size_Gklr} show that $a_f(G_{\ell}^{k,r}) = \frac{2g}{g + 2}$ for $g = 8k+4+2a+b$.
\end{proof}

Finally we can prove \Cref{lem:existence_perfect_g_coating_digirth_g}.

\begin{replemma}{lem:existence_perfect_g_coating_digirth_g}
    For all $g \ge 12$, there exists an infinite family of perfect $g$-coatings of digirth $g$.
\end{replemma}

\begin{proof}
    Let $k\ge 1$, $a \in \{0,1,2,3\}$ and $b \in \{0,1\}$ such that $g=  8k+4+2a+b$. Set $r = 2a+b$. For all $\ell \ge 1$, the graph $G_\ell^{k,r}$ admits a perfect $g$-coating $H_\ell$ (\Cref{prop:perfect_coating_Gklr}). Moreover \Cref{prop:arborization_Gklr} shows that $a_f(G_{\ell}^{k,r}) = \frac{2g}{g + 2}$. Then by \Cref{thm:digirth_perfect_coating}, $H_{\ell}$ has digirth $g$.
\end{proof}

\end{document}